\documentclass[titlepage,12pt]{article} 
\usepackage{hyperref}
\usepackage[usenames,dvipsnames]{pstricks} 
\usepackage{pst-plot} 
\usepackage{amssymb,amsthm,amsmath} 
\usepackage[a4paper]{geometry}
\usepackage{datetime2}
\usepackage[utf8]{inputenc}
\usepackage[italian,english]{babel}
\usepackage{mathrsfs}
\usepackage{bbm}

\date{}

\newcommand{\re}{\mathbb{R}}
\newcommand{\z}{\mathbb{Z}}
\newcommand{\n}{\mathbb{N}}
\newcommand{\ep}{\varepsilon}

\newcommand{\AM}{\operatorname{AM}}
\newcommand{\sPM}{\operatorname{PM}}
\newcommand{\sPMF}{\operatorname{PMF}}
\newcommand{\sRPMF}{\operatorname{RPMF}}
\newcommand{\sRPM}{\operatorname{RPM}}
\newcommand{\PMF}{\operatorname{\mathbb{PMF}}}
\newcommand{\RPM}{\operatorname{\mathbb{RPM}}}
\newcommand{\RPMF}{\operatorname{\mathbb{RPMF}}}
\newcommand{\JF}{\operatorname{\mathbb{JF}}}
\newcommand{\J}{\operatorname{\mathbb{J}}}
\newcommand{\PJ}{P\!J}

\newcommand{\argmin}{\operatorname{argmin}}
\newcommand{\sign}{\operatorname{sign}}
\newcommand{\uep}{u_{\ep}}
\newcommand{\glim}{\Gamma\mbox{--}\lim}
\newcommand{\auto}{\mathrel{\substack{\vspace{-0.1ex}\\\displaystyle\leadsto\\[-0.9em] \displaystyle\leadsto}}}
\newcommand{\loc}{_{\mathrm{loc}}}
\newcommand{\Rloc}{_{\mathrm{R-loc}}}

\newcommand{\obl}{\mathrm{Obl}}
\newcommand{\vrt}{\mathrm{Vert}}
\newcommand{\orz}{\mathrm{Hor}}
\newcommand{\omep}{\omega(\ep)}
\newcommand{\omepn}{\omega(\ep_{n})}
\newcommand{\X}{\mathbb{X}}

\newtheorem{thm}{Theorem}[section]

\newtheorem{rmk}[thm]{Remark}
\newtheorem{prop}[thm]{Proposition}
\newtheorem{defn}[thm]{Definition}
\newtheorem{cor}[thm]{Corollary}

\newtheorem{lemma}[thm]{Lemma}
\newtheorem{open}{Open problem}

\title{A quantitative variational analysis of the staircasing phenomenon for a second order regularization of the Perona-Malik functional}

\author{
Massimo Gobbino\vspace{1ex}\\ 
{\normalsize Università degli Studi di Pisa} \\
{\normalsize Dipartimento di Matematica}\\ 
{\normalsize PISA (Italy)}\\  
{\normalsize e-mail: \texttt{massimo.gobbino@unipi.it}}
\and
Nicola Picenni\vspace{1ex}\\ 
{\normalsize Scuola Normale Superiore} \\
{\normalsize PISA (Italy)}\\
{\normalsize e-mail: \texttt{nicola.picenni@sns.it}}
}

\begin{document}
\maketitle

\begin{abstract}

We consider the Perona-Malik functional in dimension one, namely an integral functional whose Lagrangian is convex-concave with respect to the derivative, with a convexification that is identically zero. We approximate and regularize the functional by adding a term that depends on second order derivatives multiplied by a small coefficient.

We investigate the asymptotic behavior of minima and minimizers as this small parameter vanishes. In particular, we show that minimizers exhibit the so-called staircasing phenomenon, namely they develop a sort of microstructure that looks like a piecewise constant function at a suitable scale. 

Our analysis relies on Gamma-convergence results for a rescaled functional, blow-up techniques, and a characterization of local minimizers for the limit problem. This approach can be extended to more general models.

\vspace{6ex}

\noindent{\bf Mathematics Subject Classification 2020 (MSC2020):} 
49J45, 35B25, 49Q20.

\vspace{6ex}


\noindent{\bf Key words:} 
Perona-Malik functional, singular perturbation, higher order regularization, Gamma-convergence, blow-up, piecewise constant functions, local minimizers, Young measures, varifolds.

\end{abstract}

 
\section{Introduction}

Let us consider the minimum problem for the one-dimensional functional
\begin{equation}
\sPMF(u):=\int_{0}^{1}\log\left(1+u'(x)^{2}\right)\,dx+
\beta\int_{0}^{1}(u(x)-f(x))^{2}\,dx,
\label{defn:PMF}
\end{equation}
where $\beta>0$ is a real number, and $f\in L^{2}((0,1))$ is a given function that we call  \emph{forcing term}. The second integral is a sort of \emph{fidelity term}, tuned by the parameter $\beta$, that penalizes the distance between $u$ and the forcing term $f$. The principal part of (\ref{defn:PMF}) is the functional
\begin{equation}
\sPM(u):=\int_{0}^{1}\log\left(1+u'(x)^{2}\right)\,dx,
\label{defn:PM}
\end{equation}
whose Lagrangian $\phi(p):=\log(1+p^{2})$ is not convex.  To make matters worse, the convex envelope of $\phi(p)$ is identically~0, and this implies that the relaxation of (\ref{defn:PM}) is identically~0 in every reasonable functional space. As a consequence, it is well-known that
\begin{equation}
\inf\left\{\sPMF(u):u\in C^{1}([0,1])\right\}=0
\qquad
\forall f\in L^{2}((0,1)).
\nonumber
\end{equation}

We refer to (\ref{defn:PM}) as the \emph{Perona-Malik functional}, because its formal gradient-flow is (up to a factor~2) the celebrated forward-backward parabolic equation
\begin{equation}
u_{t}=\left(\frac{u_{x}}{1+u_{x}^{2}}\right)_{x}=\frac{1-u_{x}^{2}}{(1+u_{x}^{2})^{2}}\,u_{xx},
\label{defn:PM-eqn}
\end{equation}
introduced by P.~Perona and J.~Malik~\cite{PeronaMalik}. Numerical experiments seem to suggest that this diffusion process has good stability properties, but at the present there is no rigorous theory that explains why a model that is ill-posed from the analytical point of view exhibits this unexpected stability, known in literature as the \emph{Perona-Malik paradox}~\cite{Kichenassami}. 

The same qualitative analysis applies when the principal part is of the form
\begin{equation}
\sPM(\phi,u):=\int_{0}^{1}\phi(u'(x))\,dx,
\nonumber
\end{equation}
provided that $\phi(p)$ is convex in a neighborhood of the origin, concave when $|p|$ is large, and with convex envelope identically equal to~0. Some notable examples are 
\begin{equation}
\phi(p)=\arctan(p^{2})
\qquad\mbox{or}\qquad
\phi(p)=(1+p^{2})^{\alpha}
\quad\text{with }\alpha\in(0,1/2),
\label{defn:phi-1}
\end{equation}
or even more generally
\begin{equation}
\phi(p)=(1+|p|^{\gamma})^{\alpha}
\quad\text{with }\gamma>1\text{ and }\alpha\in(0,1/\gamma).
\label{defn:phi-2}
\end{equation}

\paragraph{\textmd{\textit{Singular perturbation of the Perona-Malik functional}}}

Several approximating models have been introduced in order to mitigate the ill-posed nature of (\ref{defn:PM-eqn}). These approximating models are obtained via convolution~\cite{1992-SIAM-Lions}, space discretization~\cite{2008-JDE-BNPT,2001-CPAM-Esedoglu,GG:grad-est}, time delay~\cite{2007-Amann}, fractional derivatives~\cite{2009-JDE-Guidotti}, fourth order regularization~\cite{1996-Duke-DeGiorgi,2008-TAMS-BF,2019-SIAM-BerGiaTes}, addition of a dissipative term (see~\cite{2020-JDE-BerSmaTes} and the references quoted therein). For a more complete list of references on the evolution problem (\ref{defn:PM-eqn}) we refer to the recent papers~\cite{2018-Poincare-KimYan,2019-SIAM-BerGiaTes,2020-JDE-BerSmaTes} and to the references quoted therein. In this paper we limit ourselves to the variational background, and we consider the functional whose formal gradient flow is the fourth order regularization of (\ref{defn:PM-eqn}), namely the functional (see~\cite{1996-Duke-DeGiorgi,ABG,2006-DCDS-BelFusGug,2008-TAMS-BF,2014-M3AS-BelChaGol})
\begin{equation}
\sPMF_{\ep}(u):=
\int_{0}^{1}\left\{\ep^{10}|\log\ep|^{2}u''(x)^{2}+
\log\left(1+u'(x)^{2}\right)+
\beta(u(x)-f(x))^{2}\right\}dx,
\label{defn:SPM-intro}
\end{equation}
where the bizarre form of the $\ep$-dependent coefficient is just aimed at preventing the appearance of decay rates defined in an implicit way in the sequel of the paper. For every choice of $\ep\in(0,1)$ and $\beta>0$ the model is well-posed, in the sense that the minimum problem for  (\ref{defn:SPM-intro}) admits at least one minimizer of class $C^{2}$ for every choice of the forcing term $f\in L^{2}((0,1))$.  Here we investigate the asymptotic behavior of minima and minimizers as $\ep\to 0^{+}$.  Before describing our results, it is useful to open a parenthesis on a related problem that has already been studied in the literature.

\paragraph{\textmd{\textit{The Alberti-M\"uller model}}}

Let us consider the functional 
\begin{equation}
\AM_{\ep}(u):=\int_{0}^{1}\left\{\ep^{2}u''(x)^{2}+(u'(x)^{2}-1)^{2}+\beta(x) u(x)^{2}\right\}dx,
\label{defn:AM}
\end{equation}
where $\beta\in L^{\infty}((0,1))$ is positive for almost every $x\in(0,1)$. The minimizers of (\ref{defn:AM}) with periodic boundary conditions were studied by G.~Alberti and S.~M\"uller in~\cite{2001-CPAM-AlbertiMuller} (see also~\cite{1993-CalcVar-Muller}). In this model the forcing term $f(x)$ is identically~0, and the dependence on first order derivatives is described by the double-well potential $\phi(p):=(p^{2}-1)^{2}$. As in (\ref{defn:SPM-intro}) the function $\phi(p)$ is non-convex, but in this case its convex envelope vanishes just for $|p|\leq 1$, while it coincides with $\phi(p)$ elsewhere, and in particular it is coercive at infinity. 

From the heuristic point of view, minimizers to (\ref{defn:AM}) would like to be identically~0, but with constant derivative equal to $\pm 1$. Of course this is not possible if we think of $u(x)$ and $u'(x)$ as functions, but it becomes possible if we consider $u(x)$ as a function whose ``derivative'' $u'(x)$ is a Young measure. More formally, given a family $\{u_{\ep}(x)\}$ of minimizers to (\ref{defn:AM}), one can show that $u_{\ep}(x)\to 0$ uniformly, $u_{\ep}'(x)\rightharpoonup 0$ weakly in $L^{4}((0,1))$, and more precisely $u_{\ep}'(x)$ converges to the Young measure that in every point $x\in(0,1)$ assumes the two values $\pm 1$ with probability $1/2$. 

The next step consists in analyzing the asymptotic profile of minimizers. The intuitive idea is that minimizers develop a \emph{microstructure} at some scale $\omep$, and this microstructure resembles a triangular wave (sawtooth function). In other words, one expects minimizers to be of the form 
\begin{equation}
u_{\ep}(x)\sim\omep\varphi\left(\frac{x}{\omep}+b(\ep)\right),
\label{th:AM-expansion}
\end{equation}
where
\begin{itemize}

\item  the function $\varphi$ that describes the asymptotic profile of minimizers is a triangular wave with slopes $\pm 1$, for example the function defined by $\varphi(x):=|x|-1$ for every $x\in[-2,2]$, and then extended by periodicity to the whole real line,

\item  $\omep$ is a suitable scaling factor that vanishes as $\ep\to 0^{+}$ and is proportional to the asymptotic ``period'' of minimizers (which however are not necessarily periodic functions ``in large''),

\item  $b(\ep)$ is a sort of phase parameter, that can be assumed to be less than the period of $\varphi$.

\end{itemize}

We point out that the limit of $u_{\ep}'(x)$ as a Young measure carries no information concerning the asymptotic behavior of $\omep$, and actually it does not even imply the existence of any form of asymptotic period or asymptotic profile. 

The first big issue is giving a rigorous formal meaning to an asymptotic expansion of the form (\ref{th:AM-expansion}). In~\cite{2001-CPAM-AlbertiMuller} the formalization relies on the notion of Young measure with values in compact metric spaces. In a nutshell, starting form every minimizer $u_{\ep}(x)$, the authors consider the function that associates to every $x\in(0,1)$ the rescaled function 
\begin{equation}
y\mapsto\frac{u_{\ep}(x+\omep y)}{\omep},
\nonumber
\end{equation}
where $\omep=\ep^{1/3}$. This new function is interpreted as a Young measure on the interval $(0,1)$ with values in $L^{\infty}(\re)$, which is a \emph{compact} metric space with respect to the distance according to which $g_{n}$ converges to $g_{\infty}$ if and only if $\arctan(g_{n})$ converges to $\arctan(g_{\infty})$ with respect to the weak* convergence in $L^{\infty}(\re)$. The result is that this family of Young measures converges (in the sense of Young measures with values in a compact metric space) to a limit Young measure that in almost every point is concentrated in the translations of the triangular wave. This statement is a rigorous, although rather abstract and technical, formulation of expansion (\ref{th:AM-expansion}). 

\paragraph{\textmd{\textit{From Young measures to varifolds}}}

There are some notable differences between our model and (\ref{defn:AM}). The first one is that in our case the trivial forcing term $f(x)\equiv 0$ would lead to the trivial solution $u_{\ep}(x)\equiv 0$ for every $\ep\in(0,1)$. Therefore, here a nontrivial forcing term is required if we want nontrivial solutions.

The second difference lies in the growth of the convex envelope of $\phi(p)$. In the case of (\ref{defn:AM}) the convex envelope grows at infinity as $p^{4}$, and this guarantees a uniform bound in $L^{4}((0,1))$ for the derivatives of all sequences with bounded energy. In our case the convex envelope vanishes identically, and therefore there is no hope to obtain bounds on derivatives in terms of bounds on the energies.

The third, and more relevant, difference lies in the construction of the convex envelope. In the case of (\ref{defn:AM}) the convex envelope of $\phi$ vanishes in the interval $[-1,1]$ because every $p$ in this interval can be written as a convex combination of $\pm 1$, and $\phi(1)=\phi(-1)=0$. This is the ultimate reason why the derivatives of minimizers tend to stay close to the two values $\pm 1$ when $\ep$ is small enough.

In our case the convex envelope of $\phi$ vanishes identically on the whole real line, but no real number $p$ can be written as the convex combination of two distinct points where $\phi$ vanishes. Roughly speaking, the vanishing of the convex envelope is achieved only in the limit, in some sense by writing every real number $p$ as a convex combination of 0 and $\pm\infty$, depending on the sign of $p$. This implies that minimizers $\uep(x)$ tend to assume a staircase-like shape, with regions where they are ``almost horizontal'' and regions where they are ``almost vertical'' (as described in the left and central section of Figure~\ref{figure:multi-scale}). From the technical point of view, this means  that there is no hope that the family $\{u_{\ep}'(x)\}$ admits a limit in the sense of Young measures. 

This is the point in which varifolds come into play, because varifolds allow ``functions'' whose graph has in every point a mix of horizontal and vertical ``tangent'' lines.

\paragraph{\textmd{\textit{Our results}}}

In our analysis of the asymptotic behavior of minima and minimizers, we restrict ourselves to forcing terms $f(x)$ of class $C^{1}$, and we prove three main results. 
\begin{itemize}

\item  The first result (Theorem~\ref{thm:asympt-min}) concerns the asymptotic behavior of minima. We prove that the minimum $m_{\ep}$ of (\ref{defn:SPM-intro}) over $H^{2}((0,1))$ satisfies $m_{\ep}\sim c_{0}\ep^{2}|\log\ep|$, where $c_{0}$ is proportional to the integral of $|f'(x)|^{4/5}$.

\item  The second result (Theorem~\ref{thm:BU}) concerns the asymptotic behavior of minimizers $u_{\ep}(x)$. To this end, for every family $x_{\ep}\to x_{0}\in(0,1)$ we consider the families of functions
\begin{equation}
y\mapsto\frac{u_{\ep}(x_{\ep}+\omep y)-f(x_{\ep})}{\omep}
\qquad\text{and}\qquad
y\mapsto\frac{u_{\ep}(x_{\ep}+\omep y)-u_{\ep}(x_{\ep})}{\omep},
\label{defn:BU-uep}
\end{equation}
which correspond to the intuitive idea of zooming the graph of a minimizer $u_{\ep}(x)$ in a neighborhood of $(x_{\ep},f(x_{\ep}))$ and $(x_{\ep},u_{\ep}(x_{\ep}))$ at scale $\omep$. We show that, when $\omep=\ep|\log\ep|^{1/2}$, these functions converge (up to subsequences) in a rather strong sense (strict convergence of bounded variation functions, see Definition~\ref{defn:BV-sc}) to a piecewise constant function, a sort of staircase with steps whose height and length depend on $f'(x_{0})$. This result provides a quantitative description of the staircase-like microstructure of minimizers, with a notion of convergence that is much stronger than weak* convergence in $L^{\infty}(\re)$, and without the technical machinery of Young measures with values in metric spaces (see Remark~\ref{rmk:AM}). 

\item  The third result (Theorem~\ref{thm:varifold}) shows that $u_{\ep}(x)\to f(x)$ first in the sense of uniform convergence, then in the sense of strict convergence of bounded variation functions, and finally in the sense of varifolds, provided that we consider the graph of $f(x)$ as a varifold with a suitable density and a suitable combination of horizontal and vertical tangent lines in every point. 

\end{itemize}

The three results described above are only the first order analysis of what is actually a \emph{multi-scale problem}. In a companion paper (in preparation) we plan to investigate higher-resolution zooms of minimizers (from the center to the right of Figure~\ref{figure:multi-scale}), in order to reveal the exact structure of the horizontal and vertical parts of each step of the staircase. 

\begin{figure}[ht]
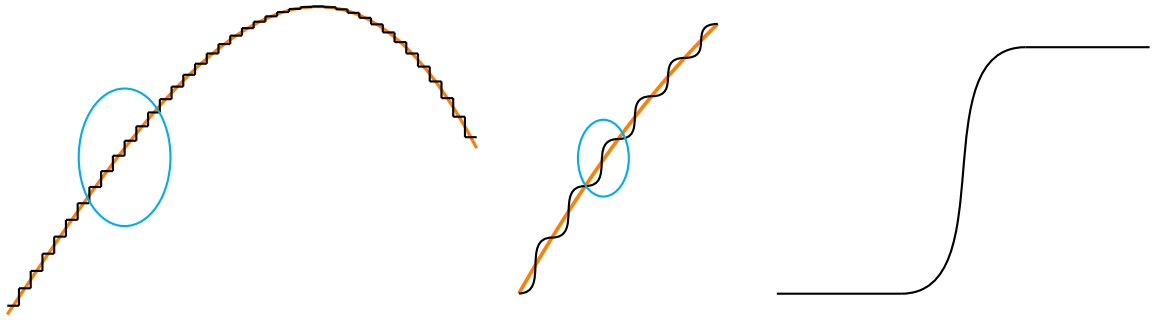

\begin{center}

\hfill
\psset{unit=8.5ex}
\pspicture(0,0)(4,3)

\psplot[linecolor=orange,linewidth=1.5\pslinewidth]{0}{4}{x 1.5 mul x 3 exp 14 div sub}

\multido{\n=0.05+0.10}{40}{
\psline(!\n \space 0.05 sub \n \space 1.5 mul \n \space 3 exp 14 div sub)
(!\n \space 0.05 add \n \space 1.5 mul \n \space 3 exp 14 div sub)
}

\multido{\n=0.05+0.10}{39}{
\psline
(!\n \space 0.05 add \n \space 1.5 mul \n \space 3 exp 14 div sub)
(!\n \space 0.05 add \n \space 0.1 add 1.5 mul \n \space 0.1 add 3 exp 14 div sub)
}

\psellipse[linecolor=cyan](1,1.35)(0.4,0.6)

\endpspicture
\hfill
\psset{unit=12ex}
\pspicture(-0.2,-0.3)(1.3,1.7)

\psplot[linecolor=orange,linewidth=1.7\pslinewidth]{-0.1}{1.1}{x 1.7 mul x 2 exp 3 div sub}

\multido{\n=-0.1+0.2}{6}{\psbezier
(!\n \space \n \space 1.7 mul \n \space 2 exp 3 div sub)
(!\n \space 0.19 add \n \space 1.7 mul \n \space 2 exp 3 div sub)
(!\n \space 0.01 add \n \space 0.2 add 1.7 mul \n \space 0.2 add 2 exp 3 div sub)
(!\n \space 0.2 add \n \space 0.2 add 1.7 mul \n \space 0.2 add 2 exp 3 div sub)
}

\psellipse[linecolor=cyan](0.41,0.65)(0.16,0.24)

\endpspicture
\hfill
\psset{unit=3ex}
\pspicture(-3,-0.5)(6,6)

\psline(-3,0)(0,0)
\psbezier(0,0)(2.5,0)(0.5,6)(3,6)
\psline(3,6)(6,6)

\endpspicture
\hfill
\mbox{}

\caption{description of the multi-scale problem at three levels of resolution. Left: staircasing effect in large around the forcing term. Center: zoom of the staircase in a region. Right: cubic transition between two consecutive steps.}
\label{figure:multi-scale}
\end{center}
\end{figure}

\paragraph{\textmd{\textit{Overview of the technique}}}

Our analysis relies on Gamma-convergence techniques. The easy remark is that minimum values of (\ref{defn:SPM-intro}) tend to~0, and minimizers tend to the forcing term in $L^{2}((0,1))$. This is because the unstable character of (\ref{defn:PM}) comes back again when $\ep\to 0^{+}$, and forces the Gamma-limit of the family of functionals (\ref{defn:SPM-intro}) to be identically~0. 

More delicate is finding the vanishing order of minimum values, and the fine structure of minimizers as $\ep\to 0^{+}$. The starting observation is that, if $v_{\ep}(y)$ denotes the blow-up defined in (\ref{defn:BU-uep}) on the left, with $\omep=\ep|\log\ep|^{1/2}$, then $v_{\ep}(y)$ minimizes a rescaled version of (\ref{defn:SPM-intro}), namely the functional
\begin{equation}
\sRPMF_{\ep}(v):=
\int_{I_{\ep}}\left\{\ep^{6}(v'')^{2}+\frac{1}{\ep^{2}|\log\ep|}\log\left(1+(v')^{2}\right)
+\beta\left(v-g_{\ep}\right)^{2}\right\}\,dy,
\label{defn:RPMF-intro}
\end{equation}
where the new forcing term $g_{\ep}(y)$ is a suitable blow-up of $f(x)$, and the new integration interval $I_{\ep}$ depends on the blow-up center $x_{\ep}$, but in any case its length is equal to $\omep^{-1}$, and therefore it diverges.

If $f(x)$ is of class $C^{1}$, then $g_{\ep}(y)\to f'(x_{0})y$ when $x_{\ep}\to x_{0}$. Moreover, the results of~\cite{ABG,2008-TAMS-BF} suggest that, if we consider the functional (\ref{defn:RPMF-intro}) restricted to a \emph{finite fixed interval} $(a,b)$, its Gamma-limit has the form
\begin{equation}
\alpha_{0}J_{1/2}(v)+\beta\int_{a}^{b}\left(v(y)-f'(x_{0})y\right)^{2}\,dy,
\label{defn:J-intro}
\end{equation}
where $\alpha_{0}$ is a suitable positive constant, and the functional $J_{1/2}(v)$ is finite only if $v$ is a ``pure jump function'' (see Definition~\ref{defn:PJF}), and in this class it coincides with the sum of the square roots of the jump heights of $v$.

At the end of the day, this means that the minimum problem for (\ref{defn:SPM-intro}) can be approximated, at a suitable small scale, by a family of minimum problems for functionals such as (\ref{defn:J-intro}), and these minimum problems, due to the simpler form and to the linear forcing term, can be solved almost explicitly.

However, things are not so simple. A first issue is that the integration intervals $I_{\ep}$ in (\ref{defn:RPMF-intro}) invade the whole real line. This forces us to work with local minimizers (namely minimizers up to perturbations with compact support) instead of global minimizers. So we have to adapt the classical Gamma-convergence results in order to deal with local minimizers, and we need also to classify all local minimizers to (\ref{defn:J-intro}). These local minimizers are characterized in Proposition~\ref{prop:loc-min-class}, and they turn out to be staircases whose steps have length and height that depend on $f'(x_{0})$.

The second issue is compactness. We observed before that a bound on $\sPM_{\ep}(\uep)$ does not provide compactness of the family $\{\uep\}$ in any reasonable space. After rescaling and introducing (\ref{defn:RPMF-intro}), on the one hand the good news is that a classical coerciveness result implies that a uniform bound on $\sRPMF_{\ep}(v_{\ep})$ is enough to deduce that the family $\{v_{\ep}\}$ is relatively compact, for example in $L^{2}$. On the other hand, the bad news is that an asymptotic estimate of the form $\sPM_{\ep}(\uep)\sim c_{0}\omep^{2}$ yields only a uniform bound on $\omep\sRPMF_{\ep}(v_{\ep})$, which does not exclude that $\sRPMF_{\ep}(v_{\ep})$ might diverge as $\ep\to 0^{+}$.

We overcome this difficulty by showing that a bound of this type in some interval yields a true uniform bound for $\sRPMF_{\ep}(v_{\ep})$ in a \emph{smaller interval}, and this is enough to guarantee the compactness of local minimizers. This improvement of the bound (see Proposition~\ref{prop:iteration}) requires a delicate iteration argument in a sequence of nested intervals, which probably represents the technical core of this paper.

\paragraph{\textmd{\textit{Possible extensions}}}

In order to contain this paper in a reasonable length, we decided to focus our presentation only on the singular perturbation (\ref{defn:SPM-intro}) of the original functional with the logarithm. Nevertheless, many parts of the theory can be extended to more general models. We discuss some possible generalizations in section~\ref{sec:extension}.

\paragraph{\textmd{\textit{Dynamic consequences}}}

We hope that our variational analysis could be useful in the investigation of solutions to the evolution equation (\ref{defn:PM-eqn}). Numerical experiments with different approximating schemes seem to suggest that solutions develop \emph{instantaneously} a staircase-like pattern consistent with the results of this paper. Due to its instantaneous character, this phase of the dynamic is usually referred to as ``fast time'' (see~\cite{2006-DCDS-BelFusGug}). 

The connection between the dynamic and the variational behavior is hardly surprising if we think of gradient-flows as limits of discrete-time evolutions, as in De Giorgi's theory of minimizing movements. In this context the minimum problem for (\ref{defn:PMF}) with forcing term $f(x)$ equal to the initial datum $u_{0}(x)$ is just the first step in the construction of the minimizing movement. Transforming this intuition into a rigorous statement concerning the fast-time behavior of solutions to (\ref{defn:PM-eqn}) is a challenging problem.

Another issue is that actually the staircasing effect seems to appear only in the so-called supercritical regions of $u_{0}(x)$, namely where $u_{0}'(x)$ falls in the concavity region of $\phi(p)$ (see the simulations in~\cite{2006-SIAM-Esedoglu,2006-DCDS-BelFusGug,2009-JMImVis-GuiLam,2009-JDE-Guidotti,2012-JDE-Guidotti}). The variational analysis can not produce this effect, in some sense because the convexification involves a ``global procedure'', and therefore it is very likely that an explanation should rely also on dynamical arguments. 

\paragraph{\textmd{\textit{Structure of the paper}}}

This paper is organized as follows. In section~\ref{sec:statements} we introduce the notations and we state our main results concerning the asymptotic behavior of minima and minimizers for (\ref{defn:SPM-intro}). In section~\ref{sec:gconv} we state the results that we need concerning the rescaled functionals (\ref{defn:RPMF-intro}) and their Gamma-limit.  In section~\ref{sec:loc-min} we recall the notion of local minimizers, both for (\ref{defn:SPM-intro}) and for the Gamma-limit, and we state their main properties.  In section~\ref{sec:strategy} we show that our main results follow from the properties of local minimizers, that we prove later in section~\ref{sec:loc-min-proof}.  Finally, in section~\ref{sec:extension} we mention some different models to which our theory can be extended, and in section~\ref{sec:open} we present some open problems. We also add an appendix with a proof of the results stated in section~\ref{sec:gconv}, some of which are apparently missing, or present with flawed proofs, in the literature. 


\setcounter{equation}{0}
\section{Statements}\label{sec:statements}

For every $\ep\in(0,1)$ let us set
\begin{equation}
\omep:=\ep|\log\ep|^{1/2}.
\label{defn:omep}
\end{equation}

Let $\beta>0$ be a real number, let $\Omega\subseteq\re$ be an open set, and let $f\in L^{2}(\Omega)$ be a function that we call forcing term. In order to emphasize the dependence on all the parameters, we write (\ref{defn:SPM-intro}) in the form
\begin{equation}
\PMF_{\ep}(\beta,f,\Omega,u):=
\int_{\Omega}\left\{\ep^{6}\omep^{4}u''(x)^{2}+
\log\left(1+u'(x)^{2}\right)+
\beta(u(x)-f(x))^{2}\right\}dx.
\label{defn:SPM}
\end{equation}

The first result that we state concerns existence and regularity of minimizers, and their convergence to the fidelity term in $L^{2}((0,1))$. We omit the proof because it is a standard application of the direct method in the calculus of variations, and of the fact that the convex envelope of the function $p\mapsto\log(1+p^{2})$ is identically~0.

\begin{prop}[Existence and regularity of minimizers]\label{prop:basic}

Let $\omep$ be defined by (\ref{defn:omep}), and let $\PMF_{\ep}(\beta,f,(0,1),u)$ be defined by (\ref{defn:SPM}), where $\ep\in(0,1)$ and $\beta>0$ are two real numbers, and $f\in L^{2}((0,1))$ is a given function. 

Then the following facts hold true.

\begin{enumerate}
\renewcommand{\labelenumi}{(\arabic{enumi})}

\item  \emph{(Existence)} There exists
\begin{equation}
m(\ep,\beta,f):=\min\left\{\PMF_{\ep}(\beta,f,(0,1),u):u\in H^{2}((0,1))\strut\right\}.
\label{defn:DMnf}
\end{equation} 

\item  \emph{(Regularity)} Every minimizer belongs to $H^{4}((0,1))$, and in particular to $C^{2}([0,1])$.

\item  \emph{(Minimum values vanish in the limit)} It turns out that $m(\ep,\beta,f)\to 0$ as $\ep\to 0^{+}$.

\item  \emph{(Convergence of minimizers to the fidelity term)} If $\{u_{\ep}\}$ is any family of minimizers for (\ref{defn:DMnf}), then $u_{\ep}(x)\to f(x)$ in $L^{2}((0,1))$ as $\ep\to 0^{+}$.

\end{enumerate}

\end{prop}

In the sequel we assume that the forcing term $f(x)$ belongs to $C^{1}([0,1])$. Under this regularity assumption, our first result concerns the asymptotic behavior of minima.

\begin{thm}[Asymptotic behavior of minima]\label{thm:asympt-min}

Let $\omep$ be defined by (\ref{defn:omep}), and let $\PMF_{\ep}(\beta,f,(0,1),u)$ be defined by (\ref{defn:SPM}), where $\ep\in(0,1)$ and $\beta>0$ are two real numbers, and $f\in C^{1}([0,1])$ is a given function. 

Then the minimum value $m(\ep,\beta,f)$ defined in (\ref{defn:DMnf}) satisfies
\begin{equation}
\lim_{\ep\to 0^{+}}\frac{m(\ep,\beta,f)}{\omep^{2}}=
10\left(\frac{2\beta}{27}\right)^{1/5}\int_{0}^{1}|f'(x)|^{4/5}\,dx.
\label{th:asympt-min}
\end{equation}

\end{thm}

The asymptotic behavior of $m(\ep,\beta,f)$ under weaker regularity assumptions on $f(x)$ is a largely open problem. We refer to section~\ref{sec:open} for further details.


Now we investigate the asymptotic behavior of minimizers. The intuitive idea is that they tend to develop a staircase structure. In order to formalize this idea, we need several definitions. To begin with, we define some classes of ``staircase-like functions''.

\begin{defn}[Canonical staircases]\label{defn:staircase}
\begin{em}

Let $S:\re\to\re$ be the function defined by
\begin{equation*}
S(x):=2\left\lfloor\frac{x+1}{2}\right\rfloor
\qquad
\forall x\in\re,
\end{equation*}
where, for every real number $\alpha$, the symbol $\lfloor\alpha\rfloor$ denotes the greatest integer less than or equal to $\alpha$. For every pair $(H,V)$ of real numbers, with $H>0$, we call \emph{canonical $(H,V)$-staircase} the function $S_{H,V}:\re\to\re$ defined by
\begin{equation}
S_{H,V}(x):=V\cdot S(x/H)
\qquad
\forall x\in\re.
\label{defn:SC0}
\end{equation}

\end{em}
\end{defn}

Roughly speaking, the graph of $S_{H,V}(x)$ is a staircase with steps of horizontal length $2H$ and vertical height $2V$. The origin is the midpoint of the horizontal part of one of the steps. The staircase degenerates to the null function when $V=0$, independently of the value of~$H$.

\begin{defn}[Translations of the canonical staircase]\label{defn:translations}
\begin{em}

Let $(H,V)$ be a pair of real numbers, with $H>0$, and let $S_{H,V}(x)$ be the function defined in (\ref{defn:SC0}). Let $v:\re\to\re$ be a function.
\begin{itemize}

\item We say that $v$ is an \emph{oblique translation} of $S_{H,V}(x)$, and we write $v\in\obl(H,V)$, if there exists a real number $\tau_{0}\in[-1,1]$ such that
\begin{equation}
v(x)=S_{H,V}(x-H\tau_{0})+V\tau_{0}
\qquad
\forall x\in\re.
\nonumber
\end{equation}

\item We say that $v$ is a \emph{graph translation of horizontal type} of $S_{H,V}(x)$, and we write $v\in\orz(H,V)$, if there exists a real number $\tau_{0}\in[-1,1]$ such that
\begin{equation}
v(x)=S_{H,V}(x-H\tau_{0})
\qquad
\forall x\in\re.
\nonumber
\end{equation}

\item We say that $v$ is a \emph{graph translation of vertical type} of $S_{H,V}(x)$, and we write $v\in\vrt(H,V)$, if there exists a real number $\tau_{0}\in[-1,1]$ such that
\begin{equation}
v(x)=S_{H,V}(x-H)+V(1-\tau_{0})
\qquad
\forall x\in\re.
\nonumber
\end{equation}

\end{itemize}

\end{em}
\end{defn}

\begin{rmk}
\begin{em}

Let us interpret translations of the canonical staircase in terms of graph (see Figure~\ref{figure:staircase}).
\begin{itemize}

\item  Oblique translations correspond to taking the graph of the canonical staircase $S_{H,V}(x)$ and moving the origin along the line $Hy=Vx$, namely the line that connects the midpoints of the steps.

\item  Graph translations of horizontal type correspond to moving the origin to some point in the horizontal part of some step.

\item  Graph translations of vertical type correspond to moving the origin to some point in the vertical part of some step.

\end{itemize}

We observe that graph translations of horizontal type with $\tau_{0}=\pm 1$ coincide with graph translations of vertical type with the same value of $\tau_{0}$. In those cases the origin is moved to the ``corners'' of the graph.

\end{em}
\end{rmk}

\begin{figure}[h]
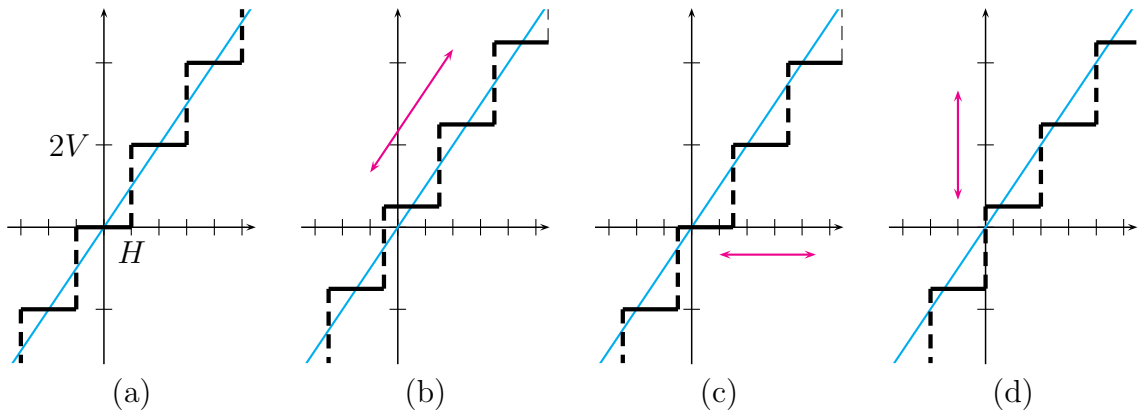


\hfill
\psset{unit=2ex}
\pspicture(-4,-6.5)(6,8.5)

\psline[linewidth=0.7\pslinewidth]{->}(-3.5,0)(5.5,0)
\psline[linewidth=0.7\pslinewidth]{->}(0,-5)(0,8)

\multiput(0,-3)(0,3){4}{\psline[linewidth=0.5\pslinewidth](-0.3,0)(0.3,0)}
\multiput(-3,0)(1,0){9}{\psline[linewidth=0.5\pslinewidth](0,-0.3)(0,0.3)}

\psclip{\psframe[linestyle=none](-3.5,-5)(5.5,8)}
\psline[linecolor=cyan](-4,-6)(6,9)
\multiput(-3,-3)(2,3){4}{\psline[linewidth=2\pslinewidth](0,0)(2,0)}
\multiput(-3,-3)(2,3){5}{\psline[linewidth=2\pslinewidth,linestyle=dashed](0,0)(0,-3)}
\endpsclip

\rput[r](-0.5,3){$2V$}
\rput[t](1,-0.5){$H$}
\rput[t](1,-5.5){(a)}

\endpspicture
\hfill
\pspicture(-4,-6.5)(6,8.5)

\psline[linewidth=0.7\pslinewidth]{->}(-3.5,0)(5.5,0)
\psline[linewidth=0.7\pslinewidth]{->}(0,-5)(0,8)

\multiput(0,-3)(0,3){4}{\psline[linewidth=0.5\pslinewidth](-0.3,0)(0.3,0)}
\multiput(-3,0)(1,0){9}{\psline[linewidth=0.5\pslinewidth](0,-0.3)(0,0.3)}

\psline[linecolor=magenta]{<->}(-1,2)(2,6.5)

\psclip{\psframe[linestyle=none](-3.55,-5)(5.5,8)}
\psline[linecolor=cyan](-4,-6)(6,9)
\multiput(-2.5,-2.25)(2,3){4}{\psline[linewidth=2\pslinewidth](0,0)(2,0)}
\multiput(-2.5,-2.25)(2,3){5}{\psline[linewidth=2\pslinewidth,linestyle=dashed](0,0)(0,-3)}
\endpsclip

\rput[t](1,-5.5){(b)}

\endpspicture
\hfill
\pspicture(-4,-6.5)(6,8.5)

\psline[linewidth=0.7\pslinewidth]{->}(-3.5,0)(5.5,0)
\psline[linewidth=0.7\pslinewidth]{->}(0,-5)(0,8)

\multiput(0,-3)(0,3){4}{\psline[linewidth=0.5\pslinewidth](-0.3,0)(0.3,0)}
\multiput(-3,0)(1,0){9}{\psline[linewidth=0.5\pslinewidth](0,-0.3)(0,0.3)}

\psline[linecolor=magenta]{<->}(1,-1)(4.5,-1)

\psclip{\psframe[linestyle=none](-3.5,-5)(5.5,8)}
\psline[linecolor=cyan](-4,-6)(6,9)
\multiput(-2.5,-3)(2,3){4}{\psline[linewidth=2\pslinewidth](0,0)(2,0)}
\multiput(-2.5,-3)(2,3){5}{\psline[linewidth=2\pslinewidth,linestyle=dashed](0,0)(0,-3)}
\endpsclip

\rput[t](1,-5.5){(c)}

\endpspicture
\hfill
\pspicture(-4,-6.5)(6,8.5)

\psline[linewidth=0.7\pslinewidth]{->}(-3.5,0)(5.5,0)
\psline[linewidth=0.7\pslinewidth]{->}(0,-5)(0,8)

\multiput(0,-3)(0,3){4}{\psline[linewidth=0.5\pslinewidth](-0.3,0)(0.3,0)}
\multiput(-3,0)(1,0){9}{\psline[linewidth=0.5\pslinewidth](0,-0.3)(0,0.3)}

\psline[linecolor=magenta]{<->}(-1,1)(-1,5)

\psclip{\psframe[linestyle=none](-3.5,-5)(5.5,8)}
\psline[linecolor=cyan](-4,-6)(6,9)
\multiput(-2,-2.25)(2,3){4}{\psline[linewidth=2\pslinewidth](0,0)(2,0)}
\multiput(-2,-2.25)(2,3){6}{\psline[linewidth=2\pslinewidth,linestyle=dashed](0,0)(0,-3)}
\endpsclip

\rput[t](1,-5.5){(d)}

\endpspicture
\hfill\mbox{}

\caption{(a)~Canonical staircase. (b)~Oblique translation. (c)~Graph translation of horizontal type. (d)~Graph translation of vertical type. In all translations the parameter is $\tau_{0}=1/2$.}
\label{figure:staircase}
\end{figure}


In the sequel $BV((a,b))$ denotes the space of functions with bounded variation in the interval $(a,b)\subseteq\re$. For every function $u$ in this space, $Du$ denotes its distributional derivative, which is  a signed measure, and $|Du|((a,b))$ denotes the total variation of $u$ in $(a,b)$. We call jump points of $u$ the points $x\in(a,b)$ where $u$ is not continuous. As usual, $BV\loc(\re)$ denotes the set of all functions $u:\re\to\re$ whose restriction to every interval $(a,b)$ belongs to $BV((a,b))$. The staircase-like functions we have introduced above are typical examples of elements of the space $BV\loc(\re)$. 

Our result for the asymptotic behavior of minimizers involves smooth functions converging to staircases. The strongest sense in which this convergence is possible is the so-called strict converge. We recall here the definitions (see~\cite[Definition~3.14]{AFP}).

\begin{defn}[Strict convergence in an interval]\label{defn:BV-sc}
\begin{em}

Let $(a,b)\subseteq\re$ be an interval. A sequence of functions $\{u_{n}\}\subseteq BV((a,b))$ converges \emph{strictly} to some $u_{\infty}\in BV((a,b))$, and we write
\begin{equation}
u_{n}\auto u_{\infty}
\quad\text{in }BV((a,b)),
\nonumber
\end{equation} 
if
\begin{equation}
u_{n}\to u_{\infty} \text{ in } L^{1}((a,b))
\qquad\text{and}\qquad
|Du_{n}|((a,b))\to|Du_{\infty}|((a,b)).
\nonumber
\end{equation} 

\end{em}
\end{defn}

\begin{defn}[Locally strict convergence on the whole real line]\label{defn:BV-lsc}
\begin{em}

A sequence of functions $\{u_{n}\}\subseteq BV\loc(\re)$ converges \emph{locally strictly} to some $u_{\infty}\in BV\loc(\re)$, and we write
\begin{equation}
u_{n}\auto u_{\infty}
\quad\text{in }BV\loc(\re),
\nonumber
\end{equation} 
if $u_{n}\auto u_{\infty}$ in $BV((a,b))$ for every interval $(a,b)\subseteq\re$ whose endpoints are not jump points of the limit $u_{\infty}$.

\end{em}
\end{defn}

Both definitions can be extended in the usual way to families depending on real parameters. For example, $u_{\ep}\auto u_{0}$ in $BV((a,b))$ as $\ep\to 0^{+}$ if and only if $u_{\ep_{n}}\auto u_{0}$ in $BV((a,b))$ for every sequence $\ep_{n}\to 0^{+}$.

In the following remark we recall some consequences of strict convergence.

\begin{rmk}[Consequences of strict convergence]\label{rmk:strict}
\begin{em}

Let us assume that $u_{n}\auto u_{\infty}$ in $BV((a,b))$. Then the following facts hold true.
\begin{enumerate}
\renewcommand{\labelenumi}{(\arabic{enumi})}

\item  It turns out that $\{u_{n}\}$ is bounded in $L^{\infty}((a,b))$, and $u_{n}\to u_{\infty}$ in $L^{p}((a,b))$ for every $p\geq 1$ (but not necessarily for $p=+\infty$). 

\item \label{strict:cont} For every $x\in(a,b)$, and every sequence $x_{n}\to x$, it turns out that
\begin{equation}
\liminf_{y\to x}u_{\infty}(y)\leq\liminf_{n\to +\infty}u_{n}(x_{n})\leq
\limsup_{n\to +\infty}u_{n}(x_{n})\leq\limsup_{y\to x}u_{\infty}(y),
\nonumber
\end{equation}
and in particular $u_{n}(x_{n})\to u_{\infty}(x)$ whenever $u_{\infty}$ is continuous in $x$, and the convergence is uniform in $(a,b)$ if the limit $u_{\infty}$ is continuous in $(a,b)$.

\item \label{strict:sub-int}  It turns out that $u_{n}\auto u_{\infty}$ in $BV((c,d))$ for every interval $(c,d)\subseteq(a,b)$ whose endpoints are not jump points of the limit $u_{\infty}.$

\item\label{th:conv*+-}  The positive and negative part of the distributional derivatives converge separately in the \emph{closed} interval (see~\cite[Proposition~3.15]{AFP}). More precisely, if $D^{+}u_{n}$ and $D^{-}u_{n}$ denote, respectively, the positive and negative part of the signed measure $Du_{n}$, and similarly for $u_{\infty}$, then for every continuous test function $\phi:[a,b]\to\re$ it turns out that
\begin{equation*}
\lim_{n\to +\infty}\int_{[a,b]}\phi(x)\,dD^{+}u_{n}(x)=
\int_{[a,b]}\phi(x)\,dD^{+}u_{\infty}(x),
\end{equation*}
and similarly with $D^{-}u_{n}$ and $D^{-}u_{\infty}$.
\end{enumerate}

\end{em}
\end{rmk}


In our second main result we consider any family $\{\uep(x)\}$ of minimizers to (\ref{defn:SPM}) and any family of points $x_{\ep}\to x_{0}\in(0,1)$, and we investigate the asymptotic behavior of the family of fake blow-ups (we call them ``fake'' because in the numerator we subtract $f(x_{\ep})$ instead of $\uep(x_{\ep})$) defined by
\begin{equation}
w_{\ep}(y):=\frac{\uep(x_{\ep}+\omep y)-f(x_{\ep})}{\omep}
\qquad
\forall y\in\left(-\frac{x_{\ep}}{\omep},\frac{1-x_{\ep}}{\omep}\right),
\label{defn:wep}
\end{equation}
and the asymptotic behavior of the family of true blow-ups defined by
\begin{equation}
v_{\ep}(y):=\frac{\uep(x_{\ep}+\omep y)-\uep(x_{\ep})}{\omep}\qquad
\forall y\in\left(-\frac{x_{\ep}}{\omep},\frac{1-x_{\ep}}{\omep}\right).
\label{defn:vep}
\end{equation}

We prove that both families are relatively compact in the sense of locally strict convergence, and all their limit points are suitable staircases.

\begin{thm}[Blow-up of minimizers at standard resolution]\label{thm:BU}

Let $\omep$ be defined by (\ref{defn:omep}), and let $\PMF_{\ep}(\beta,f,(0,1),u)$ be defined by (\ref{defn:SPM}), where $\ep\in(0,1)$ and $\beta>0$ are two real numbers, and $f\in C^{1}([0,1])$ is a given function. 

Let $\{u_{\ep}\}\subseteq H^{2}((0,1))$ be a family of functions with
\begin{equation}
u_{\ep}\in\operatorname{argmin}\left\{\PMF_{\ep}(\beta,f,(0,1),u):u\in H^{2}((0,1))\strut\right\}
\qquad
\forall\ep\in(0,1),
\nonumber
\end{equation}
and let $x_{\ep}\to x_{0}\in(0,1)$ be a family of points. Let us consider the canonical $(H,V)$-staircase with parameters
\begin{equation}
H:=\left(\frac{24}{\beta^{2}|f'(x_{0})|^{3}}\right)^{1/5},
\qquad\qquad
V:=f'(x_{0})H,
\label{defn:HV}
\end{equation}
with the agreement that this staircase is identically equal to~0 when $f'(x_{0})=0$. 

Then the following statements hold true.

\begin{enumerate}
\renewcommand{\labelenumi}{(\arabic{enumi})}

\item \emph{(Compactness of fake blow-ups).} The family $\{w_{\ep}(y)\}$ defined by (\ref{defn:wep}) is relatively compact with respect to locally strict convergence, and every limit point is an oblique translation of the canonical $(H,V)$-staircase.

More precisely, for every sequence $\{\ep_{n}\}\subseteq(0,1)$ with $\ep_{n}\to 0^{+}$ there exist an increasing sequence $\{n_{k}\}$ of positive integers and a function $w_{\infty}\in\obl(H,V)$ such that
\begin{equation}
w_{\ep_{n_{k}}}(y)\auto w_{\infty}(y)
\quad\text{in }BV\loc(\re).
\nonumber
\end{equation}

\item \emph{(Compactness of true blow-ups).} The family $\{v_{\ep}(y)\}$ defined by (\ref{defn:vep}) is relatively compact with respect to locally strict convergence, and every limit point is a graph translation of the canonical $(H,V)$-staircase.

More precisely, for every sequence $\{\ep_{n}\}\subseteq(0,1)$ with $\ep_{n}\to 0^{+}$ there exist an increasing sequence $\{n_{k}\}$ of positive integers and a function $v_{\infty}\in\orz(H,V)\cup\vrt(H,V)$ such that
\begin{equation}
v_{\ep_{n_{k}}}(y)\auto v_{\infty}(y)
\quad\text{in }BV\loc(\re).
\nonumber
\end{equation}

\item \emph{(Realization of all possible oblique translations)}. Let $w_{0}\in\obl(H,V)$ be any oblique translation of the canonical $(H,V)$-staircase. 

Then there exists a family $\{x'_{\ep}\}\subseteq(0,1)$ such that
\begin{equation}
\limsup_{\ep\to 0^{+}}\frac{|x'_{\ep}-x_{\ep}|}{\omep}\leq H,
\label{th:liminf-limsup}
\end{equation}
and
\begin{equation}
\frac{u_{\ep}(x'_{\ep}+\omep y)-f(x'_{\ep})}{\omep}\auto w_{0}(y)
\quad\text{in }BV\loc(\re).
\label{th:BU-conv-fake}
\end{equation} 

\item \emph{(Realization of all possible graph translations)}. Let $v_{0}\in\orz(H,V)\cup\vrt(H,V)$ be any graph translation of the canonical $(H,V)$-staircase. 

Then there exists a family $\{x'_{\ep}\}\subseteq(0,1)$ satisfying (\ref{th:liminf-limsup})
and
\begin{equation}
\frac{u_{\ep}(x'_{\ep}+\omep y)-u_{\ep}(x'_{\ep})}{\omep}\auto v_{0}(y)
\quad\text{in }BV\loc(\re).
\nonumber
\end{equation} 

\end{enumerate}

\end{thm}

Let us make some comments about Theorem~\ref{thm:BU} above. To begin with, we consider the special case of stationary points, and the special case of blow-ups in boundary points.

\begin{rmk}[Stationary points of the forcing term]
\begin{em}

In the special case where $f'(x_{0})=0$, the canonical $(H,V)$-staircase is identically equal to~0, and it coincides with all its oblique or graph translations. In this case the whole family of fake blow-ups and the whole family of true blow-ups converge to~0, without any need of subsequences.

\end{em}
\end{rmk}

\begin{rmk}[Internal vs boundary blow-ups]\label{rmk:BU-semi-int}
\begin{em}

For the sake of shortness, we stated the result in the case where $x_{0}\in(0,1)$. The very same conclusions hold true, with exactly the same proof, even if $x_{0}\in\{0,1\}$, provided that
\begin{equation}
\lim_{\ep\to 0^{+}}\frac{\min\{x_{\ep},1-x_{\ep}\}}{\omep}=+\infty.
\label{hp:BU-internal}
\end{equation}

When $x_{0}\in\{0,1\}$ and (\ref{hp:BU-internal}) fails, we can again characterize the limits of fake and true blow-ups, more or less with the same techniques. This requires a one-sided variant of the canonical staircases that we discuss later in section~\ref{sec:loc-min}. We refer to Remark~\ref{rmk:boundary-proof} for further details.

\end{em}
\end{rmk}

In the following remark we present the result from two different points of view.

\begin{rmk}[Further interpretations of Theorem~\ref{thm:BU}]\label{rmk:AM}
\begin{em}

Let us consider any distance in the space $\X:=BV\loc(\re)$ that induces the locally strict convergence. Given any minimizer $\uep$ to (\ref{defn:SPM}), we extend it to a continuous function $\widehat{u}_{\ep}:\re\to\re$ by setting 
\begin{equation}
\widehat{u}_{\ep}(x):=
\begin{cases}
\uep(0)      & \text{if }x\leq 0, \\
\uep(x)      & \text{if }x\in[0,1], \\
\uep(1)      & \text{if }x\geq 1.
\end{cases}
\nonumber
\end{equation}

Then we consider the function $U_{\ep}:(0,1)\to \X$ defined by
\begin{equation}
[U_{\ep}(x)](y):=\frac{\widehat{u}_{\ep}(x+\omep y)-f(x)}{\omep}
\qquad
\forall y\in\re,
\nonumber
\end{equation}
namely the function that associates to every $x\in(0,1)$ the fake blow-up of $\widehat{u}_{\ep}$ with center in $x$ at scale $\omep$.

Finally, for every $x\in(0,1)$ we consider the set $T(x)\subseteq \X$ consisting of all oblique translations of the canonical $(H,V)$-staircase with parameters given by (\ref{defn:HV}). We observe that $T(x)$ is homeomorphic to the circle $S^{1}$ if $f'(x)\neq 0$, and $T(x)$ is a singleton if $f'(x)=0$.

Then ``$U_{\ep}(x)$ converges to $T(x)$'' in the following senses.
\begin{enumerate}
\renewcommand{\labelenumi}{(\arabic{enumi})}

\item  (Hausdorff convergence). For every interval $[a,b]\subseteq(0,1)$ we consider the graph of $U_{\ep}$ over $[a,b]$, namely
\begin{equation}
G_{\ep}(a,b):=\{(x,w):x\in[a,b],\ w=U_{\ep}(x)\}\subseteq [a,b]\times \X,
\nonumber
\end{equation}
and the graph of the multi-function $T(x)$, namely the set
\begin{equation}
G_{0}(a,b):=\{(x,w):x\in[a,b],\ w\in T(x)\}\subseteq [a,b]\times \X.
\nonumber
\end{equation}

Then it turns out that $G_{\ep}\to G_{0}$ as $\ep\to 0^{+}$ with respect to the Hausdorff distance between compact subsets of $(0,1)\times \X$.

This convergence result is a direct consequence of statements~(1) and~(3) of Theorem~\ref{thm:BU}. It can also be extended to true blow-ups, just by defining $T(x)$ as the set of graph translations instead of oblique translations.

\item  (Young measure convergence). Let us consider $U_{\ep}$ as a Young measure $\nu_{\ep}$ in $(0,1)$ with values in $\X$. Let $\nu_{0}$ denote the Young measure that associates to every $x\in(0,1)$ the probability measure in $T(x)$ that is invariant by oblique translations. Then it turns out that 
\begin{equation}
\nu_{\ep}\rightharpoonup\nu_{0}
\qquad\text{as }\ep\to 0^{+},
\nonumber
\end{equation}
where the convergence is in the sense of $\X$-valued Young measures in $(0,1)$. We point out that the strict convergence induced by the distance in our space $\X$ is much stronger than the convergence in~\cite{2001-CPAM-AlbertiMuller}, where the distance just induces the weak* topology in a ball of $L^{\infty}$. For this reason, our space $\X$ is not compact, but we could easily recover the compactness by restricting ourselves to the subset consisting of all blow-ups of all minimizers for $\ep$ is some interval $(0,\ep_{0}]\subseteq(0,1)$, together with all their possible limits as $\ep\to 0^{+}$.

This convergence in the sense of Young measures follows from the Hausdorff convergence and the invariance of $\nu_{0}$ by oblique translations, which in turn follows from a remake of~\cite[Proposition~3.1 and Lemma~2.7]{2001-CPAM-AlbertiMuller}. The argument is however analogous to the proof of statement~(3) of Theorem~\ref{thm:BU}. The idea is that any translation of the blow-up point of order $\omep$ delivers a proportional oblique translation of the limit. 

In the case of true blow-ups we expect the limit Young measure $\nu_{0}$ to be uniformly concentrated only on graph translations of horizontal type, while graph translations of vertical type should have zero measure because they correspond to a very special choice of the blow-up points.

\end{enumerate}

\end{em}
\end{rmk}


Theorem~\ref{thm:BU} shows that minimizers develop a microstructure at scale~$\omep$. As a consequence, this microstructure does not appear if we consider blow-ups at a coarser scale, as in the following statement.

\begin{cor}[Low-resolution blow-ups of minimizers]\label{thm:spm-bu-low}

Let $\ep$, $\omep$, $\beta$, $f$, $u_{\ep}$, $x_{0}$ be as in Theorem~\ref{thm:BU}. Let $\{x_{\ep}\}\subseteq(0,1)$ be a family of real numbers such that $x_{\ep}\to x_{0}$, and let $\{\alpha_{\ep}\}$ be a family of positive real numbers such that $\alpha_{\ep}\to 0$ but $\omep/\alpha_{\ep}\to 0$.

Then it turns out that
\begin{equation}
\frac{u_{\ep}(x_{\ep}+\alpha_{\ep}y)-u_{\ep}(x_{\ep})}{\alpha_{\ep}}\auto f'(x_{0})y
\qquad
\text{in }BV\loc(\re),
\nonumber
\end{equation}
and therefore also uniformly on bounded subsets of $\re$.

\end{cor}

The second consequence of Theorem~\ref{thm:BU} is an improvement of statement~(4) in Proposition~\ref{prop:basic}, at least in the case where the forcing term $f(x)$ is of class $C^{1}$. In this case indeed we obtain that minimizers converge to $f$ also in the sense of strict convergence. Moreover, as $u_{\ep}(x)$ converges to $f(x)$, its derivative $u_{\ep}'(x)$ converges to a mix of $0$ and $\pm\infty$, and this mix is ``in the average'' equal to $f'(x)$. We state the result using an elementary language, and then we interpret it in the formalism of varifolds.

\begin{thm}[Convergence of minimizers to the forcing term]\label{thm:varifold}

Let $\ep$, $\beta$, $f$, $u_{\ep}$ be as in Theorem~\ref{thm:BU}. 

Then the family $\{u_{\ep}(x)\}$ of minimizers converges to $f(x)$ in the following senses.

\begin{enumerate}
\renewcommand{\labelenumi}{(\arabic{enumi})}

\item \emph{(Strict convergence).} It turns out that $u_{\ep}(x)\auto f(x)$ in $BV((0,1))$, and therefore also uniformly in $[0,1]$.

\item \emph{(Convergence as varifolds).}  Let us set
\begin{equation}
V_{0}^{+}:=\left\{x\in[0,1]:f'(x)>0\right\},
\qquad
V_{0}^{-}:=\left\{x\in[0,1]:f'(x)<0\right\}.
\label{defn:V0+-}
\end{equation}

Then for every continuous test function
\begin{equation}
\phi:[0,1]\times\re\times\left[-\frac{\pi}{2},\frac{\pi}{2}\right]\to\re
\nonumber
\end{equation}
it turns out that
\begin{multline}
\lim_{\ep\to 0^{+}}\int_{0}^{1}
\phi\left(x,u_{\ep}(x),\arctan(u_{\ep}'(x))\strut\right)\sqrt{1+u_{\ep}'(x)^{2}}\,dx=
\int_{0}^{1}\phi(x,f(x),0)\,dx
\\[1ex]
+\int_{V_{0}^{-}}\phi\left(x,f(x),-\frac{\pi}{2}\strut\right)|f'(x)|\,dx
+\int_{V_{0}^{+}}\phi\left(x,f(x),\frac{\pi}{2}\strut\right)|f'(x)|\,dx.
\label{th:varifold}
\end{multline}

\end{enumerate}

\end{thm}

The conclusions of Theorem~\ref{thm:varifold} is weaker than Theorem~\ref{thm:BU}, because it does not carry so much information about the asymptotic profile of minimizers. Just for comparison, the counterpart of this result in the Alberti-M\"uller model is the convergence of $u_{\ep}'(x)$ to a Young measure that in every point assumes the two values $\pm 1$ with equal probability. Therefore, we suspect that the same conclusion might be true under weaker assumptions on the forcing term $f(x)$, and we refer to section~\ref{sec:open} for further discussion. 

\begin{rmk}[Varifold interpretation]
\begin{em}

Let us limit ourselves for a while to test functions such that $\phi(x,s,\pi/2)=\phi(x,s,-\pi/2)$ for all admissible values of $x$ and $s$. Let us observe that the function $p\mapsto\arctan(p)$ is a homeomorphism between the projective line and the interval $[-\pi/2,\pi/2]$ with the endpoint identified. Under these assumptions we can interpret the two sides of (\ref{th:varifold}) as the action of two suitable varifolds on the test function $\phi$. 

In the left-hand side we have the varifold associated to the graph of $u_{\ep}(x)$ in the canonical way, namely with ``weight'' (projection into $\re^{2}$) equal to the restriction of the one-dimensional Hausdorff measure to the graph of $\uep(x)$, and ``tangent component'' in the direction of the derivative $\uep'(x)$. In the right-hand side we have a varifold such that
\begin{itemize}

\item  the ``weight'' is the one-dimensional Hausdorff measure restricted to the graph of $f(x)$, multiplied by the density
\begin{equation}
\frac{1+|f'(x)|}{\sqrt{1+f'(x)^{2}}},
\nonumber
\end{equation}
which in turn coincides with the push-forward of the Lebesgue measure through the map $x\mapsto(x,f(x))$ multiplied by $1+|f'(x)|$,

\item  the ``tangent component'' in the point $(x,f(x))$ is equal to
\begin{equation}
\frac{1}{1+|f'(x)|}\,\delta_{(1,0)}+\frac{|f'(x)|}{1+|f'(x)|}\,\delta_{(0,1)},
\nonumber
\end{equation}
where $\delta_{(1,0)}$ and $\delta_{(0,1)}$ are the Dirac measures concentrated in the horizontal direction $(1,0)$ and in the vertical direction $(0,1)$, respectively.

\end{itemize}
 
It follows that statement~(2) of Theorem~\ref{thm:varifold} above is a reinforced version of varifold convergence. The reinforcement consists in considering the vertical tangent line in the direction $(0,1)$ as different from the vertical tangent line in the direction $(0,-1)$.

\end{em}
\end{rmk}

\begin{rmk}[Minimality is essential]
\begin{em}

In Theorem~\ref{thm:BU} and Theorem~\ref{thm:varifold} we can not replace the requirement that $\{\uep\}$ is a family of minimizers by weaker ``almost minimality'' conditions such as
\begin{equation}
\lim_{\ep\to 0^{+}}\frac{\PMF_{\ep}(\beta,f,(0,1),\uep)}{m(\ep,\beta,f)}=1.
\label{hp:almost-min}
\end{equation}

Indeed, one can check that the cost of adding an isolated bump that simulates two opposite jumps in a neighborhood of some point is proportional to $\omep^{5/2}$ (see also (\ref{th:lim-PJ}) below). Since the denominator in (\ref{hp:almost-min}) is proportional to $\omep^{2}$, this condition does not even imply a uniform bound on the total variation of $\uep$.

\end{em}
\end{rmk}


\setcounter{equation}{0}
\section{Functional setting and Gamma-convergence}\label{sec:gconv}

This section deals with the rescaled version of the Perona-Malik functional (\ref{defn:SPM}) and its Gamma-limit. The results are somewhat classical, and rather close to similar results in the literature. On the other hand, in some cases they are not stated in the literature in the form we need, and in some other cases the proofs that we found in the literature do not work. Therefore, for the convenience of the reader we decided to include at least a sketch of the proofs in an appendix at the end of the paper.

\paragraph{\textmd{\textit{Functional setting}}}

Let us consider the functional
\begin{equation}
\RPM_{\ep}(\Omega,u):=
\int_{\Omega}\left\{\ep^{6}u''(x)^{2}+\frac{1}{\omep^{2}}\log\left(1+u'(x)^{2}\right)\right\}\,dx
\label{defn:ABG}
\end{equation}
defined for every real number $\ep\in(0,1)$, every open set $\Omega\subseteq\re$, and every function $u\in H^{2}(\Omega)$. This functional is a rescaled version of the principal part of (\ref{defn:SPM}). When we add the usual ``fidelity term'', depending on a real parameter $\beta>0$ and on a forcing term $f\in L^{2}(\Omega)$, we obtain the rescaled Perona-Malik functional with fidelity term
\begin{equation}
\RPMF_{\ep}(\beta,f,\Omega,u):=
\RPM_{\ep}(\Omega,u)+\beta\int_{\Omega}\left(u(x)-f(x)\right)^{2}\,dx.
\label{defn:ABGF}
\end{equation}

The Gamma-limit of (\ref{defn:ABG}) as $\ep\to 0^{+}$ turns out to be finite only in the space of ``pure jump functions'', defined as finite or countable linear combination of Heaviside functions. More formally, the notion is the following.

\begin{defn}[Pure jump functions]\label{defn:PJF}
\begin{em}

Let $(a,b)\subseteq\re$ be an interval. A function $u:(a,b)\to\re$ is called a \emph{pure jump function}, and we write $u\in\PJ((a,b))$, if there exist a real number $c$, a finite or countable set $S_{u}\subseteq(a,b)$, and a function $J:S_{u}\to\re\setminus\{0\}$ such that
\begin{equation}
\sum_{s\in S_{u}}|J(s)|<+\infty
\label{defn:sum-conv}
\end{equation}
and
\begin{equation}
u(x)=c+\sum_{s\in S_{u}}J(s)\mathbbm{1}_{(s,b)}(x)
\qquad
\forall x\in (a,b),
\label{defn:PJ}
\end{equation}
where $\mathbbm{1}_{(s,b)}:\re\to\{0,1\}$ is the indicator function of the interval $(s,b)$, defined as
\begin{equation}
\mathbbm{1}_{(s,b)}(x):=\begin{cases}
1     & \text{if }x\in(s,b), \\
0     & \text{otherwise}.
\end{cases}
\nonumber
\end{equation} 

The set $S_{u}$ is called the \emph{jump set} of $u$, every element $s\in S_{u}$ is called a \emph{jump point} of $u$, and $|J(s)|$ is called the \emph{height of the jump} of $u$ in $s$.

We call \emph{boundary values} of $u$ the numbers
\begin{equation}
u(a):=\lim_{x\to a^{+}}u(x)=c
\qquad\quad\text{and}\qquad\quad
u(b):=\lim_{x\to b^{-}}u(x)=c+\sum_{s\in S_{u}}J(s).
\label{defn:PJ-BC}
\end{equation}

\end{em}
\end{defn}

Pure jump functions can be defined in an alternative way as those functions in $BV((a,b))$ whose distributional derivative is a finite or countable linear combination of atomic measures. In particular, it can be verified that the representation (\ref{defn:PJ}) is unique, and defines a function $u\in BV((a,b))$ whose total variation is the sum of the series in (\ref{defn:sum-conv}), and whose distributional derivative $Du$ is the sum of Dirac measures concentrated in the points of the set $S_{u}$ with weight $J(s)$. Moreover, $S_{u}$ coincides with the set of discontinuity points of $u$, and 
\begin{equation}
J(s)=\lim_{x\to s^{+}}u(x)-\lim_{x\to s^{-}}u(x)
\qquad
\forall s\in S_{u}.
\nonumber
\end{equation}

We can now introduce the functional 
\begin{equation}
\operatorname{\mathbb{J}}_{1/2}(\Omega,u):=\sum_{s\in S_{u}\cap\Omega}|J(s)|^{1/2},
\label{defn:J}
\end{equation}
defined for every $u\in\PJ((a,b))$ and every open subset $\Omega\subseteq(a,b)$. Of course the convergence of the series in (\ref{defn:sum-conv}) does not imply the convergence of the series in (\ref{defn:J}), and therefore at this level of generality it may happen that $\operatorname{\mathbb{J}}_{1/2}(\Omega,u)=+\infty$ for some choices of $u$ and $\Omega$.


\paragraph{\textmd{\textit{Gamma-convergence}}}

The following result concerns the convergence of the family $\RPM_{\ep}$ to a multiple of $\J_{1/2}$. The compactness statement is similar to~\cite[Theorem~4.1]{2008-TAMS-BF}, while the Gamma-convergence statement coincides with \cite[Theorem~4.4]{2008-TAMS-BF} in the special case $\phi(p)=\log(1+p^{2})$. Nevertheless, unfortunately the proof in~\cite{2008-TAMS-BF} relies on~\cite[Lemma~3.1]{2008-TAMS-BF}, which is clearly false for this choice of $\phi(p)$. In the appendix at the end of the paper we present a specific proof for this case.

\begin{thm}[Gamma-convergence, compactness, properties of recovery sequences]\label{thm:ABG}

Let $(a,b)\subseteq\re$ be an interval, let us consider the functionals defined in (\ref{defn:ABG}) and (\ref{defn:J}), and let us set 
\begin{equation}
\alpha_{0}:=
\frac{16}{\sqrt{3}}.
\label{defn:alpha-0}
\end{equation}

Then the following statements hold true. 

\begin{enumerate}
\renewcommand{\labelenumi}{(\arabic{enumi})}

\item \label{stat:ABG-gconv} \emph{(Gamma convergence)} Let us extend the functionals (\ref{defn:ABG}) and (\ref{defn:J}) to the space $L^{2}((a,b))$ by setting them equal to $+\infty$ outside their original domains. 

Then with respect to the metric of $L^{2}((a,b))$ it turns out that
\begin{equation}
\glim_{\ep\to 0^{+}}\RPM_{\ep}((a,b),u)=\alpha_{0}\J_{1/2}((a,b),u)
\qquad
\forall u\in L^{2}((a,b)).
\nonumber
\end{equation}

\item \label{stat:ABG-cpt} \emph{(Compactness)} Let $\{\ep_{n}\}\subseteq(0,1)$ be any sequence such that $\ep_{n}\to 0^{+}$, and let $\{u_{n}\}\subseteq H^{2}((a,b))$ be any sequence such that
\begin{equation}
\sup_{n\in\n}\left\{\RPM_{\ep_{n}}((a,b),u_{n})+\int_{a}^{b}u_{n}(x)^{2}\,dx\right\}<+\infty.
\label{hp:Gconv-coercive}
\end{equation}

Then there exist an increasing sequence $\{n_{k}\}$ of positive integers, and a function $u_{\infty}\in \PJ((a,b))$ such that $u_{n_{k}}\to u_{\infty}$ in $L^{2}((a,b))$ as $k\to +\infty$.

\item \label{stat:ABG-uc} \emph{(Strict convergence of recovery sequences)}  Let $u\in\PJ((a,b))$ be a pure jump function with $\J_{1/2}((a,b),u)<+\infty$. Let $\{\ep_{n}\}\subseteq(0,1)$ be any sequence such that $\ep_{n}\to 0^{+}$, and let $\{u_{n}\}\subseteq H^{2}((a,b))$ be any sequence such that $u_{n}\to u$ in $L^{2}((a,b))$, and
\begin{equation}
\lim_{n\to +\infty}\RPM_{\ep_{n}}((a,b),u_{n})=\alpha_{0}\J_{1/2}((a,b),u).
\label{hp:recovery}
\end{equation}

Then actually $u_{n}\auto u$ in $BV((a,b))$, according to Definition~\ref{defn:BV-sc}.

\item \label{stat:ABG-recovery} \emph{(Recovery sequences with given boundary data)}  Let $\{\ep_{n}\}$ and $u$ be as in the previous statement, and let $\{A_{0,n}\}$, $\{A_{1,n}\}$, $\{B_{0,n}\}$, $\{B_{1,n}\}$ be four sequences of real numbers such that
\begin{equation}
\lim_{n\to +\infty}\left(A_{0,n},A_{1,n},B_{0,n},B_{1,n}\right)=(u(a),0,u(b),0),
\nonumber
\end{equation}
where the boundary values of $u$ are intended as usual in the sense of (\ref{defn:PJ-BC}).

Then there exists a sequence $\{u_{n}\}\subseteq H^{2}((a,b))$ with boundary data 
\begin{equation}
\left(u_{n}(a),u_{n}'(a),u_{n}(b),u_{n}'(b)\strut\right)=
(A_{0,n}, A_{1,n}, B_{0,n}, B_{1,n})
\qquad
\forall n\in\n
\label{hp:recovery-BC}
\end{equation}
such that $u_{n}\to u$ in $L^{2}((a,b))$ and (\ref{hp:recovery}) holds true.

\end{enumerate}

\end{thm}

\begin{rmk}
\begin{em}

The choice of the ambient space $L^{2}((a,b))$ is not essential in Theorem~\ref{thm:ABG}, and actually it can be replaced with $L^{p}((a,b))$ for any real exponent $p\geq 1$ (but not for $p=+\infty$, at least in statements~(2) and~(4)).

\end{em}
\end{rmk}


\paragraph{\textmd{\textit{Convergence of minima and minimizers}}}

Since the fidelity term in (\ref{defn:ABGF}) is continuous with respect to the metric of $L^{2}(\Omega)$, and Gamma-convergence is stable with respect to continuous perturbations, we deduce that the Gamma-limit of (\ref{defn:ABGF}) is the functional
\begin{equation}
\JF_{1/2}(\alpha,\beta,f,\Omega,u):=
\alpha\J_{1/2}(\Omega,u)+\beta\int_{\Omega}(u(x)-f(x))^{2}\,dx,
\label{defn:JF}
\end{equation}
with $\alpha$ equal to the constant $\alpha_{0}$ defined in (\ref{defn:alpha-0}).  Now we concentrate on the special case where $\Omega=(0,L)$ and the forcing term is the linear function $f(x)=Mx$, for suitable real numbers $L>0$ and $M$, and we consider the following minimum values without boundary conditions
\begin{gather}
\mu_{\ep}(\beta,L,M):=\min_{u\in H^{2}((0,L))}
\RPMF_{\ep}(\beta,Mx,(0,L),u),
\label{defn:muep}
\\[1ex]
\mu_{0}(\alpha,\beta,L,M):=
\min_{u\in\PJ((0,L))}\JF_{1/2}(\alpha,\beta,Mx,(0,L),u).
\label{defn:mu0}
\end{gather}

Then we introduce boundary conditions. In the case of (\ref{defn:ABGF}) we call $H^{2}((0,L),M)$ the set of all functions $v\in H^{2}((0,L))$ such that $v(0)=0$, $v(L)=ML$, and $v'(0)=v'(L)=0$. In the case of (\ref{defn:JF}) we call $\PJ((0,L),M)$ the set of all functions $v\in\PJ((0,L))$ such that $v(0)=0$ and $v(L)=ML$, where these boundary values are intended in the sense of (\ref{defn:PJ-BC}). At this point we consider the following minimum values with boundary conditions
\begin{gather}
\mu_{\ep}^{*}(\beta,L,M):=\min_{u\in H^{2}((0,L),M)}
\RPMF_{\ep}(\beta,Mx,(0,L),u),
\label{defn:muep*}
\\[1ex]
\mu_{0}^{*}(\alpha,\beta,L,M):=
\min_{u\in\PJ((0,L),M)}\JF_{1/2}(\alpha,\beta,Mx,(0,L),u).
\label{defn:mu0*}
\end{gather}

The following result contains the properties of these minimum values that we exploit in the sequel (a sketch of the proof is in the appendix).

\begin{prop}[Asymptotic analysis of minima with linear forcing term]\label{prop:mu}

The minimum values defined in (\ref{defn:muep}) through (\ref{defn:mu0*}) have the following properties.

\begin{enumerate}
\renewcommand{\labelenumi}{(\arabic{enumi})}

\item \label{prop:existence} \emph{(Existence).} The minimum problems (\ref{defn:muep}) through (\ref{defn:mu0*}) admit a solution for every $(\ep,\alpha,\beta,L,M)\in(0,1)\times(0,+\infty)^{3}\times\re$. 

\item \label{prop:M} \emph{(Symmetry, continuity and monotonicity with respect to $M$).} For every admissible value of $\ep$, $\alpha$, $\beta$, $L$ the four functions
\begin{gather*}
M\mapsto\mu_{\ep}(\beta,L,M),
\qquad\qquad
M\mapsto\mu_{\ep}^{*}(\beta,L,M),
\\[0.5ex]
M\mapsto\mu_{0}(\alpha,\beta,L,M),
\qquad\qquad
M\mapsto\mu_{0}^{*}(\alpha,\beta,L,M),
\end{gather*}
are even, continuous in $\re$, and nondecreasing in $[0,+\infty)$.

\item \label{prop:L} \emph{(Monotonicity with respect to $L$).} For every admissible value of $\ep$, $\alpha$, $\beta$, $M$, the three functions
\begin{equation}
L\mapsto\mu_{\ep}(\beta,L,M),
\qquad\quad
L\mapsto\mu_{0}(\alpha,\beta,L,M),
\qquad\quad
L\mapsto\mu_{0}^{*}(\alpha,\beta,L,M)
\nonumber
\end{equation}
are nondecreasing with respect to $L$ in $(0,+\infty)$. As for $\mu_{\ep}^{*}$, it turns out that
\begin{equation}
\mu_{\ep}^{*}(\beta,L_{2},M)\leq\left(\frac{L_{2}}{L_{1}}\right)^{3}\mu_{\ep}^{*}(\beta,L_{1},M)
\label{th:monot-muep*}
\end{equation}
for every $0<L_{1}\leq L_{2}$.

\item \label{prop:pointwise} \emph{(Pointwise convergence).} For every admissible value of $\beta$, $M$ and $L$ it turns out that
\begin{equation}
\lim_{\ep\to 0^{+}}\mu_{\ep}(\beta,L,M)=
\mu_{0}(\alpha_{0},\beta,L,M),
\label{th:lim-mu}
\end{equation}
and
\begin{equation}
\lim_{\ep\to 0^{+}}\mu_{\ep}^{*}(\beta,L,M)=
\mu_{0}^{*}(\alpha_{0},\beta,L,M),
\label{th:lim-mu*}
\end{equation}
where $\alpha_{0}$ is the constant defined in (\ref{defn:alpha-0}). 

\item \label{prop:uniform} \emph{(Uniform convergence).} The limits (\ref{th:lim-mu}) and (\ref{th:lim-mu*}) are uniform for bounded values of $M$, in the sense that for every positive value of $\beta$ and $L$ it turns out that
\begin{equation}
\lim_{\ep\to 0^{+}}\sup_{|M|\leq M_{0}}
\left|\mu_{\ep}(\beta,L,M)-\mu_{0}(\alpha_{0},\beta,L,M)\right|=0
\qquad
\forall M_{0}>0,
\label{th:lim-mu-unif}
\end{equation}
and 
\begin{equation}
\lim_{\ep\to 0^{+}}\sup_{|M|\leq M_{0}}
\left|\mu^{*}_{\ep}(\beta,L,M)-\mu_{0}^{*}(\alpha_{0},\beta,L,M)\right|=0
\qquad
\forall M_{0}>0.
\label{th:lim-mu*-unif}
\end{equation}

\end{enumerate}

\end{prop}


\setcounter{equation}{0}
\section{Local minimizers}\label{sec:loc-min}

In this section we state the key tools for the proof of our main results. The key idea is that also local minimizers for functionals of the form (\ref{defn:ABGF}) converge to local minimizers for functionals of the form (\ref{defn:JF}). This extends the Gamma convergence results of the previous section.

The notion of local minimizers can be introduced in a very general framework by asking minimality with respect to compactly supported perturbations. In many concrete examples this is equivalent to saying that a given function is a minimizer with respect to its own boundary conditions. Of course the number and the form of these boundary conditions depend on the nature of the functional, as we explain below.

\begin{defn}[Local minimizers in intervals]\label{defn:loc-min}
\begin{em}

Let $(a,b)\subseteq\re$ be an interval, and let $\mathcal{F}(u)$ be a functional defined in some functional space $\mathcal{S}((a,b))$.
\begin{itemize}

\item  Let us assume that $\mathcal{S}((a,b))=H^{2}((a,b))$. A \emph{local minimizer }is any function $u\in H^{2}((a,b))$ such that $\mathcal{F}(u)\leq\mathcal{F}(v)$ for every function $v\in H^{2}((a,b))$ such that
\begin{equation}
\left(v(a),v'(a),v(b),v'(b)\strut\right)=\left(u(a),u'(a),u(b),u'(b)\strut\right).
\nonumber
\end{equation}

\item   Let us assume that $\mathcal{S}((a,b))=\PJ((a,b))$. A \emph{local minimizer} is any function $u\in\PJ((a,b))$ such that $\mathcal{F}(u)\leq\mathcal{F}(v)$ for every function $v\in\PJ((a,b))$ such that $(v(a),v(b))=(u(a),u(b))$, where boundary values of pure jump functions are intended in the sense of (\ref{defn:PJ-BC}). 

\end{itemize}

In both cases we write
\begin{equation}
u\in\argmin\loc\left\{\mathcal{F}(u):u\in\mathcal{S}((a,b))\right\}.
\nonumber
\end{equation}

\end{em}
\end{defn}

We observe that in Definition~\ref{defn:loc-min} the two endpoints of the interval play the same role. In the sequel we need also the following notion of one-sided local minimizer, where we focus just on one of the endpoints.

\begin{defn}[One-sided local minimizers in an interval]\label{defn:R-loc-min}
\begin{em}

Let $(a,b)\subseteq\re$ be an interval, and let $\mathcal{F}(u)$ be a functional defined in some functional space $\mathcal{S}((a,b))$.
\begin{itemize}

\item  Let us assume that $\mathcal{S}((a,b))=H^{2}((a,b))$. A \emph{right-hand local minimizer} is any function $u\in H^{2}((a,b))$ such that $\mathcal{F}(u)\leq\mathcal{F}(v)$ for every function $v\in H^{2}((a,b))$ such that $(v(b),v'(b))=(u(b),u'(b))$.

\item   Let us assume that $\mathcal{S}((a,b))=\PJ((a,b))$. A \emph{right-hand local minimizer} is any function $u\in\PJ((a,b))$ such that $\mathcal{F}(u)\leq\mathcal{F}(v)$ for every function $v\in\PJ((a,b))$ such that $v(b)=u(b)$.

\end{itemize}

In both cases we write
\begin{equation}
u\in\argmin\Rloc\left\{\mathcal{F}(u):u\in\mathcal{S}((a,b))\right\}.
\nonumber
\end{equation}

Left-hand local minimizers are defined in a symmetric way, just focusing on the endpoint~$a$.

\end{em}
\end{defn}

\begin{defn}[Entire and semi-entire local minimizers]
\begin{em}

Let us consider functionals $\mathcal{F}(I,u)$ defined for every interval $I$ and every $u$ in some function space $\mathcal{S}(I)$.
\begin{itemize}

\item  An \emph{entire local minimizer} is a function $u:\re\to\re$ such that, for every interval $(a,b)\subseteq\re$, the restriction of $u$ to $(a,b)$ is a local minimizer in $(a,b)$.

\item  A \emph{right-hand semi-entire local minimizer} is a function $u:(0,+\infty)\to\re$ such that, for every real number $L>0$, the restriction of $u$ to $(0,L)$ is a right-hand local minimizer in $(0,L)$.

\item  A \emph{left-hand semi-entire local minimizer} is a function $u:(-\infty,0)\to\re$ such that, for every real number $L>0$, the restriction of $u$ to $(-L,0)$ is a left-hand local minimizer in $(-L,0)$.

\end{itemize}

\end{em}
\end{defn}


The following result is crucial both in the proof of Theorem~\ref{thm:asympt-min}, and as a preliminary step toward the characterization of entire and semi-entire local minimizers to the limit functional (\ref{defn:JF}).

\begin{prop}[Estimates for minima of the limit problem]\label{prop:loc-min-discr}

For every $(\alpha,\beta,L,M)\in(0,+\infty)^{3}\times\re$ the minimum values defined in (\ref{defn:mu0}) and (\ref{defn:mu0*}) satisfy
\begin{align}
\hspace{3em}
c_{1}|M|^{4/5}L-c_{2}|M|^{1/5}&\leq
\mu_{0}(\alpha,\beta,L,M)
\label{th:est-mu-abLM}
\\[0.5ex]
&\leq
\mu_{0}^{*}(\alpha,\beta,L,M)\leq
c_{1}|M|^{4/5}L+c_{3}|M|^{1/5},
\hspace{3em}
\label{th:est-mu*-abLM}
\end{align}
where
\begin{equation}
c_{1}:=\frac{5}{4}\left(\frac{\alpha^{4}\beta}{3}\right)^{1/5},
\qquad
c_{2}:=20\left(\frac{2\alpha^{6}}{3\beta}\right)^{1/5},
\qquad
c_{3}:=\frac{5}{4}\left(\frac{3\alpha^{6}}{\beta}\right)^{1/5}.
\label{defn:c123}
\end{equation}

\end{prop}

We are now ready to state the first main result of this section, namely the characterization of all entire and semi-entire local minimizers for the functional (\ref{defn:JF}).

\begin{prop}[Classification of entire and semi-entire local minimizers]\label{prop:loc-min-class}

For every choice of the real numbers $(\alpha,\beta,M)\in(0,+\infty)^{2}\times\re$ let us consider the functional $\JF_{1/2}(\alpha,\beta,M,\re,v)$ defined in (\ref{defn:JF}). Let us consider the canonical $(H,V)$-staircase $S_{H,V}(x)$ with parameters 
\begin{equation}
H:=\frac{1}{2}\left(\frac{9\alpha^{2}}{\beta^{2}|M|^{3}}\right)^{1/5},
\qquad\qquad
V:=MH,
\label{defn:lambda-0}
\end{equation}
and the usual agreement that $S_{H,V}(x)\equiv 0$ when $M=0$.

Then the following statements hold true.
\begin{enumerate}
\renewcommand{\labelenumi}{(\arabic{enumi})}

\item \emph{(Entire local minimizers).} The set of entire local minimizers coincides with the set of the oblique translations of the canonical $(H,V)$-staircase $S_{H,V}(x)$, as introduced in Definition~\ref{defn:staircase} and Definition~\ref{defn:translations}.

\item \emph{(Semi-entire local minimizers).} The unique right-hand semi-entire local minimizer is the function $w:(0,+\infty)\to\re$ defined by
\begin{equation}
w(x):=\begin{cases}
Mz_{0} & \text{if }x\in(0,z_{0}), \\
S_{H,V}(x-z_{0})+Mz_{0}\quad & \text{if }x\geq z_{0},
\end{cases}
\label{defn:semi-entire}
\end{equation} 
where $z_{0}:=(5/3)^{1/2}H$ (if $M=0$ the value of $z_{0}$ is not relevant).

The unique left-hand semi-entire local minimizer is the function $w(-x)$.

\end{enumerate}

\end{prop}

In words, the right-hand semi-entire local minimizer is an oblique translation of the canonical $(H,V)$-staircase, but with a first step that is longer. Intuitively, this is due to the fact that the ``jump at the origin'' has no cost in terms of energy.


The second main result of this section is the convergence of local minimizers for (\ref{defn:ABGF}) to local minimizers for (\ref{defn:JF}). Let us start with the symmetric case.

\begin{prop}[Convergence to entire local minimizers]\label{prop:bu2minloc}

Let $M$ and $\beta$ be real numbers, with $\beta>0$. For every positive integer $n$, let $\ep_{n}\in(0,1)$ and $A_{n}<B_{n}$ be real numbers, let $g_{n}:(A_{n},B_{n})\to\re$ be a continuous function, and let $w_{n}\in H^{2}((A_{n},B_{n}))$.

Let us assume that
\begin{enumerate}
\renewcommand{\labelenumi}{(\roman{enumi})}

\item  as $n\to +\infty$ it turns out that $\ep_{n}\to 0^{+}$, $A_{n}\to -\infty$, and $B_{n}\to +\infty$,

\item  $g_{n}(x)\to Mx$ uniformly on bounded subsets of $\re$,

\item   for every positive integer $n$ it turns out that
\begin{equation}
w_{n}\in\argmin\loc\left\{\RPMF_{\ep_{n}}(\beta,g_{n},(A_{n},B_{n}),w):w\in H^{2}((A_{n},B_{n}))\right\},
\nonumber
\end{equation}

\item  there exists a positive real number $C_{0}$ such that
\begin{equation}
\RPMF_{\ep_{n}}(\beta,g_{n},(A_{n},B_{n}),w_{n})\leq\frac{C_{0}}{\ep_{n}}
\qquad
\forall n\geq 1.
\label{hp:bound-G}
\end{equation}

\end{enumerate}

Then there exists an increasing sequence $\{n_{k}\}$ of positive integers such that
\begin{equation}
w_{n_{k}}(x)\auto w_{\infty}(x)
\quad\text{in }BV\loc(\re),
\nonumber
\end{equation}
where $w_{\infty}$ is an entire local minimizer for the functional (\ref{defn:JF}) with $\alpha$ given by~(\ref{defn:alpha-0}).

\end{prop}

The result for one-sided local minimizers is analogous. We state it in the case of right-hand local minimizers. 

\begin{prop}[Convergence to semi-entire local minimizers]\label{prop:bu2Rminloc}

Let $M$ and $\beta$ be real numbers, with $\beta>0$. For every positive integer $n$, let $\ep_{n}\in(0,1)$ and $L_{n}>0$ be real numbers, let $g_{n}:(0,L_{n})\to\re$ be a continuous function, and let $w_{n}\in H^{2}((0,L_{n}))$.

Let us assume that
\begin{enumerate}
\renewcommand{\labelenumi}{(\roman{enumi})}

\item  as $n\to +\infty$ it turns out that $\ep_{n}\to 0^{+}$ and $L_{n}\to +\infty$,

\item  $g_{n}(x)\to Mx$ uniformly on bounded subsets of $(0,+\infty)$,

\item   for every positive integer $n$ it turns out that
\begin{equation}
w_{n}\in\argmin\Rloc\left\{\RPMF_{\ep_{n}}(\beta,g_{n},(0,L_{n}),w):w\in H^{2}((0,L_{n}))\right\},
\nonumber
\end{equation}

\item  there exists a positive real number $C_{0}$ such that
\begin{equation}
\RPMF_{\ep_{n}}(\beta,g_{n},(0,L_{n}),w_{n})\leq\frac{C_{0}}{\ep_{n}}
\qquad
\forall n\geq 1.
\nonumber
\end{equation}

\end{enumerate}

Let $w_{\infty}$ denote the unique right-hand semi-entire local minimizer for the functional (\ref{defn:JF}) with $\alpha$ given by~(\ref{defn:alpha-0}), namely the function defined by (\ref{defn:semi-entire}).

Then for every $L>0$ that is not a jump point of $w_{\infty}$ it turns out that
\begin{equation}
w_{n}(x)\auto w_{\infty}(x)
\quad\text{in }BV((0,L)).
\nonumber
\end{equation}

\end{prop}


\setcounter{equation}{0}
\section{Proofs of main results}\label{sec:strategy}

In this section we assume that the results stated in section~\ref{sec:gconv} and section~\ref{sec:loc-min} are valid, and using them we prove all the main results of section~\ref{sec:statements} concerning the behavior of minima and minimizers. We hope that this presentation allows to highlight the main ideas without focusing on the technical details that will be presented in the next section.

\subsection{Asymptotic behavior of minima (Theorem~\ref{thm:asympt-min})}

The proof of Theorem~\ref{thm:asympt-min} consists of two main parts. In the first part (estimate from below) we consider any family $\{u_{\ep}\}\subseteq H^{2}((0,1))$ and we show that
\begin{equation}
\liminf_{\ep\to 0^{+}}\frac{\PMF_{\ep}(\beta,f,(0,1),u_{\ep})}{\omep^{2}}\geq
10\left(\frac{2\beta}{27}\right)^{1/5}\int_{0}^{1}|f'(x)|^{4/5}\,dx.
\label{est:lim-SPM-below}
\end{equation}

In the second part (estimate from above) we construct a family $\{u_{\ep}\}\subseteq H^{2}((0,1))$ such that
\begin{equation}
\limsup_{\ep\to 0^{+}}\frac{\PMF_{\ep}(\beta,f,(0,1),u_{\ep})}{\omep^{2}}\leq
10\left(\frac{2\beta}{27}\right)^{1/5}\int_{0}^{1}|f'(x)|^{4/5}\,dx.
\label{est:lim-SPM-above}
\end{equation}

\subsubsection{Estimate from below}

\paragraph{\textmd{\textit{Interval subdivision and approximation of the forcing term}}}

Let us fix two real numbers $L>0$ and $\eta\in(0,1)$. For every $\ep\in(0,1)$ we set
\begin{equation}
N_{\ep,L}:=\left\lfloor\frac{1}{L\omep}\right\rfloor
\qquad\quad\text{and}\quad\qquad
L_{\ep}:=\frac{1}{N_{\ep,L}\omep}.
\label{defn:NepAep}
\end{equation}

We observe that $N_{\ep,L}$ is an integer, and that $L_{\ep}\to L$ when $\ep\to 0^{+}$. We observe also that $[0,1]$ is (up to a finite number of points) the disjoint union of the $N_{\ep,L}$ intervals of length $L_{\ep}\omep$ defined by
\begin{equation}
I_{\ep,k}:=((k-1)L_{\ep}\omep,kL_{\ep}\omep)
\qquad
\forall k\in\{1,\ldots,N_{\ep,L}\},
\label{defn:Iepk}
\end{equation}
and we consider  the piecewise affine function $f_{\ep,L}:[0,1]\to\re$ that interpolates the values of $f$ at the endpoints of these intervals, namely the function defined by
\begin{equation}
f_{\ep,L}(x):=M_{\ep,L,k}(x-(k-1)L_{\ep}\omep)+f((k-1)L_{\ep}\omep)
\qquad
\forall x\in I_{\ep,k},
\label{defn:fepL}
\end{equation}
where
\begin{equation}
M_{\ep,L,k}:=\frac{f(kL_{\ep}\omep)-f((k-1)L_{\ep}\omep)}{L_{\ep}\omep}.
\nonumber
\end{equation}

From the $C^{1}$ regularity of $f$ we deduce that
\begin{equation}
|M_{\ep,L,k}|\leq\max\{|f'(x)|:x\in[0,1]\}=:M_{\infty}
\label{est:mepL-bound}
\end{equation}
for every admissible value of $\ep$, $L$ and $k$, and that the family $\{f_{\ep,L}\}$ converges to $f$ in the sense that
\begin{equation}
\lim_{\ep\to 0^{+}}\frac{1}{\omep^{2}}\int_{0}^{1}(f(x)-f_{\ep,L}(x))^{2}\,dx=0.
\label{lim:fepL2f}
\end{equation}

Moreover, we deduce also that $f_{\ep,L}'(x)\to f'(x)$ uniformly in $[0,1]$, and in particular
\begin{equation}
\lim_{\ep\to 0^{+}}L_{\ep}\omep\sum_{k=1}^{N_{\ep,L}}|M_{\ep,L,k}|^{4/5}=
\lim_{\ep\to 0^{+}}\int_{0}^{1}|f_{\ep,L}'(x)|^{4/5}\,dx=
\int_{0}^{1}|f'(x)|^{4/5}\,dx.
\label{lim:fepL'2f'}
\end{equation}

Finally, from the inequality
\begin{equation}
(a+b)^{2}\geq(1-\eta)a^{2}+\left(1-\frac{1}{\eta}\right)b^{2}
\qquad
\forall\eta\in(0,1),
\quad
\forall (a,b)\in\re^{2},
\nonumber
\end{equation}
we obtain the estimate
\begin{equation}
\int_{0}^{1}(u_{\ep}-f)^{2}\,dx\geq
(1-\eta)\int_{0}^{1}(u_{\ep}-f_{\ep,L})^{2}\,dx
+\left(1-\frac{1}{\eta}\right)\int_{0}^{1}(f-f_{\ep,L})^{2}\,dx,
\nonumber
\end{equation}
from which we conclude that
\begin{eqnarray}
\PMF_{\ep}(\beta,f,(0,1),u_{\ep})& \geq & (1-\eta)\PMF_{\ep}(\beta,f_{\ep,L},(0,1),u_{\ep})
\nonumber
\\[0.5ex]
&  & 
\mbox{}+\left(1-\frac{1}{\eta}\right)\beta\int_{0}^{1}(f(x)-f_{\ep,L}(x))^{2}\,dx.
\label{est:SPM-ABGF}
\end{eqnarray}

\paragraph{\textmd{\textit{Reduction to a common interval}}}

We prove that
\begin{equation}
\PMF_{\ep}(\beta,f_{\ep,L},(0,1),u_{\ep})\geq
\omep^{3}\sum_{k=1}^{N_{\ep,L}}\mu_{\ep}(\beta,L,M_{\ep,L,k}),
\label{ineq:ABGF-mLM}
\end{equation}
where $\mu_{\ep}(\beta,L,M_{\ep,L,k})$ is defined by (\ref{defn:muep}). To this end, we begin by observing that
\begin{equation}
\PMF_{\ep}(\beta,f_{\ep,L},(0,1),u_{\ep})=
\sum_{k=1}^{N_{\ep,L}}\PMF_{\ep}(\beta,f_{\ep,L},I_{\ep,k},u_{\ep}).
\label{ineq:ABGF-sum}
\end{equation}

Each of the terms of the sum can be reduced to the common interval $(0,L_{\ep})$ by introducing the function $v_{\ep,L,k}:(0,L_{\ep})\to\re$ defined by
\begin{equation}
v_{\ep,L,k}(y):=\frac{\uep((k-1)L_{\ep}\omep+\omep y)-f((k-1)L_{\ep}\omep)}{\omep}
\qquad
\forall y\in(0,L_{\ep}).
\label{defn:veLK}
\end{equation}

Indeed, with the change of variable $x=(k-1)L_{\ep}\omep+\omep y$, we obtain that
\begin{equation}
\int_{I_{\ep,k}}(\uep(x)-f_{\ep,L}(x))^{2}\,dx=
\omep^{3}\int_{0}^{L_{\ep}}(v_{\ep,L,k}(y)-M_{\ep,L,k}\,y)^{2}\,dy
\nonumber
\end{equation}
and
\begin{equation}
\int_{I_{\ep,k}}
\left\{\ep^{6}\omep^{4}\uep''(x)^{2}+\log\left(1+\uep'(x)^{2}\right)\right\}dx
=\omep^{3}\RPM_{\ep}((0,L_{\ep}),v_{\ep,L,k}),
\nonumber
\end{equation}
and therefore
\begin{eqnarray*}
\PMF_{\ep}(\beta,f_{\ep,L},I_{\ep,k},u_{\ep}) & = &
\omep^{3}\RPMF_{\ep}(\beta,M_{\ep,L,k}\,x,(0,L_{\ep}),v_{\ep,L,k})
\\[1ex]
& \geq & \omep^{3}\mu_{\ep}(\beta,L_{\ep},M_{\ep,L,k})
\\[1ex]
& \geq &  \omep^{3}\mu_{\ep}(\beta,L,M_{\ep,L,k}),
\label{eqn:1000-620}
\end{eqnarray*}
where in the last inequality we exploited that $L_{\ep}\geq L$, and $\mu_{\ep}$ is monotone with respect to the length of the interval. Plugging this inequality into (\ref{ineq:ABGF-sum}) we obtain (\ref{ineq:ABGF-mLM}).

\paragraph{\textmd{\textit{Convergence to minima of the limit problem}}}

There exists $\ep_{0}\in(0,1)$ such that
\begin{equation}
\mu_{\ep}(\beta,L,M_{\ep,L,k})\geq\mu_{0}(\alpha_{0},\beta,L,M_{\ep,L,k})-\eta
\label{th:meLM-mu0}
\end{equation}
for every $\ep\in(0,\ep_{0})$ and every $k\in\{1,\ldots,N_{\ep,L}\}$, where the function $\mu_{0}$ is defined according to (\ref{defn:mu0}), and $\alpha_{0}$ is defined by (\ref{defn:alpha-0}).

Indeed, this estimate follows from Proposition~\ref{prop:mu}, and in particular from the uniform convergence (\ref{th:lim-mu-unif}), after observing that the values of $M_{\ep,L,k}$ are uniformly bounded because of (\ref{est:mepL-bound}).

\paragraph{\textmd{\textit{Conclusion}}}

From the estimate from below in (\ref{th:est-mu-abLM}) we know that
\begin{equation*}
\mu_{0}(\alpha_{0},\beta,L,M_{\ep,L,k})\geq
c_{1}|M_{\ep,L,k}|^{4/5}L-c_{2}|M_{\ep,L,k}|^{1/5},
\end{equation*}
where $c_{1}$ and $c_{2}$ are given by (\ref{defn:c123}), and therefore in particular
\begin{equation}
c_{1}:=\frac{5}{4}\left(\frac{\alpha_{0}^{4}\beta}{3}\right)^{1/5}=
10\left(\frac{2\beta}{27}\right)^{1/5}.
\label{defn:c1}
\end{equation}

Summing over $k$, from (\ref{ineq:ABGF-mLM}) and (\ref{th:meLM-mu0}) we obtain that
\begin{eqnarray*}
\lefteqn{\hspace{-5em}
\frac{\PMF_{\ep}(\beta,f_{\ep,L},(0,1),u_{\ep})}{\omep^{2}} \;\geq\;  
\omep\sum_{k=1}^{N_{\ep,L}}\mu_{\ep}(\beta,L,M_{\ep,L,k})}
\\
\qquad & \geq & \omep\sum_{k=1}^{N_{\ep,L}}\mu_{0}(\alpha_{0},\beta,L,M_{\ep,L,k})-\eta\omep N_{\ep,L}
\\
& \geq & c_{1}L\omep\sum_{k=1}^{N_{\ep,L}}|M_{\ep,L,k}|^{4/5}
-c_{2}\omep N_{\ep,L}M_{\infty}^{1/5}-\eta\omep N_{\ep,L}.
\end{eqnarray*}

Finally, plugging this estimate into (\ref{est:SPM-ABGF}) we deduce that
\begin{eqnarray*}
\frac{\PMF_{\ep}(\beta,f,(0,1),u_{\ep})}{\omep^{2}} & \geq &
(1-\eta)c_{1}\frac{L}{L_{\ep}}\cdot L_{\ep}\omep\sum_{k=1}^{N_{\ep,L}}|M_{\ep,L,k}|^{4/5}
\\[1ex]
& & \mbox{}-\omep N_{\ep,L}\cdot(1-\eta)\left(c_{2}M_{\infty}^{1/5}+\eta\right)
\\[1ex]
& & \mbox{}+\left(1-\frac{1}{\eta}\right)
\frac{\beta}{\omep^{2}}\int_{0}^{1}(f(x)-f_{\ep,L}(x))^{2}\,dx.
\end{eqnarray*}

Now we let $\ep\to 0^{+}$, and we exploit (\ref{lim:fepL'2f'}) in the first line, the fact that $\omep N_{\ep,L}\to 1/L$ in the second line, and (\ref{lim:fepL2f}) in the third line. We conclude that
\begin{equation*}
\liminf_{\ep\to 0^{+}}\frac{\PMF_{\ep}(\beta,f,(0,1),u_{\ep})}{\omep^{2}} \geq 
(1-\eta)\left\{c_{1}\int_{0}^{1}|f'(x)|^{4/5}\,dx
-\frac{c_{2}M_{\infty}^{1/5}+\eta}{L}\right\}.
\end{equation*}

Finally, letting $\eta\to 0^{+}$ and $L\to +\infty$, and recalling that $c_{1}$ is given by (\ref{defn:c1}), we obtain exactly (\ref{est:lim-SPM-below}).


\subsubsection{Estimate from above}

We show the existence of a family $\{u_{\ep}\}\subseteq H^{2}((0,1))$ for which (\ref{est:lim-SPM-above}) holds true. This amounts to proving the asymptotic optimality of all the steps in the proof of the estimate from below.

\paragraph{\textmd{\textit{Interval subdivision and approximation of the forcing term}}}

Let us fix again two real numbers $L>0$ and $\eta\in(0,1)$, and for every $\ep\in(0,1)$ let us define $N_{\ep,L}$ and $L_{\ep}$ as in (\ref{defn:NepAep}), the intervals $I_{\ep,k}$ as in (\ref{defn:Iepk}), and the piecewise affine function $f_{\ep,L}:(0,1)\to\re$ as in (\ref{defn:fepL}). Then we exploit the inequality

\begin{equation*}
(a+b)^{2}\leq(1+\eta)a^{2}+\left(1+\frac{1}{\eta}\right)b^{2}
\qquad
\forall\eta\in(0,1),
\quad
\forall (a,b)\in\re^{2},
\end{equation*}
and for every $u\in H^{2}((0,1))$ we obtain the estimate
\begin{eqnarray*}
\PMF_{\ep}(\beta,f,(0,1),u) & \leq &
(1+\eta)\PMF_{\ep}(\beta,f_{\ep,L},(0,1),u)
\\[1ex]
& & \mbox{}+\left(1+\frac{1}{\eta}\right)\beta\int_{0}^{1}(f(x)-f_{\ep,L}(x))^{2}\,dx.
\label{est:SPM-ABGF-above}
\end{eqnarray*}

\paragraph{\textmd{\textit{Reduction to a common interval}}}

We claim that there exists $u_{\ep}\in H^{2}((0,1))$ such that
\begin{eqnarray*}
\PMF_{\ep}(\beta,f_{\ep,L},(0,1),u_{\ep}) & = &
\omep^{3}\sum_{k=1}^{N_{\ep,L}}\mu_{\ep}^{*}(\beta,L_{\ep},M_{\ep,L,k})
\\
& \leq & 
\left(\frac{L_{\ep}}{L}\right)^{3}\omep^{3}\sum_{k=1}^{N_{\ep,L}}\mu_{\ep}^{*}(\beta,L,M_{\ep,L,k}),
\label{ineq:ABGF-mLM*}
\end{eqnarray*}
where $\mu_{\ep}^{*}$ is defined by (\ref{defn:muep*}), and the inequality follows from (\ref{th:monot-muep*}).

To this end, in analogy with the previous case we observe that the equalities
\begin{eqnarray*}
\PMF_{\ep}(\beta,f_{\ep,L},(0,1),u_{\ep}) & = &
\sum_{k=1}^{N_{\ep,L}}\PMF_{\ep}(\beta,f_{\ep,L},I_{\ep,k},u_{\ep})
\\[1ex]
& = & 
\omep^{3}\sum_{k=1}^{N_{\ep,L}}\RPMF_{\ep}(\beta,M_{\ep,L,k}\,x,(0,L_{\ep}),v_{\ep,L,k})
\end{eqnarray*}
hold true for every $u_{\ep}\in H^{2}((0,1))$, provided that $u_{\ep}(x)$ and $v_{\ep,L,k}(x)$ are related by (\ref{defn:veLK}). At this point it is enough to choose $u_{\ep}$ in such a way that $v_{\ep,L,k}(x)$ coincides with a minimizer in the definition of $\mu_{\ep}^{*}(\beta,L_{\ep},M_{\ep,L,k})$ for every admissible choice of $k$. 

Due to the boundary conditions in (\ref{defn:muep*}), the resulting function $\uep(x)$ coincides with the forcing term $f(x)$ in the nodes of the form $x=kL_{\ep}\omep$, its derivative vanishes in the same points, and the profile in each subinterval is (up to homotheties and translations) a minimizer to (\ref{defn:muep*}). As a consequence, the different pieces glue together in a $C^{1}$ way, and thus the resulting function belongs to $H^{2}((0,1))$. 

\paragraph{\textmd{\textit{Convergence to minima of the limit problem}}}

As in the case of the estimates from below we rely on Proposition~\ref{prop:mu}, and in particular on the uniform convergence (\ref{th:lim-mu*-unif}), in order to deduce that there exists $\ep_{0}\in(0,1)$ such that
\begin{equation*}
\left(\frac{L_{\ep}}{L}\right)^{3}\mu_{\ep}^{*}(\beta,L,M_{\ep,L,k})\leq
\mu_{0}^{*}(\alpha_{0},\beta,L,M_{\ep,L,k})+\eta
\end{equation*}
for every $\ep\in(0,\ep_{0})$ and every $k\in\{1,\ldots,N_{\ep,L}\}$. We can absorb the cubic factor into  $\eta$ because $L_{\ep}\to L$, and $\mu_{\ep}^{*}(\beta,L,M_{\ep,L,k})$ is uniformly bounded for $\ep$ small because of the uniform bound on the slopes $M_{\ep,L,k}$ and the continuity of the limit $\mu_{0}^{*}$ with respect to~$M$.

\paragraph{\textmd{\textit{Conclusion}}}

Now we exploit the estimate from above in (\ref{th:est-mu*-abLM}), and we find that
\begin{equation*}
\mu_{0}^{*}(\alpha_{0},\beta,L,M_{\ep,L,k})\leq
c_{1}|M_{\ep,L,k}|^{4/5}L+c_{3}|M_{\ep,L,k}|^{1/5},
\end{equation*}
where again $c_{1}$ is given by (\ref{defn:c1}), and as in the previous case we conclude that
\begin{eqnarray*}
\frac{\PMF_{\ep}(\beta,f,(0,1),u_{\ep})}{\omep^{2}} & \leq &
(1+\eta)c_{1}L\omep\sum_{k=1}^{N_{\ep,L}}|M_{\ep,L,k}|^{4/5}
\\[1ex]
& & \mbox{}+\omep N_{\ep,L}\cdot(1+\eta)\left(c_{3}M_{\infty}^{1/5}+\eta\right)
\\[1ex]
& & \mbox{}+\left(1+\frac{1}{\eta}\right)
\frac{\beta}{\omep^{2}}\int_{0}^{1}(f(x)-f_{\ep,L}(x))^{2}\,dx.
\end{eqnarray*}

Letting $\ep\to 0^{+}$, we obtain that this family $\{u_{\ep}\}$ satisfies
\begin{equation*}
\limsup_{\ep\to 0^{+}}\frac{\PMF_{\ep}(\beta,f,(0,1),u_{\ep})}{\omep^{2}} \leq 
(1+\eta)\left\{c_{1}\int_{0}^{1}|f'(x)|^{4/5}\,dx
+\frac{c_{3}M_{\infty}^{1/5}+\eta}{L}\right\}.
\end{equation*}

Now we observe that the right-hand side tends to the right-hand side of (\ref{est:lim-SPM-above}) when $\eta\to 0^{+}$ and $L\to +\infty$. Therefore, with a standard diagonal procedure we can find a family $\{u_{\ep}\}\subseteq H^{2}((0,1))$ for which exactly (\ref{est:lim-SPM-above}) holds true.
\qed


\subsection{Blow-ups at standard resolution (Theorem~\ref{thm:BU})}

The proof of Theorem~\ref{thm:BU} consists of three main parts. In the first two parts we address the compactness of fake and true blow-ups. In the final part we show how to achieve all possible translations of the canonical staircase.

\subsubsection{Compactness of fake blow-ups and oblique translations}\label{sec:fake-bu}

Let us set for simplicity $x_{n}:=x_{\ep_{n}}$, and let $w_{n}(y):=w_{\ep_{n}}(y)$ denote the corresponding fake blow-ups, defined in the interval $(A_{n},B_{n})$ with
\begin{equation}
A_{n}:=-\frac{x_{n}}{\omepn},
\qquad\qquad
B_{n}:=\frac{1-x_{n}}{\omepn}.
\label{defn:An-Bn}
\end{equation}

We need to show that the sequence $\{w_{n}(y)\}$ has a subsequence that converges locally strictly in $BV\loc(\re)$ to some oblique translation of the canonical $(H,V)$-staircase. To this end, we introduce the function $g_{n}:(A_{n},B_{n})\to\re$ defined by
\begin{equation}
g_{n}(y):=\frac{f(x_{n}+\omepn y)-f(x_{n})}{\omepn}
\qquad
\forall y\in(A_{n},B_{n}).
\label{defn:gn}
\end{equation}

We are now in a position to apply Proposition~\ref{prop:bu2minloc}. Let us check the assumptions.
\begin{itemize}

\item  Since $x_{n}\to x_{0}\in(0,1)$, passing to the limit in (\ref{defn:An-Bn}) we see that  $A_{n}\to -\infty$ and $B_{n}\to +\infty$.

\item  Since the forcing term $f(x)$ is of class $C^{1}$, passing to the limit in (\ref{defn:gn}) we see that $g_{n}(y)\to f'(x_{0})\cdot y$ uniformly on bounded subsets of $\re$.

\item   With the change of variable $x=x_{n}+\omepn y$ we obtain that
\begin{equation}
\PMF_{\ep_{n}}(\beta,f,(0,1),u_{\ep_{n}})=
\omepn^{3}\cdot\RPMF_{\ep_{n}}(\beta,g_{n},(A_{n},B_{n}),w_{n}).
\label{eqn:PMF-vs-RPMF}
\end{equation}

Since $u_{\ep_{n}}(x)$ is a minimizer of the original functional $u\mapsto\PMF_{\ep_{n}}(\beta,f,(0,1),u)$, it follows that $w_{n}(y)$ is a minimizer of $w\mapsto\RPMF_{\ep_{n}}(\beta,g_{n},(A_{n},B_{n}),w)$.

\item  Due to (\ref{eqn:PMF-vs-RPMF}), estimate (\ref{hp:bound-G}) follows from Theorem~\ref{thm:asympt-min} as soon as $|\log\ep_{n}|\geq 1$.

\end{itemize}

At this point, from Proposition~\ref{prop:bu2minloc} we deduce that the sequence $\{w_{n}(y)\}$ converges locally strictly in $BV\loc(\re)$, at least up to subsequences, to an entire local minimizer of the limit functional (\ref{defn:JF}), with $\alpha$ given by (\ref{defn:alpha-0}). Finally, from Proposition~\ref{prop:loc-min-class} we know that all these entire local minimizers are oblique translations of the canonical $(H,V)$-staircase, with parameters given by (\ref{defn:HV}). 

\begin{rmk}[Back to Remark~\ref{rmk:BU-semi-int}]\label{rmk:boundary-proof}
\begin{em}

Let consider the case where $x_{\ep}\to x_{0}\in\{0,1\}$. If (\ref{hp:BU-internal}) holds true, then again $A_{n}\to -\infty$ and $B_{n}\to +\infty$ for every sequence $\ep_{n}\to 0^{+}$, and hence the previous proof still works. If (\ref{hp:BU-internal}) fails, then when $x_{0}=0$ it may happen that $A_{n}\to A_{\infty}\in(-\infty,0]$ and $B_{n}\to +\infty$ for some sequence $\ep_{n}\to 0^{+}$. 

In this case it is convenient to introduce the translated functions
\begin{equation*}
\widehat{w}_{n}(y):=w_{n}(y+A_{n})-f'(0)A_{n}
\qquad\text{and}\qquad
\widehat{g}_{n}(y):=g_{n}(y+A_{n})-f'(0)A_{n}.
\end{equation*}

We observe that these functions are defined in the interval $(0,L_{n})$ with $L_{n}:=B_{n}-A_{n}$, so that $L_{n}\to +\infty$. We observe also that
\begin{equation*}
\widehat{w}_{n}\in\argmin\Rloc\left\{\RPMF_{\ep_{n}}(\beta,\widehat{g}_{n},(0,L_{n}),v):v\in H^{2}((0,L_{n}))\right\},
\end{equation*}
and that $\widehat{g}_{n}(y)\to f'(0)\cdot y$ uniformly on bounded subsets of $(0,+\infty)$.

This means that we are in the framework of Proposition~\ref{prop:bu2Rminloc}, from which we deduce that the whole sequence $\{\widehat{w}_{n}(y)\}$ converges to the unique semi-entire local minimizer in $(0,+\infty)$ of the limit functional (\ref{defn:JF}), with $\alpha$ given by (\ref{defn:alpha-0}). This semi-entire local minimizer is given by (\ref{defn:semi-entire}), and the convergence is strict in $BV((0,L))$ for every $L>0$ that is not a jump point of the limit. This is a rigorous way of saying that $w_{n}(y)$ converges to $w(y-A_{\infty})+f'(0)A_{\infty}$, and the latter is the oblique translation of the unique semi-entire local minimizer that ``starts in $y=A_{\infty}$''. 

The case where $x_{0}=1$, and for some sequence $\ep_{n}\to 0^{+}$ it happens that $A_{n}\to -\infty$ and $B_{n}\to B_{\infty}\in[0,+\infty)$, is symmetric.

\end{em}
\end{rmk}


\subsubsection{Compactness of true blow-ups and graph translations}\label{sec:true-bu}

Let us define $x_{n}$ and $w_{n}(y)$ as before, and let $v_{n}(y):=v_{\ep_{n}}(y)$ denote the corresponding true blow-ups. We observe that true blow-ups are related to the fake blow-ups by the equality
\begin{equation}
v_{n}(y)=w_{n}(y)-w_{n}(0)
\qquad
\forall y\in(A_{n},B_{n}),
\label{fake-vs-true-bu}
\end{equation}
and therefore the asymptotic behavior of the sequence $\{v_{n}(y)\}$ can be deduced from the asymptotic behavior of the sequence $\{w_{n}(y)\}$. More precisely, let us assume that
\begin{equation*}
w_{n_{k}}(y)\auto S_{H,V}(y-H\tau_{0})+V\tau_{0}
\quad\mbox{in }BV\loc(\re)
\end{equation*}
for some sequence $n_{k}\to+\infty$ and some $\tau_{0}\in[-1,1]$. Then we distinguish two cases. 

\begin{itemize}

\item  Let us assume that $|\tau_{0}|<1$. In this case $y=0$ is not a discontinuity point of the limit of fake blow-ups, and hence the strict convergence implies pointwise convergence (see statement~(\ref{strict:cont}) in Remark~\ref{rmk:strict}), so that
\begin{equation*}
\lim_{k\to +\infty}w_{n_{k}}(0)= 
S_{H,V}(-H\tau_{0})+V\tau_{0}= 
V\tau_{0}.
\end{equation*}

Therefore, from (\ref{fake-vs-true-bu}) we deduce that $v_{n_{k}}(y)\auto S_{H,V}(y-H\tau_{0})$ in $BV\loc(\re)$, and we conclude by observing that the limit is a graph translation of horizontal type of the canonical $(H,V)$-staircase, as required.

\item  Let us assume that $\tau_{0}=\pm 1$, and hence $\tau_{0}=1$ without loss of generality (because  oblique translations corresponding to $\tau_{0}=1$ and $\tau_{0}=-1$ coincide). In this case $y=0$ is a discontinuity point of the limit of fake blow-ups, and hence strict convergence (see statement~(\ref{strict:cont}) in Remark~\ref{rmk:strict}) implies only that
\begin{equation*}
-V\leq\liminf_{k\to +\infty}w_{n_{k}}(0)\leq\limsup_{k\to +\infty}w_{n_{k}}(0)\leq V.
\end{equation*}

As a consequence, up to a further subsequence (not relabeled), $w_{n_{k}}(0)$ tends to some value in $[-V,V]$ that we can always write in the form $V\tau_{1}$ for some real number $\tau_{1}\in[-1,1]$. Therefore, from (\ref{fake-vs-true-bu}) we deduce that, along this further subsequence,
\begin{equation*}
v_{n_{k}}(y)\auto
S_{H,V}(y-H)+V-V\tau_{1}
\quad\mbox{in }BV\loc(\re),
\end{equation*}
and we conclude by observing that the limit is a graph translation of vertical type of the canonical $(H,V)$-staircase, as required.

\end{itemize}


\subsubsection{Realization of all possible oblique/horizontal/vertical translations}

In the constructions we can assume, without loss of generality, that $f'(x_{0})\neq 0$, because otherwise all families of fake or true blow-ups converge to the trivial staircase that is identically~0, in which case there is nothing to prove.

\paragraph{\textmd{\textit{Canonical staircase}}}

We show that there exists a family $x_{\ep}'\to x_{0}$ satisfying (\ref{th:liminf-limsup}), and (\ref{th:BU-conv-fake}) with $w_{0}(y):=S_{H,V}(y)$. The natural idea is to look for the fake blow-ups that minimize some distance from the desired limit. To this end, for every $\ep\in(0,1)$ small enough we consider the function
\begin{equation*}
\psi_{\ep}(x):=
\int_{-H}^{H}\left|\frac{u_{\ep}(x+\omep y)-f(x)}{\omep}\right|\,dy.
\end{equation*}

It is a continuous function of $x$, and therefore it admits at least one minimum point $x_{\ep}'$ in the interval $[x_{\ep}-H\omep,x_{\ep}+H\omep]$. We claim that $\{x_{\ep}'\}$ is the required family. To begin with, we observe that (\ref{th:liminf-limsup}) is automatic from the definition, and we call
\begin{equation}
w_{\ep}(y):=\frac{u_{\ep}(x_{\ep}'+\omep y)-f(x_{\ep}')}{\omep}
\label{defn:wep'}
\end{equation}
the corresponding fake blow-ups. If we assume by contradiction that $\{w_{\ep}(y)\}$ does not converge to $S_{H,V}(x)$, then from the compactness result we know that there exists a sequence $\ep_{n}\to 0^{+}$ such that $w_{\ep_{n}}(y)$ converges locally strictly in $BV\loc(\re)$ to some oblique translation $z_{0}(y)$ of $S_{H,V}(y)$, different from $S_{H,V}(y)$ itself, and in particular
\begin{equation*}
\lim_{n\to +\infty}\psi_{\ep_{n}}(x_{\ep_{n}}')=
\lim_{n\to+\infty}\int_{-H}^{H}|w_{\ep_{n}}(y)|\,dy=
\int_{-H}^{H}|z_{0}(y)|\,dy>0.
\end{equation*}

On the other hand, since $z_{0}$ is an oblique translation, it can be written in the form
\begin{equation*}
z_{0}(y)=S_{H,V}(y-H\tau_{1})+V\tau_{1}
\end{equation*}
for a suitable $\tau_{1}\in[-1,1]$, with $\tau_{1}\neq 0$. Now for every positive integer $n$ we set
\begin{equation*}
x_{\ep_{n}}'':=x_{\ep_{n}}'+(2k_{n}+\tau_{1})H\omepn,
\end{equation*}
where $k_{n}\in\{-1,0,1\}$ is chosen in such a way that 
\begin{equation*}
x_{\ep_{n}}-H\omepn\leq
x_{\ep_{n}}''<
x_{\ep_{n}}+H\omepn
\end{equation*}
(we point out that there is always exactly one possible choice of $k_{n}$).  We claim that 
\begin{equation}
\frac{u_{\ep_{n}}(x_{\ep_{n}}''+\omepn y)-f(x_{\ep_{n}}'')}{\omepn}
\auto S_{H,V}(y)
\qquad
\text{in }BV\loc(\re),
\label{claim:BU2v0}
\end{equation}
and in particular the convergence is also in $L^{1}((-H,H))$. This implies that $\psi_{\ep_{n}}(x_{\ep_{n}}'')\to 0$, and hence $\psi_{\ep_{n}}(x_{\ep_{n}}'')<\psi_{\ep_{n}}(x_{\ep_{n}}')$ when $n$ is large enough, thus contradicting the minimality of $x_{\ep_{n}}'$. In order to prove (\ref{claim:BU2v0}), up to subsequences (not relabeled) we can always assume that $k_{n}$ is actually a constant $k_{\infty}$. Now we observe that
\begin{equation*}
\frac{u_{\ep_{n}}(x_{\ep_{n}}''+\omepn y)-f(x_{\ep_{n}}'')}{\omepn}=
w_{\ep_{n}}(y+(2k_{\infty}+\tau_{1})H)-
\frac{f(x_{\ep_{n}}'')-f(x_{\ep_{n}}')}{\omepn},
\end{equation*}
so that in particular
\begin{equation*}
w_{\ep_{n}}(y+(2k_{\infty}+\tau_{1})H)\auto
z_{0}(y+(2k_{\infty}+\tau_{1})H)=
S_{H,V}(y+2k_{\infty}H)+V\tau_{1},
\end{equation*}
and
\begin{equation*}
\lim_{n\to +\infty}\frac{f(x_{\ep_{n}}'')-f(x_{\ep_{n}}')}{\omepn}=
(2k_{\infty}+\tau_{1})H\cdot f'(x_{0})=
(2k_{\infty}+\tau_{1})V.
\end{equation*}

It follows that
\begin{equation*}
\frac{u_{\ep_{n}}(x_{\ep_{n}}''+\omepn y)-f(x_{\ep_{n}}'')}{\omepn}\auto
S_{H,V}(y+2k_{\infty}H)-2k_{\infty}V,
\end{equation*}
and we conclude by observing that the latter coincides with $S_{H,V}(y)$. This completes the proof of (\ref{claim:BU2v0}).

\paragraph{\textmd{\textit{All oblique translations}}}

Let $x_{\ep}'\to x_{0}$ be the family that we found in the previous paragraph, namely a family satisfying (\ref{th:liminf-limsup}), and (\ref{th:BU-conv-fake}) with $w_{0}(y):=S_{H,V}(y)$. If we need to obtain a different oblique translation of the form $w_{0}(y)=S_{H,V}(y-H\tau_{0})+V\tau_{0}$ for some $\tau_{0}\in(-1,1]$, then it is enough to consider the family
\begin{equation*}
x_{\ep}'':=x_{\ep}'-H\tau_{0}\omep+2k_{\ep}H\omep,
\end{equation*}
where $k_{\ep}\in\{-1,0,1\}$ is chosen in such a way that $x_{\ep}''\in[x_{\ep}-H\omep,x_{\ep}+H\omep]$. Indeed, it is enough to observe that
\begin{equation}
\frac{u_{\ep}(x_{\ep}''+\omep y)-f(x_{\ep}'')}{\omep}=
w_{\ep}(y+(2k_{\ep}-\tau_{0})H)-
\frac{f(x_{\ep}'')-f(x_{\ep}')}{\omep},
\label{eqn:fbu-''-'}
\end{equation}
where $w_{\ep}$ is the family of fake blow-ups with centers in $x_{\ep}'$. At this point, if needed we split the family into three subfamilies according to the value of $k_{\ep}$. In the subfamily where $k_{\ep}$ is equal to some constant $k_{0}$ we obtain that
\begin{equation*}
w_{\ep}(y+(2k_{\ep}-\tau_{0})H)\auto
S_{H,V}(y+(2k_{0}-\tau_{0})H),
\end{equation*}
and
\begin{equation*}
\frac{f(x_{\ep}'')-f(x_{\ep}')}{\omep}\to
(2k_{0}-\tau_{0})H\cdot f'(x_{0})=(2k_{0}-\tau_{0})V.
\end{equation*}

This implies that the left-hand side of (\ref{eqn:fbu-''-'}) converges locally strictly to
\begin{equation*}
S_{H,V}(y+(2k_{0}-\tau_{0})H)-(2k_{0}-\tau_{0})V,
\end{equation*}
which is equal to $S_{H,V}(y-H\tau_{0})+V\tau_{0}$, independently of $k_{0}$, as required.

\paragraph{\textmd{\textit{Graph translations of horizontal type}}}

In this paragraph we show that any graph translation of the form $S_{H,V}(y-H\tau_{0})$, with $\tau_{0}\in[-1,1]$, can be obtained as the limit of a suitable family of true blow-ups whose centers satisfy (\ref{th:liminf-limsup}). 

To begin with, we observe that the set of possible limits is closed with respect to the locally strict convergence in $BV\loc(\re)$, and therefore it is enough to obtain all limits with $\tau_{0}$ in the open interval $(-1,1)$. In this case, we claim that we can take the same family $x_{\ep}'\to x_{0}$ whose fake blow-ups converge to $S_{H,V}(y-H\tau_{0})-V\tau_{0}$, with the same value of $\tau_{0}$. Indeed, we observe again that
\begin{equation}
\frac{\uep(x_{\ep}'+\omep y)-\uep(x_{\ep}')}{\omep}=
w_{\ep}(y)-w_{\ep}(0),
\label{eqn:fake-true}
\end{equation}
where $w_{\ep}(y)$ is defined by (\ref{defn:wep'}). Now we know that $w_{\ep}(y)\auto S_{H,V}(y-H\tau_{0})-V\tau_{0}$. Moreover, if $\tau_{0}\in(-1,1)$ the limit function is continuous in $y=0$, and therefore the strict convergence implies also that $w_{\ep}(0)\to S_{H,V}(-H\tau_{0})-V\tau_{0}=-V\tau_{0}$. Plugging these two result into (\ref{eqn:fake-true}) we obtain that the left-hand side converges to $S_{H,V}(y-H\tau_{0})$, as required.

\paragraph{\textmd{\textit{Graph translations of vertical type}}}

In this final paragraph we show that any graph translation of the form $S_{H,V}(y-H)+(1-\tau_{0})V$, with $\tau_{0}\in[-1,1]$, can be obtained as the limit of a suitable family of true blow-ups whose centers satisfy (\ref{th:liminf-limsup}). To begin with, as in the case of graph translations of horizontal type we reduce ourselves to the case where $\tau_{0}\in(-1,1)$.  

In this case we consider the family $x_{\ep}'\to x_{0}$ whose fake blow-ups $w_{\ep}(y)$ defined by (\ref{defn:wep'}) converge to $S_{H,V}(y)$. Since $S_{H,V}(y)$ is continuous in $y=-2H$ and $y=2H$, the strict convergence implies in particular that
\begin{equation*}
\lim_{\ep\to 0^{+}}w_{\ep}(-2H)=-2V
\qquad\quad\text{and}\quad\qquad
\lim_{\ep\to 0^{+}}w_{\ep}(2H)=2V.
\end{equation*}

Recalling that $w_{\ep}(y)$ is continuous in $y$ and vanishes for $y=0$, this means that when $\ep\in(0,1)$ is small enough there exist $a_{\ep}\in(-2H,0)$ and $b_{\ep}\in(0,2H)$ such that
\begin{equation}
w_{\ep}(a_{\ep})=(-1+\tau_{0})V
\qquad\quad\text{and}\quad\qquad
w_{\ep}(b_{\ep})=(1+\tau_{0})V.
\label{defn:aep-bep}
\end{equation}

These two conditions imply in particular that
\begin{equation}
\lim_{\ep\to 0^{+}}a_{\ep}=-H
\qquad\quad\text{and}\quad\qquad
\lim_{\ep\to 0^{+}}b_{\ep}=H,
\label{lim:aep-bep}
\end{equation}
because a limit of blow-ups can be different from an integer multiple of $2V$ only in the jump points of the limit function $S_{H,V}(y)$.

Now we set
\begin{equation*}
x_{\ep}'':=\begin{cases}
x_{\ep}'+\omep a_{\ep}\quad  & \text{if }x_{\ep}'\geq x_{\ep}, \\
x_{\ep}'+\omep b_{\ep}  & \text{if }x_{\ep}'< x_{\ep},
\end{cases}
\end{equation*}
and we claim that this is the required family. Indeed, from the definition it follows that
\begin{equation*}
\frac{|x_{\ep}''-x_{\ep}|}{\omep}\leq
\max\left\{\frac{|x_{\ep}'-x_{\ep}|}{\omep},-a_{\ep},b_{\ep}\right\},
\end{equation*}
and therefore (\ref{th:liminf-limsup}) for $x_{\ep}''$ follows from (\ref{th:liminf-limsup}) for $x_{\ep}'$ and (\ref{lim:aep-bep}). 

In order to compute the limit of the true blow-ups with center in $x_{\ep}''$, we consider the two subfamilies where $x_{\ep}''$ is defined using $a_{\ep}$ or $b_{\ep}$. In the first case from (\ref{defn:aep-bep}) and (\ref{lim:aep-bep}) we deduce that
\begin{eqnarray*}
\frac{\uep(x_{\ep}''+\omep y)-\uep(x_{\ep}'')}{\omep} & = & 
w_{\ep}(y+a_{\ep})-w_{\ep}(a_{\ep}),
\\
& \auto &
S_{H,V}(y-H)-(-1+\tau_{0})V,
\end{eqnarray*}
as required. Analogously, in the second case we obtain that
\begin{equation*}
\frac{\uep(x_{\ep}''+\omep y)-\uep(x_{\ep}'')}{\omep}\auto S_{H,V}(y+H)-(1+\tau_{0})V,
\end{equation*}
which again coincides with $S_{H,V}(y-H)+(1-\tau_{0})V$, as required.
\qed


\subsection{Convergence of minimizers to the forcing term}

\subsubsection{Strict convergence (statement~(1) of Theorem~\ref{thm:varifold})}

Since the limit $f(x)$ is continuous, we know that uniform convergence in $[0,1]$ follows from strict convergence (see statement~(\ref{strict:cont}) in Remark~\ref{rmk:strict}). As for strict convergence, we already know from Proposition~\ref{prop:basic} that $u_{\ep}\to f$ in $L^{2}((0,1))$. Therefore, it remains to show that (the opposite inequality is trivial)
\begin{equation}
\limsup_{\ep\to 0^{+}}\int_{0}^{1}|u_{\ep}'(x)|\,dx\leq
\int_{0}^{1}|f'(x)|\,dx.
\label{th:limsup-strict}
\end{equation}

Let us assume by contradiction that this is not the case, and hence there exist a positive real number $\eta_{0}$ and a sequence $\{\ep_{n}\}\subseteq(0,1)$ such that $\ep_{n}\to 0^{+}$ and
\begin{equation}
\int_{0}^{1}|u_{\ep_{n}}'(x)|\,dx\geq
\int_{0}^{1}|f'(x)|\,dx+\eta_{0}
\qquad
\forall n\in\n.
\label{defn:eta-0}
\end{equation}

For every fixed positive real number $L$, in analogy with (\ref{defn:NepAep}) we set
\begin{equation*}
N_{n}:=\left\lfloor\frac{1}{L\omepn}\right\rfloor
\qquad\quad\text{and}\quad\qquad
L_{n}:=\frac{1}{N_{n}\omepn},
\end{equation*}
and we consider the intervals $I_{n,k}:=((k-1)L_{n}\omepn,kL_{n}\omepn)$ with $k\in\{1,\ldots,N_{n}\}$. Since we can rewrite (\ref{defn:eta-0}) in the form
\begin{equation*}
\sum_{k=1}^{N_{n}}\int_{I_{n,k}}\left(|u_{\ep_{n}}'(x)|-|f'(x)|\right)\,dx\geq\eta_{0},
\end{equation*}
we deduce that for every $n\in\n$ there exists an integer $k_{n}$ such that
\begin{equation}
\int_{I_{n,k_{n}}}\left(|u_{\ep_{n}}'(x)|-|f'(x)|\right)\,dx\geq\frac{\eta_{0}}{N_{n}}.
\label{defn:kn}
\end{equation}

Now we set $x_{n}:=(k_{n}-1)\omepn L_{n}$, and we consider the corresponding fake blow-ups 
\begin{equation}
w_{n}(y):=\frac{u_{\ep_{n}}(x_{n}+\omepn y)-f(x_{n})}{\omepn}.
\label{defn:FBU-strict}
\end{equation}

With the change of variable $x=x_{n}+\omep y$, we can rewrite (\ref{defn:kn}) in the form
\begin{equation}
\int_{0}^{L_{n}}\left(|w_{n}'(y)|-|f'(x_{n}+\omepn y)|\strut\right)\,dy\geq
\frac{\eta_{0}}{N_{n}\omepn}\geq
\eta_{0}L.
\label{defn:kn-bis}
\end{equation}

Up to subsequences (not relabeled) we can always assume that $x_{n}$ converges to some $x_{\infty}\in[0,1]$. Let us assume for a while that $x_{\infty}\in(0,1)$. From the continuity of $f'(x)$ we deduce that $f'(x_{n}+\omepn y)\to f'(x_{\infty})$ uniformly on bounded subsets of $\re$ and in particular, since $L_{n}\to L$, we obtain that
\begin{equation}
\lim_{n\to+\infty}\int_{0}^{L_{n}}|f'(x_{n}+\omepn y)|\,dy=|f'(x_{\infty})|L.
\label{est:lim-f'}
\end{equation}

Moreover, from statement~(1) of Theorem~\ref{thm:BU} we deduce that, up to a further subsequence (not relabeled), $w_{n}\auto w_{\infty}$ in $BV\loc(\re)$, where $w_{\infty}$ is an oblique translation of a canonical staircase with parameters depending on $\beta$ and $f'(x_{\infty})$. As a consequence, from statement~(\ref{strict:sub-int}) of Remark~\ref{rmk:strict} we obtain that 
\begin{equation*}
\limsup_{n\to +\infty}\int_{0}^{L_{n}}|w_{n}'(y)|\,dy\leq
\lim_{n\to +\infty}\int_{a}^{b}|w_{n}'(y)|\,dy=
|Dw_{\infty}|((a,b))
\end{equation*}
for every interval $(a,b)\supseteq[0,L]$ whose endpoints $a$ and $b$ are not jump points of $w_{\infty}$. If we consider any sequence of such intervals whose intersection is $[0,L]$, we deduce that
\begin{equation}
\limsup_{n\to +\infty}\int_{0}^{L_{n}}|w_{n}'(y)|\,dy\leq
|Dw_{\infty}|([0,L]).
\label{est:lim-wn'}
\end{equation}

From (\ref{defn:kn-bis}), (\ref{est:lim-f'}) and (\ref{est:lim-wn'}) we conclude that
\begin{equation}
|Dw_{\infty}|([0,L])-|f'(x_{\infty})|L\geq\eta_{0}L.
\label{est:L-V}
\end{equation}

Now we observe that the left-hand side is the difference between the total variation of $w_{\infty}$ in $[0,L]$ and the total variation of the line $y\mapsto |f'(x_{\infty})|y$ in the same interval. Since $w_{\infty}(y)$ is a staircase with the property that the midpoints of the vertical parts of the steps lie on the same line, the left-hand side of (\ref{est:L-V}) is bounded from above by the height of each step of the staircase. Now both $w_{\infty}$ and $x_{\infty}$ might depend on $L$, but in any case the height of the steps depends only on $\beta$ and $|f'(x_{\infty})|$, and the latter is bounded independently on $L$ because $f$ is of class $C^{1}$. In conclusion, the left-hand side of (\ref{est:L-V}) is bounded from above independently of $L$, and this contradicts (\ref{est:L-V}) when $L$ is large enough.

Let us consider now the case where $x_{\infty}=0$ (the case $x_{\infty}=1$ is symmetric). In this case we consider the sequence $\{k_{n}\}$. If it is unbounded, then up to subsequences we can assume that it diverges to $+\infty$. In this case the intervals where the functions $w_{n}(y)$ of (\ref{defn:FBU-strict}) are defined invade eventually the whole real line, and therefore the previous argument works without any change (see also Remark~\ref{rmk:BU-semi-int}).

If the sequence $\{k_{n}\}$ is bounded, then up to subsequences we can assume that it is equal to some fixed positive integer $k_{\infty}$. In this case the functions $w_{n}(y)$ are all defined in the same half-line $y>y_{\infty}$ with $y_{\infty}:=-(k_{\infty}-1)L$, and in this half-line they converge to a limit staircase $w_{\infty}(y)$ (see Remark~\ref{rmk:boundary-proof}). The convergence is strict in every interval of the form $(y_{\infty},b)$, where $b$ is not a jump point of $w_{\infty}(y)$, and of course also in all intervals of the form $(a,b)$ where $a$ and $b$ are not jump points of $w_{\infty}(y)$. Moreover, the function $w_{\infty}$ is the unique semi-entire right-hand local minimizer of $\JF_{1/2}$ with the appropriate parameters in this half-line, namely the suitable oblique translation of the function defined in (\ref{defn:semi-entire}), which is again a staircase with the property that the midpoints of the vertical parts of the steps lie on the line $f'(x_{\infty})y$. 

At the end of the day, we obtain that (\ref{est:L-V}) holds true also in this case, and as before the right-hand side depends only on the difference between the ``values'' of $w_{\infty}(y)$ and of the line $f'(x_{\infty})y$ at the two endpoints. This difference is bounded from above by the height of the steps of $w_{\infty}$. These steps could be either the ``ordinary steps'' or the ``initial step'', which is higher, but in any case their height is independent of $L$.


\subsubsection{Varifold convergence (statement~(2) of Theorem~\ref{thm:varifold})}

\paragraph{\textmd{\textit{Notations and splitting of the graph}}}

In analogy with (\ref{defn:V0+-}), for every $\ep\in(0,1)$ we set
\begin{equation*}
V_{\ep}^{+}:=\{x\in[0,1]:u_{\ep}'(x)>0\}
\qquad\text{and}\qquad
V_{\ep}^{-}:=\{x\in[0,1]:u_{\ep}'(x)<0\}.
\end{equation*}

From statement~(\ref{th:conv*+-}) of Remark~\ref{rmk:strict} we know that the strict convergence of $u_{\ep}(x)$ to $f(x)$ implies in particular that 
\begin{equation}
\lim_{\ep\to 0^{+}}\int_{V_{\ep}^{+}}g(x)u_{\ep}'(x)\,dx=
\int_{V_{0}^{+}}g(x)f'(x)\,dx
\label{th:conv-measure}
\end{equation}
for every continuous function $g:[0,1]\to\re$, and similarly with $V_{\ep}^{-}$ and $V_{0}^{-}$.

We observe also that the strict convergence $u_{\ep}\auto f$ in $BV((0,1))$ implies that the family $\{u_{\ep}(x)\}$ is bounded in $L^{\infty}((0,1))$, and therefore there exist real numbers $\ep_{0}\in(0,1)$ and $M_{0}\geq 0$ such that
\begin{equation}
|\phi(x,u_{\ep}(x),\arctan p)|\leq M_{0}
\qquad
\forall(x,p)\in[0,1]\times\re
\quad
\forall\ep\in(0,\ep_{0}).
\label{est:bound-phi}
\end{equation}

Now for every $a\in(0,1)$ we define the three sets
\begin{gather*}
I_{a}:=\left\{(x,s,p)\in[0,1]\times\re\times\re:|s-f(x)|\leq a,\ |p|\leq a\right\},
\\[1ex]
I_{a}^{+}:=\left\{(x,s,p)\in[0,1]\times\re\times\re:|s-f(x)|\leq a,\ p\geq 1/a\right\},
\\[1ex]
I_{a}^{-}:=\left\{(x,s,p)\in[0,1]\times\re\times\re:|s-f(x)|\leq a,\ p\leq-1/a\right\},
\end{gather*}
and the corresponding three constants
\begin{gather*}
\Gamma_{a}:=\max\left\{|\phi(x,s,\arctan p)-\phi(x,f(x),0)|:(x,s,p)\in I_{a}\right\},
\\[1ex]
\Gamma_{a}^{+}:=\max\left\{|\phi(x,s,\arctan p)-\phi(x,f(x),\pi/2)|:(x,s,p)\in I_{a}^{+}\right\},
\\[1ex]
\Gamma_{a}^{-}:=\max\left\{|\phi(x,s,\arctan p)-\phi(x,f(x),-\pi/2)|:(x,s,p)\in I_{a}^{-}\right\}.
\end{gather*}

We observe that, due to the boundedness of $f(x)$ and the uniform continuity of $\phi$ in bounded sets, these constants satisfy
\begin{equation}
\lim_{a\to 0^{+}}\Gamma_{a}=\lim_{a\to 0^{+}}\Gamma_{a}^{+}=\lim_{a\to 0^{+}}\Gamma_{a}^{-}=0.
\label{lim-gamma-vari}
\end{equation} 

Finally, for every $\ep\in(0,1)$ and every $a\in(0,1)$, we write the interval $[0,1]$ as the disjoint union of the four sets
\begin{gather}
H_{a,\ep}:=\left\{x\in[0,1]:|u_{\ep}'(x)|\leq a\right\},
\label{defn:H-a-ep}
\\[1ex]
V_{a,\ep}^{+}:=\left\{x\in[0,1]:u_{\ep}'(x)\geq 1/a\right\},
\qquad
V_{a,\ep}^{-}:=\left\{x\in[0,1]:u_{\ep}'(x)\leq -1/a\right\},
\label{defn:V-a-ep}
\\[1ex]
M_{a,\ep}:=\left\{x\in[0,1]:a<|u_{\ep}'(x)|<1/a\right\},
\label{defn:M-a-ep}
\end{gather}
and accordingly we write
\begin{equation*}
\int_{0}^{1}
\phi\left(x,u_{\ep}(x),\arctan(u_{\ep}'(x))\strut\right)\sqrt{1+u_{\ep}'(x)^{2}}\,dx=
I_{a,\ep}^{H}+I_{a,\ep}^{+}+I_{a,\ep}^{-}+I_{a,\ep}^{M},
\end{equation*}
where the four terms in the right-hand side are the integrals over the four sets defined above. We observe that
\begin{eqnarray*}
\PMF_{\ep}(\beta,f,(0,1),u_{\ep}) & \geq & 
\int_{0}^{1}\log\left(1+u_{\ep}'(x)^{2}\right)\,dx
\\[0.5ex]
& \geq & \log\left(1+a^{2}\right)\left(|V_{a,\ep}^{+}|+|V_{a,\ep}^{-}|+|M_{a,\ep}|\right),
\end{eqnarray*}
and, since the left-hand side tends to~0, we deduce that
\begin{equation*}
\lim_{\ep\to 0^{+}}|V_{a,\ep}^{+}|=
\lim_{\ep\to 0^{+}}|V_{a,\ep}^{-}|
=\lim_{\ep\to 0^{+}}|M_{a,\ep}|=0
\qquad
\forall a\in(0,1),
\end{equation*}
and as a consequence
\begin{equation*}
\lim_{\ep\to 0^{+}}|H_{a,\ep}^{+}|=1
\qquad
\forall a\in(0,1).
\end{equation*}

We claim that for every fixed $a\in(0,1)$ it turns out that
\begin{gather}
\limsup_{\ep\to 0^{+}}\left|
I_{a,\ep}^{H}-\int_{0}^{1}\phi(x,f(x),0)\,dx
\right|\leq M_{0}\left(\sqrt{1+a^{2}}-1\right)+\Gamma_{a},
\label{th:limsup-H}
\\[1ex]
\lim_{\ep\to 0^{+}}I_{a,\ep}^{M}=0,
\label{th:limsup-M}
\\[1ex]
\limsup_{\ep\to 0^{+}}\left|
I_{a,\ep}^{+}-\int_{V_{0}^{+}}\phi(x,f(x),\pi/2)\cdot f'(x)\,dx
\right|\leq\Gamma_{a}^{+}\int_{0}^{1}|f'(x)|\,dx+M_{0}a,
\label{th:limsup-V+}
\\[1ex]
\limsup_{\ep\to 0^{+}}\left|
I_{a,\ep}^{-}-\int_{V_{0}^{-}}\phi(x,f(x),-\pi/2)\cdot |f'(x)|\,dx
\right|\leq\Gamma_{a}^{-}\int_{0}^{1}|f'(x)|\,dx+M_{0}a.
\label{th:limsup-V-}
\end{gather}

If we prove these claims, then we let $a\to 0^{+}$ and from (\ref{lim-gamma-vari}) we obtain exactly (\ref{th:varifold}). 

In words, this means that the integral in the left-hand side of (\ref{th:varifold}) splits into the four integrals over the regions (\ref{defn:H-a-ep}), (\ref{defn:V-a-ep}), (\ref{defn:M-a-ep}), which behave as follows. 
\begin{itemize}

\item The integral over the ``intermediate'' region $M_{a,\ep}$ disappears in the limit.

\item  The integral over the ``horizontal'' region $H_{a,\ep}$ tends to the first integral in the right hand side of (\ref{th:varifold}), in which the ``tangent component'' is horizontal.

\item  The integrals over the two ``vertical'' regions $V^{+}_{a,\ep}$ and $V^{-}_{a,\ep}$ tend to the two integrals over $V_{0}^{+}$ and $V_{0}^{-}$ in the right hand side of (\ref{th:varifold}). In this two integrals the ``tangent component'' is vertical.

\end{itemize}  

\paragraph{\textmd{\textit{Estimate in the intermediate regime}}}

From (\ref{est:bound-phi}) we know that
\begin{equation*}
\left|\phi(x,u_{\ep}(x),\arctan(u_{\ep}'(x)))\right|\sqrt{1+u_{\ep}'(x)^{2}}\leq
M_{0}\sqrt{1+\frac{1}{a^{2}}}
\qquad
\forall x\in M_{a,\ep},
\end{equation*}
and therefore
\begin{equation*}
|I_{\ep,a}^{M}|\leq M_{0}\sqrt{1+\frac{1}{a^{2}}}\cdot|M_{a,\ep}|.
\end{equation*}

Since $|M_{a,\ep}|\to 0$ as $\ep\to 0^{+}$, this proves (\ref{th:limsup-M}).

\paragraph{\textmd{\textit{Estimate in the horizontal regime}}}

In order to prove (\ref{th:limsup-H}), we observe that
\begin{eqnarray*}
I_{a,\ep}^{H}-\int_{0}^{1}\phi(x,f(x),0)\,dx & = & 
\int_{H_{a,\ep}}\phi\left(x,u_{\ep}(x),\arctan(u_{\ep}'(x))\strut\right)\left(\sqrt{1+u_{\ep}'(x)^{2}}-1\right)\,dx
\\
& & +\int_{H_{a,\ep}}\left\{\phi\left(x,u_{\ep}(x),\arctan(u_{\ep}'(x))\strut\right)-\phi(x,f(x),0)\right\}\,dx,
\\
& & +\int_{H_{a,\ep}}\phi(x,f(x),0)\,dx-
\int_{0}^{1}\phi(x,f(x),0)\,dx.
\end{eqnarray*}

The absolute value of the first line in the right-hand side is less than or equal to $M_{0}\left(\sqrt{1+a^{2}}-1\right)$. The absolute value of the second line is less than or equal to $\Gamma_{a}$ provided that
\begin{equation}
|u_{\ep}(x)-f(x)|\leq a
\qquad
\forall x\in[0,1],
\label{est:uep-f}
\end{equation}
and this happens whenever $\ep$ is small enough. The third line tends to~0 because $|H_{\ep,a}|\to 1$ as $\ep\to 0^{+}$. This is enough to establish (\ref{th:limsup-H}).

\paragraph{\textmd{\textit{Estimate in the vertical regime}}}

In order to prove (\ref{th:limsup-V+}), we observe that
\begin{eqnarray*}
\lefteqn{\hspace{-3em}
I_{a,\ep}^{+}-\int_{V_{0}^{+}}\phi(x,f(x),\pi/2)\cdot f'(x)\,dx =
}
\\  
\qquad\qquad & = & \int_{V_{a,\ep}^{+}}\phi\left(x,u_{\ep}(x),\arctan(u_{\ep}'(x))\strut\right)\left(\sqrt{1+u_{\ep}'(x)^{2}}-u_{\ep}'(x)\right)\,dx
\\[0.5ex]
& & +\int_{V_{a,\ep}^{+}}\left\{\phi\left(x,u_{\ep}(x),\arctan(u_{\ep}'(x))\strut\right)-\phi(x,f(x),\pi/2)\right\}u_{\ep}'(x)\,dx,
\\[0.5ex]
& & +\int_{V_{a,\ep}^{+}}\phi(x,f(x),\pi/2)u_{\ep}'(x)\,dx
-\int_{V_{\ep}^{+}}\phi(x,f(x),\pi/2)u_{\ep}'(x)\,dx
\\[0.5ex]
& & +\int_{V_{\ep}^{+}}\phi(x,f(x),\pi/2)u_{\ep}'(x)\,dx
-\int_{V_{0}^{+}}\phi(x,f(x),\pi/2)f'(x)\,dx
\\[0.5ex]
& =: & L_{1}+L_{2}+L_{3}+L_{4}.
\end{eqnarray*}

Let us consider the four lines separately. The first line can be estimated as
\begin{equation*}
|L_{1}|\leq M_{0}\max\left\{\sqrt{1+p^{2}}-p:p\geq 1/a\right\}|V_{a,\ep}^{+}|\leq
M_{0}\cdot\frac{a}{2}\cdot|V_{a,\ep}^{+}|,
\end{equation*}
and this tends to~0 when $\ep\to 0^{+}$. The second line can be estimated as
\begin{equation*}
|L_{2}|\leq\Gamma_{a}^{+}\cdot\int_{0}^{1}|u_{\ep}'(x)|\,dx
\end{equation*}
whenever (\ref{est:uep-f}) holds true, namely when $\ep$ is small enough. For the third line we observe that $V_{\ep}^{+}\setminus V_{a,\ep}^{+}\subseteq H_{a,\ep}\cup M_{a,\ep}$, and therefore
\begin{eqnarray*}
|L_{3}| & \leq  &
\int_{H_{a,\ep}}|\phi(x,f(x),0)|\cdot|u_{\ep}'(x)|\,dx+
\int_{M_{a,\ep}}|\phi(x,f(x),0)|\cdot|u_{\ep}'(x)|\,dx
\\
& \leq &
M_{0}a+M_{0}\cdot\frac{1}{a}\cdot|M_{a,\ep}|. 
\end{eqnarray*}

Finally, we observe that $L_{4}\to 0$ as $\ep\to 0^{+}$ because of (\ref{th:conv-measure}).  Recalling (\ref{th:limsup-strict}) and the fact that $|M_{a,\ep}|\to 0$ as $\ep\to 0^{+}$, from the previous estimates we conclude that
\begin{eqnarray*}
\limsup_{\ep\to 0^{+}}|L_{1}+L_{2}+L_{3}+L_{4}| & \leq &
\limsup_{\ep\to 0^{+}}\Gamma_{a}^{+}\cdot\int_{0}^{1}|u_{\ep}'(x)|\,dx+
M_{0}a
\\
& = &
\Gamma_{a}^{+}\int_{0}^{1}|f'(x)|\,dx+M_{0}a,
\end{eqnarray*}
which proves (\ref{th:limsup-V+}). The proof of (\ref{th:limsup-V-}) is analogous.
\qed


\subsection{Low resolution blow-ups (Corollary~\ref{thm:spm-bu-low})}

The pointwise convergence for $y=0$ is trivial, and therefore it is enough to check the convergence of total variations, which in turn reduces to
\begin{equation}
\limsup_{\ep\to 0^{+}}\int_{a}^{b}|u_{\ep}'(x_{\ep}+\alpha_{\ep}y)|\,dy\leq
|f'(x_{0})|(b-a)
\label{th:limsup-TV-alpha}
\end{equation}
for every interval $(a,b)\subseteq\re$. If we assume by contradiction that (\ref{th:limsup-TV-alpha}) fails, than the same argument we exploited in the proof of (\ref{th:limsup-strict}) shows that there exist a positive real number $\eta_{0}$, a sequence $\{\ep_{n}\}\subseteq(0,1)$ such that $\ep_{n}\to 0^{+}$, and a sequence of positive integers $k_{n}$ such that
\begin{equation}
\int_{0}^{L_{n}}|u_{\ep_{n}}'(\widehat{x}_{n}+\omepn y)|\,dy-
|f'(x_{0})|L_{n}\geq
\eta_{0}\cdot\frac{\alpha_{\ep_{n}}}{\omepn N_{n}},
\label{est:TV-alpha}
\end{equation}
where now
\begin{equation*}
N_{n}:=\left\lfloor\frac{(b-a)\alpha_{\ep_{n}}}{L\omepn}\right\rfloor,
\qquad
L_{n}:=\frac{(b-a)\alpha_{\ep_{n}}}{N_{n}\omepn},
\qquad
\widehat{x}_{n}:=x_{\ep_{n}}+a\alpha_{\ep_{n}}+L_{n}\omepn(k_{n}-1),
\end{equation*}
and $k_{n}\in\{1,\ldots,N_{n}\}$. The integral in the left-hand side of (\ref{est:TV-alpha}) coincides with the total variation in the interval $(0,L_{n})$ of the fake blow-up of $u_{\ep_{n}}(x)$, at the standard scale $\omepn$, with center in $\widehat{x}_{n}$. Since $\widehat{x}_{n}\to x_{0}$ (here we exploit again that $k_{n}\leq N_{n}$ and $\omep/\alpha_{\ep}\to 0$), we know that these fake blow-ups converge strictly (up to subsequences) to some staircase $w_{\infty}$. Therefore, passing to the limit in (\ref{est:TV-alpha}) we deduce that
\begin{equation*}
|Dw_{\infty}|([0,L])-|f'(x_{0})|L\geq\frac{\eta_{0}L}{b-a},
\end{equation*}
and we conclude exactly as in the proof of (\ref{th:limsup-strict}).
\qed


\setcounter{equation}{0}
\section{Asymptotic analysis of local minimizers}\label{sec:loc-min-proof}

This section is the technical core of the paper. Here we prove all the results that we stated in section~\ref{sec:loc-min}.

\subsection{Preliminary lemmata}

\begin{lemma}\label{lemma:equipartition}

Let $C_{0}$ and $C_{1}$ be two positive real numbers. Let us consider the function $\varphi:(0,1)\to\re$ defined by
\begin{equation*}
\varphi(t):=C_{0}\left(\sqrt{t}+\sqrt{1-t}\right)+
C_{1}\left(t^{3}+(1-t)^{3}\right),
\end{equation*}
and let us assume that there exists $t_{0}\in(0,1)$ such that $\varphi(t)\geq\varphi(t_{0})$ for every $t\in(0,1)$.

Then it turns out that $t_{0}=1/2$.

\end{lemma}

\begin{proof}

With the variable change $t=\sin^{2}\theta$, we can restate the claim as follows. Let us consider the function $g:(0,\pi/2)\to\re$ defined by
\begin{equation*}
g(\theta):=C_{0}\left(\cos\theta+\sin\theta\right)+
C_{1}\left(\cos^{6}\theta+\sin^{6}\theta\right);
\end{equation*}
if there exists $\theta_{0}\in(0,\pi/2)$ such that
\begin{equation}
g(\theta)\geq g(\theta_{0})
\qquad
\forall\theta\in(0,\pi/2),
\label{hp:g-min-loc}
\end{equation}
then necessarily $\theta_{0}=\pi/4$. 

In order to prove this claim, we observe that the derivative of $g(\theta)$ is
\begin{equation}
g'(\theta)=(\cos\theta-\sin\theta)
\left(C_{0}-6C_{1}\cos\theta\sin\theta(\cos\theta+\sin\theta)\right).
\label{eqn:g'}
\end{equation}

Let us consider the function $\psi(\theta):=\cos\theta\sin\theta(\cos\theta+\sin\theta)$, whose derivative is 
\begin{equation*}
\psi'(\theta)=(\cos\theta-\sin\theta)(1+3\cos\theta\sin\theta).
\end{equation*}

It follows that $\psi(\theta)$ is increasing in $[0,\pi/4]$ and decreasing in $[\pi/4,\pi/2]$, and its maximum values is $\psi(\pi/4)=1/\sqrt{2}$. Now we distinguish two cases.
\begin{itemize}

\item  If $C_{0}\sqrt{2}\geq 6C_{1}$, then the sign of $g'(\theta)$ coincides with the sign of $\cos\theta-\sin\theta$. It follows that $\pi/4$ is the unique stationary point of $g(\theta)$ in $(0,\pi/4)$, but it is a maximum point, and therefore in this case there is no $\theta_{0}\in(0,\pi/2)$ for which (\ref{hp:g-min-loc}) holds true.

\item  If $C_{0}\sqrt{2}<6C_{1}$, then also the second term in the right-hand side of (\ref{eqn:g'}) changes its sign in two points of the form $\pi/4\pm\theta_{1}$ for some $\theta_{1}\in(0,\pi/4)$. In this case it turns out that $g(\theta)$ has three stationary points in $(0,\pi/2)$, namely $\pi/4\pm\theta_{1}$ (which are maximum points) and $\pi/4$, which is a minimum point (local or global depending on $C_{0}$ and $C_{1}$).

\end{itemize}

In any case, if $g(\theta)$ has a minimum point in $(0,\pi/2)$, this is necessarily $\pi/4$.
\end{proof}


\begin{lemma}\label{lemma:ABG}

Let $(a,b)\subseteq\re$ be an interval, and let $A_{0}$, $A_{1}$, $B_{0}$, $B_{1}$ be four real numbers. Let us consider the minimum problem
\begin{equation*}
\min\left\{\int_{a}^{b}w''(y)^{2}\,dy:
w\in H^{2}((a,b)),\ \left(w(a),w'(a),w(b),w'(b)\strut\right)=(A_{0},A_{1},B_{0},B_{1})\right\}.
\end{equation*}

Then the unique minimum point is the function
\begin{equation*}
w_{0}(y)=P\left(y-\frac{a+b}{2}\right),
\end{equation*}
where $P(x)=c_{0}+c_{1}x+c_{2}x^{2}+c_{3}x^{3}$ is the polynomial of degree three with coefficients
\begin{equation*}
\left.\begin{array}{c@{\qquad}c}
\displaystyle
c_{0}:=\frac{A_{0}+B_{0}}{2}-\frac{B_{1}-A_{1}}{8}(b-a), & 
\displaystyle
c_{1}:=\frac{3(B_{0}-A_{0})}{2(b-a)}-\frac{A_{1}+B_{1}}{4}, 
\\[3ex]
\displaystyle
c_{2}:=\frac{B_{1}-A_{1}}{2(b-a)}, & 
\displaystyle
c_{3}:=-\frac{2(B_{0}-A_{0})}{(b-a)^{3}}+\frac{A_{1}+B_{1}}{(b-a)^{2}}.
\end{array}\right.
\end{equation*}

As a consequence, the minimum value is
\begin{equation*}
\frac{(B_{1}-A_{1})^{2}}{b-a}+
\frac{12}{(b-a)^{3}}\left[(B_{0}-A_{0})-\frac{A_{1}+B_{1}}{2}(b-a)\right]^{2},
\end{equation*}
and the minimum point satisfies the pointwise estimates
\begin{equation*}
|w_{0}(y)|\leq\frac{3(|A_{0}|+|B_{0}|)}{2}+\frac{|A_{1}|+|B_{1}|}{2}(b-a)
\qquad
\forall y\in[a,b],
\end{equation*}
and
\begin{equation*}
|w_{0}'(y)|\leq\frac{3|B_{0}-A_{0}|}{b-a}+\frac{3(|A_{1}|+|B_{1}|)}{2}
\qquad
\forall y\in[a,b].
\end{equation*}

\end{lemma}

\begin{proof}

From Euler-Lagrange equation we know that minimizers are polynomials of degree three, and $w_{0}$ is the unique such polynomial that fits the boundary conditions.
\end{proof}


\begin{lemma}\label{lemma:bound-D-H}

Let $(a,b)\subseteq\re$ be an interval, and let $D$ and $H$ be positive real numbers. Let $\ep\in(0,1)$ be a real number such that 
\begin{equation}
2\ep^{2}\left(\sqrt{H}+\ep^{2}D\right)<(b-a)
\label{hp:bound-D-H-1}
\end{equation}
and
\begin{equation}
\frac{2}{|\log\ep|}\log\left(1+\frac{45}{2\ep^{4}}\left(\sqrt{H}+\ep^{2}D\right)^{2}\right)\leq 18.
\label{hp:bound-D-H-2}
\end{equation}

Then for every $(A_{0},B_{0})\in[-H,H]^{2}$ and every $(A_{1},B_{1})\in[-D,D]^{2}$ there exists a function $w\in H^{2}((a,b))$ satisfying the boundary conditions
\begin{equation}
\left(w(a),w'(a),w(b),w'(b)\strut\right)=(A_{0},A_{1},B_{0},B_{1}),
\label{hp:BC}
\end{equation}
and the estimates
\begin{gather}
\RPM_{\ep}((a,b),w)\leq 80\left(\sqrt{H}+\ep^{2}D\right),
\label{th:bound-ABG}
\\
\int_{a}^{b}w(x)^{2}\,dx\leq
10\ep^{2}\left(\sqrt{H}+\ep^{2}D\right)^{5}.
\label{th:bound-int}
\end{gather}

\end{lemma}

\begin{proof}

For every real number $\eta\in(0,(b-a)/2)$, let us consider the function
\begin{equation*}
w(x):=\begin{cases}
\varphi_{1}(x)\quad  & \text{if } x\in[a,a+\eta], 
\\
0 & \text{if } x\in[a+\eta,b-\eta], 
\\
\varphi_{2}(x) & \text{if } x\in[b-\eta,b],
\end{cases}
\end{equation*}
where $\varphi_{1}(x)$ is the unique polynomial of degree three such that 
\begin{equation*}
\varphi_{1}(a)=A_{0},
\qquad
\varphi_{1}'(a)=A_{1},
\qquad
\varphi_{1}(a+\eta)=\varphi_{1}'(a+\eta)=0,
\end{equation*}
and $\varphi_{2}(x)$ is the unique polynomial of degree three such that
\begin{equation*}
\varphi_{2}(b)=B_{0},
\qquad
\varphi_{2}'(b)=B_{1},
\qquad
\varphi_{2}(b-\eta)=\varphi_{2}'(b-\eta)=0.
\end{equation*}

We observe that $w$ belongs to $H^{2}((a,b))$, and fulfills the boundary conditions (\ref{hp:BC}). From Lemma~\ref{lemma:ABG} we deduce that $w(x)$ satisfies the integral estimate
\begin{equation*}
\int_{a}^{a+\eta}w''(x)^{2}\,dx\leq
\frac{D^{2}}{\eta}+\frac{12}{\eta^{3}}\left(H+\frac{D}{2}\eta\right)^{2}\leq
\frac{7D^{2}}{\eta}+\frac{24H^{2}}{\eta^{3}},
\end{equation*}
and the pointwise estimates
\begin{equation*}
|w(x)|\leq\frac{3H}{2}+\frac{D\eta}{2}
\qquad\quad\text{and}\quad\qquad
|w'(x)|\leq\frac{3H}{\eta}+\frac{3D}{2}
\end{equation*}
for every $x\in[a,a+\eta]$, from which we deduce that
\begin{equation*}
\int_{a}^{a+\eta}w(x)^{2}\,dx\leq
\frac{9H^{2}\eta}{2}+\frac{D^{2}\eta^{3}}{2},
\end{equation*}
and
\begin{equation*}
\int_{a}^{a+\eta}\log\left(1+w'(x)^{2}\right)\,dx\leq
\eta\log\left(1+\frac{18H^{2}}{\eta^{2}}+\frac{9D^{2}}{2}\right).
\end{equation*}

Analogous estimates hold true in the interval $[b-\eta,b]$, while of course there is no contribution from the central interval $[a+\eta,b-\eta]$. It follows that
\begin{equation}
\RPM_{\ep}((a,b),w) \leq 
\frac{\eta}{\ep^{2}}\left\{\left(\frac{14D^{2}}{\eta^{2}}+\frac{48H^{2}}{\eta^{4}}\right)\ep^{8}+
\frac{2}{|\log\ep|}\log\left(1+\frac{18H^{2}}{\eta^{2}}+\frac{9D^{2}}{2}\right)\right\},
\label{est:bound-ABG}
\end{equation}
and
\begin{equation}
\int_{a}^{b}w(x)^{2}\,dx\leq
9 H^{2}\eta+D^{2}\eta^{3}.
\label{est:bound-int}
\end{equation}

Now we set $\eta:=\ep^{2}\left(\sqrt{H}+\ep^{2}D\right)$. This choice is admissible because $\eta<(b-a)/2$ due to (\ref{hp:bound-D-H-1}). We observe also that $\eta^{4}\geq\ep^{8}H^{2}$ and $\eta^{2}\geq\ep^{8}D^{2}$. As a consequence, from (\ref{est:bound-int}) we conclude that
\begin{equation*}
\int_{a}^{b}w(x)^{2}\,dx\leq
9 H^{2}\eta+D^{2}\eta^{3}\leq
\frac{10\eta^{5}}{\ep^{8}},
\end{equation*}
which proves (\ref{th:bound-int}). Similarly, we obtain that
\begin{equation*}
\left(\frac{14D^{2}}{\eta^{2}}+\frac{48H^{2}}{\eta^{4}}\right)\ep^{8}\leq 62,
\end{equation*}
and
\begin{equation*}
\frac{2}{|\log\ep|}\log\left(1+\frac{18H^{2}}{\eta^{2}}+\frac{9D^{2}}{2}\right)\leq
\frac{2}{|\log\ep|}\log\left(1+\frac{45}{2}\frac{\eta^{2}}{\ep^{8}}\right)\leq 
18,
\end{equation*}
where in the last inequality we exploited (\ref{hp:bound-D-H-2}). Plugging these estimates into (\ref{est:bound-ABG}) we obtain (\ref{th:bound-ABG}).
\end{proof}


\subsection{Proof of Proposition~\ref{prop:loc-min-discr} and Proposition~\ref{prop:loc-min-class}}\label{subsec:loc-min-proof}

In this subsection we prove the two propositions in the same time. The common idea is that every local minimizer to the functional (\ref{defn:JF}) is a staircase where all the steps have the same length and the same height, and this staircase intersects the graph of the forcing term $Mx$ in the midpoint of every horizontal step. This structure applies to entire local minimizers, but also to minimizers to (\ref{defn:mu0}), with the possible exception that the length of the two steps at the boundary might be different. Once that this structure has been established, in both cases we only need to optimize with respect to the length of the steps. 

The proof of the structure result is rather lengthy, because we need first to show that the jump set is discrete, then that the steps are symmetric with respect to the forcing term, and finally that all the steps have the same length.

Since the parameters $\alpha$, $\beta$ and $M$ are fixed once for all, for the sake of shortness in the sequel the functional (\ref{defn:JF}) is denoted only by $\JF(\Omega,w)$. When needed, we also assume that $M>0$ (the case $M<0$ is symmetric, and the easier case $M=0$ is treated in the last paragraph of the proof).

\paragraph{\textmd{\textit{The jump set of local minimizers is discrete}}}

Let us assume that $w_{0}(x)$ is a local minimizer for the functional $\JF((a,b),w)$ in some interval $(a,b)\subseteq\re$. We prove that the set of jump points of $w_{0}$ in $(a,b)$ is finite.

To this end, let us assume by contradiction that this is not the case. Due to the structure of the elements of the space $\PJ((a,b))$, we know that there exist a sequence $\{s_{k}\}\subseteq(a,b)$ of distinct real numbers, a real number $c_{0}$, and a sequence $\{J_{k}\}$ of real numbers different from zero such that
\begin{equation*}
\sum_{k=1}^{\infty}|J_{k}|<+\infty
\end{equation*}
and
\begin{equation}
w_{0}(x)=c_{0}+\sum_{k=1}^{\infty}J_{k}\mathbbm{1}_{(s_{k},b)}(x)
\qquad
\forall x\in(a,b).
\label{defn:v0-jump}
\end{equation}

For every integer $n\geq 2$ we consider the real number
\begin{equation*}
R_{n}:=\sum_{k=n+1}^{\infty}|J_{k}|,
\end{equation*}
and the function $w_{n}:(a,b)\to\re$ defined by
\begin{equation}
w_{n}(x):=c_{0}+\left(J_{1}+\sum_{k=n+1}^{\infty}J_{k}\right)\mathbbm{1}_{(s_{1},b)}(x)
+\sum_{k=2}^{n}J_{k}\mathbbm{1}_{(s_{k},b)}(x)
\qquad
\forall x\in(a,b).
\label{defn:vn-jump}
\end{equation}

We observe that $R_{n}\to 0$, and the function $w_{n}$ has a finite number of jumps located in $s_{1},\ldots,s_{n}$, and the jump in $s_{1}$ has ``absorbed'' all the heights of the jumps in $s_{i}$ with $i\geq n+1$ (the jump height in $s_{1}$ might also vanish). In this way it turns out that
\begin{equation*}
\lim_{x\to a^{+}}w_{n}(x)=\lim_{x\to a^{+}}w_{0}(x)=c_{0}
\qquad\text{and}\qquad
\lim_{x\to b^{-}}w_{n}(x)=\lim_{x\to b^{-}}w_{0}(x)=c_{0}+\sum_{k=1}^{\infty}J_{k},
\end{equation*}
and therefore $w_{0}$ and $w_{n}$ have the same ``boundary data''. As a consequence, due to the minimality of $w_{0}(x)$ this implies that
\begin{equation}
\JF((a,b),w_{n})\geq\JF(a,b),w_{0})
\qquad
\forall n\geq 2.
\label{ineq:vn-v0}
\end{equation}

On the other hand, from (\ref{defn:v0-jump}) and (\ref{defn:vn-jump}) we obtain that
\begin{equation*}
\J_{1/2}((a,b),w_{0})-\J_{1/2}((a,b),w_{n})=
\sum_{k=n+1}^{\infty}|J_{k}|^{1/2}+|J_{1}|^{1/2}-
\left|J_{1}+\sum_{k=n+1}^{\infty}J_{k}\right|^{1/2}.
\end{equation*}

Due to the subadditivity of the square root, the first term can be estimated as
\begin{equation*}
\sum_{k=n+1}^{\infty}|J_{k}|^{1/2}\geq
\left(\sum_{k=n+1}^{\infty}|J_{k}|\right)^{1/2}=(R_{n})^{1/2},
\end{equation*}
while for the second and third term it turns out that
\begin{equation*}
|J_{1}|^{1/2}-\left|J_{1}+\sum_{k=n+1}^{\infty}J_{k}\right|^{1/2}\geq
|J_{1}|^{1/2}-\left(|J_{1}|+R_{n}\right)^{1/2}\geq
-\frac{R_{n}}{2|J_{1}|^{1/2}}.
\end{equation*}

From these two inequalities it follows that
\begin{equation}
\J_{1/2}((a,b),w_{0})-\J_{1/2}((a,b),w_{n})\geq 
(R_{n})^{1/2}-\frac{R_{n}}{2|J_{1}|^{1/2}}.
\label{ineq:J1/2-v0-vn}
\end{equation}

Moreover, from (\ref{defn:vn-jump}) we obtain also that
\begin{equation*}
|w_{0}(x)-w_{n}(x)|\leq R_{n}
\qquad
\forall x\in(a,b)
\end{equation*}
and
\begin{equation*}
|w_{n}(x)-Mx|\leq |c_{0}|+\sum_{k=1}^{\infty}|J_{k}|+M\max\{|a|,|b|\}=:V_{\infty}
\qquad
\forall x\in(a,b),
\end{equation*}
and therefore
\begin{eqnarray}
\int_{a}^{b}(w_{0}(x)-Mx)^{2}\,dx & \geq & 
\int_{a}^{b}\left[(w_{n}(x)-Mx)^{2}+2(w_{n}(x)-Mx)(w_{0}(x)-w_{n}(x))\right]\,dx
\nonumber
\\
& \geq &
\int_{a}^{b}(w_{n}(x)-Mx)^{2}\,dx-2(b-a)V_{\infty}R_{n}
\label{ineq:int-v0-vn}
\end{eqnarray}
for every $n\geq 2$. From (\ref{ineq:J1/2-v0-vn}) and (\ref{ineq:int-v0-vn}) we conclude that
\begin{equation*}
\JF((a,b),w_{0})-\JF((a,b),w_{n})\geq
\alpha(R_{n})^{1/2}-\frac{\alpha R_{n}}{2|J_{1}|^{1/2}}-2\beta(b-a)V_{\infty}R_{n}
\end{equation*}

When $R_{n}\to 0^{+}$ the right-hand side is positive, and this contradicts (\ref{ineq:vn-v0}).

\paragraph{\textmd{\textit{Existence of jump points and intersections}}}

Let us assume that $M>0$, and let us set
\begin{equation}
L_{0}:=\left(\frac{64\alpha^{2}}{\beta^{2}M^{3}}\right)^{1/5}.
\label{defn:L0}
\end{equation}

We claim that, if $w_{0}(x)$ is a local minimizer in some interval $(a,b)\subseteq\re$ with length $b-a>L_{0}$, then $w_{0}(x)$ has either at least one jump point in $(a,b)$ or at least one intersection with the line $Mx$, namely there exists $z_{0}\in(a,b)$ such that $w_{0}(z_{0})=Mz_{0}$.

Indeed, let us assume by contradiction that this is not the case. Then in $(a,b)$ the function $w_{0}(x)$ is a constant of the form $Ma-c$ or $Mb+c$ for some real number $c\geq 0$. In both cases it turns out that
\begin{equation}
\JF((a,b),w_{0})=\left(\frac{M^{2}}{3}(b-a)^{3}+M(b-a)^{2}c+(b-a)c^{2}\right)\beta.
\label{eqn:F-ab-v0}
\end{equation}

For every real number $\tau$ with $0<2\tau<b-a$, let us consider the function $w_{\tau}:(a,b)\to\re$ defined by
\begin{equation}
w_{\tau}(x):=
\begin{cases}
\dfrac{M(a+b)}{2}\quad & \text{if }a+\tau<x<b-\tau, \\[1ex]
w_{0}(x) & \text{if }x\in(a,b)\setminus(a+\tau,b-\tau).
\end{cases}
\label{defn:v-tau-1}
\end{equation}

Since $w_{\tau}$ coincides with $w_{0}$ in a neighborhood of the boundary, from the minimality of $w_{0}$ we deduce that $\JF((a,b),w_{\tau})\geq\JF((a,b),w_{0})$ for every admissible value of $\tau$, and in particular
\begin{equation}
\lim_{\tau\to 0^{+}}\JF((a,b),w_{\tau})\geq\JF((a,b),w_{0}).
\label{ineq:F-vtau-v0}
\end{equation}

The right-hand side is given by (\ref{eqn:F-ab-v0}). As for the left-hand side, we observe that $w_{\tau}$ has two equal jumps of height $c+M(b-a)/2$, while the integral term can be computed starting from the explicit expression (\ref{defn:v-tau-1}). We obtain that
\begin{equation}
\lim_{\tau\to 0^{+}}\JF((a,b),w_{\tau})=
2\alpha\left(c+\frac{M(b-a)}{2}\right)^{1/2}+
\frac{\beta M^{2}}{12}(b-a)^{3}.
\label{eqn:F-ab-vtau}
\end{equation} 

Plugging (\ref{eqn:F-ab-vtau}) and (\ref{eqn:F-ab-v0}) into (\ref{ineq:F-vtau-v0}) we conclude that
\begin{equation}
2\alpha\left(c+\frac{M(b-a)}{2}\right)^{1/2}\geq
\frac{\beta M^{2}}{4}(b-a)^{3}+\beta M(b-a)^{2}c+\beta(b-a)c^{2}.
\label{ineq:absurd}
\end{equation}

We claim that this is impossible if $c\geq 0$ and $b-a>L_{0}$. To this end, we write (\ref{defn:L0}) in the equivalent form $\beta^{2}M^{3}L_{0}^{5}=64\alpha^{2}$, from which we deduce that
\begin{equation}
\beta^{2}M^{3}(b-a)^{5}>64\alpha^{2}
\label{ineq:L0}
\end{equation}
because $b-a>L_{0}$. Now we distinguish two cases.
\begin{itemize}

\item  Let us assume that $c\leq M(b-a)/2$. Multiplying (\ref{ineq:L0}) by $M(b-a)$, and taking the square root, we obtain that
\begin{equation*}
\beta M^{2}(b-a)^{3}>8\alpha[M(b-a)]^{1/2},
\end{equation*}
and therefore
\begin{equation*}
2\alpha\left(c+\frac{M(b-a)}{2}\right)^{1/2}\leq 2\alpha[M(b-a)]^{1/2}
<\frac{\beta M^{2}}{4}(b-a)^{3}.
\end{equation*}

Since the latter is less than or equal to the right-had side of (\ref{ineq:absurd}), we have reached a contradiction in this case.

\item  Let us assume that $c\geq M(b-a)/2$, and in particular that $c$ is positive. We observe that this condition can be rewritten as
\begin{equation*}
2\sqrt{2}\,c^{3/2}\geq M^{3/2}(b-a)^{3/2},
\end{equation*}
while (\ref{ineq:L0}) can be rewritten in the form
\begin{equation*}
\beta(b-a)>\frac{8\alpha}{M^{3/2}(b-a)^{3/2}}.
\end{equation*}

Since $c>0$, from these inequalities it follows that
\begin{equation*}
\beta(b-a)c^{2}>
\frac{8\alpha}{M^{3/2}(b-a)^{3/2}}\cdot c^{2}
\geq 2\alpha\sqrt{2c}\geq
2\alpha\left(c+\frac{M(b-a)}{2}\right)^{1/2}.
\end{equation*}

Since the first term is less than or equal to the right-had side of (\ref{ineq:absurd}), we have reached a contradiction also in this case.

\end{itemize}

\paragraph{\textmd{\textit{Symmetry of jumps}}}

Let $w_{0}(x)$ be a local minimizer in some interval $(a,b)\subseteq\re$, and let $s\in(a,b)$ be a jump point of $w_{0}$. From the first step we already know that $s$ is isolated and therefore, up to restricting to a smaller interval, we can assume that $w_{0}(x)$ is equal to some constant $A$ in $(a,s)$, and to some constant $B\neq A$ in $(s,b)$. We claim that
\begin{equation}
Ms-A=B-Ms
\label{th:jump-symm}
\end{equation}
and that, if $M\neq 0$, the two terms have the same sign of $M$.

To this end, for every $\tau\in(a,b)$ we consider the function $w_{\tau}:(a,b)\to\re$ that is equal to $A$ in $(a,\tau)$, and equal to $B$ in $(\tau,b)$, and we set
\begin{equation*}
\varphi(\tau):=\JF((a,b),w_{\tau})=\alpha\sqrt{B-A}+\beta\int_{a}^{\tau}(A-Mx)^{2}\,dx
+\beta\int_{\tau}^{b}(B-Mx)^{2}\,dx.
\end{equation*}

Since $w_{\tau}$ coincides with $w_{0}$ in a neighborhood of the boundary of the interval, from the minimality of $w_{0}$ we deduce that $\varphi(\tau)$ attains its minimum in $(a,b)$ when $\tau=s$. This implies in particular that
\begin{equation}
0=\varphi'(s)=\beta\left[(Ms-A)^{2}-(B-Ms)^{2}\right]
\label{th:symm-eq}
\end{equation}
and
\begin{equation}
0\leq\varphi''(s)=2\beta M[(Ms-A)+(B-Ms)].
\label{th:symm-ineq}
\end{equation}

Since $\beta>0$ and $B\neq A$, equality (\ref{th:symm-eq}) implies (\ref{th:jump-symm}). If in addition $M\neq 0$, then (\ref{th:symm-ineq}) implies that the two terms in (\ref{th:jump-symm}) have the same sign of~$M$.

\paragraph{\textmd{\textit{Equipartition of intersections}}}

Let us assume that $M>0$, let $w_{0}(x)$ be a local minimizer in some interval $(a,b)\subseteq\re$, and let 
\begin{equation*}
a<z_{1}<z_{2}<\ldots<z_{n}<b 
\end{equation*}
denote the intersections in $(a,b)$ of $w_{0}(x)$ with the line $Mx$, namely the solutions to the equation $w_{0}(x)=Mx$. We observe that between any two intersections there is necessarily at least one jump point, and therefore from the previous steps we know that their number is finite. We claim that
\begin{equation*}
z_{2}-z_{1}=z_{3}-z_{2}=\ldots=z_{n}-z_{n-1}.
\end{equation*}

In order to show the claim it is enough to show that, if $z_{1}<z_{2}<z_{3}$ are three consecutive intersections, then $z_{2}-z_{1}=z_{3}-z_{2}$. To this end, we restrict to the interval $(z_{1},z_{3})$ and we observe that, due to the previous steps, it turns out that
\begin{equation*}
w_{0}(x)=\begin{cases}
Mz_{1}\quad & \text{if }x\in(z_{1},(z_{1}+z_{2})/2), \\
Mz_{2} & \text{if }x\in((z_{1}+z_{2})/2,(z_{2}+z_{3})/2), \\
Mz_{3} & \text{if }x\in((z_{2}+z_{3})/2,z_{3}).
\end{cases}
\end{equation*}

For every $\tau\in(z_{1},z_{3})$ we consider the function $w_{\tau}:(z_{1},z_{3})\to\re$ defined by
\begin{equation*}
w_{\tau}(x)=\begin{cases}
Mz_{1}\quad & \text{if }x\in(z_{1},(z_{1}+\tau)/2), \\
M\tau & \text{if }x\in((z_{1}+\tau)/2,(\tau+z_{3})/2), \\
Mz_{3} & \text{if }x\in((\tau+z_{3})/2,z_{3}).
\end{cases}
\end{equation*}

From this explicit expression it follows that
\begin{equation*}
\JF((z_{1},z_{3}),w_{\tau})=
\alpha\sqrt{M}\left(\sqrt{\tau-z_{1}}+\sqrt{z_{3}-\tau}\,\strut\right)+
\frac{\beta M^{2}}{12}\left\{(\tau-z_{1})^{3}+(z_{3}-\tau)^{3}\right\}.
\end{equation*}

Since $w_{\tau}$ coincides with $w_{0}$ in a neighborhood of the boundary of the interval, from the minimality of $w_{0}$ we deduce that this function of $\tau$ attains its minimum in $(z_{1},z_{3})$ when $\tau=z_{2}$. With the variable change $\tau=z_{1}+t(z_{3}-z_{1})$ this is equivalent to saying that the function
\begin{equation*}
\varphi(t):=C_{0}\left(\sqrt{t}+\sqrt{1-t}\right)+
C_{1}\left(t^{3}+(1-t)^{3}\right),
\end{equation*}
where 
\begin{equation*}
C_{0}:=\alpha\sqrt{M}\sqrt{z_{3}-z_{1}}
\qquad\quad\text{and}\quad\qquad
C_{1}:=\frac{\beta M^{2}}{12}(z_{3}-z_{1})^{3},
\end{equation*}
attains its minimum in $(0,1)$ when $t=(z_{2}-z_{1})/(z_{3}-z_{1})$. On the other hand, from Lemma~\ref{lemma:equipartition} we know that the only possible minimum point is $t=1/2$, and this implies that $z_{2}$ is the midpoint of $(z_{1},z_{3})$.

\paragraph{\textmd{\textit{Estimate from below for the minimum}}}

We are now ready to prove the estimate from below in (\ref{th:est-mu-abLM}). Again we consider the case where $M>0$.

To begin with, we observe that this estimate is trivial when $L\leq 8L_{0}$, because in this case the left-hand side is nonpositive. If $L>8L_{0}$, then from the previous steps we know that any minimizer $w_{0}\in\PJ((0,L))$ intersects the line $Mx$ in at least one point $a_{0}\in(0,4L_{0})$, and in at least one point $b_{0}\in(L-4L_{0},L)$. Indeed, we know that in $(0,2L_{0})$ there exists at least one intersection or jump (because the length of the interval is greater than $L_{0}$), and the same in $(2L_{0},4L_{0})$, and in any case between any two jumps there exists at least one intersection because of the symmetry of jumps.

Now we know that the interval $(a_{0},b_{0})$ is divided into $n\geq 1$ intervals of equal length whose endpoints are intersections. Moreover, $w_{0}(x)$ has exactly one jump point in the midpoint between any two consecutive intersection. As a consequence, the shape of $w_{0}(x)$ in $(a_{0},b_{0})$ depends only on $n$, and with an elementary computation we find that
\begin{eqnarray*}
\JF((0,L),w_{0}) & \geq & \JF((a_{0},b_{0}),w_{0})
\\[1ex]
& = & n\left\{\alpha\sqrt{\frac{M(b_{0}-a_{0})}{n}}+
\frac{\beta M^{2}}{12}\left(\frac{b_{0}-a_{0}}{n}\right)^{3}\right\}.
\end{eqnarray*}

Therefore, from the inequality
\begin{equation*}
A+B\geq 5\left(\frac{A^{4}B}{4^{4}}\right)^{1/5}
\qquad
\forall(A,B)\in[0,+\infty)^{2}
\end{equation*}
we conclude that
\begin{equation*}
\JF((0,L),w_{0})\geq\frac{5}{4}\left(\frac{\alpha^{4}\beta M^{4}}{3}\right)^{1/5}(b_{0}-a_{0})
\geq\frac{5}{4}\left(\frac{\alpha^{4}\beta M^{4}}{3}\right)^{1/5}(L-8L_{0}).
\end{equation*}

Plugging (\ref{defn:L0}) into this inequality we obtain the estimate from below in (\ref{th:est-mu-abLM}).

\paragraph{\textmd{\textit{Estimate from above for the minimum}}}

Let us prove the estimate from above in (\ref{th:est-mu-abLM}).

Let $n:=\lceil L/(2H)\rceil$ denote the smallest integer greater than or equal to $L/(2H)$, where $H$ is defined by (\ref{defn:lambda-0}), and let us consider the function $w_{0}\in\PJ((0,2nH))$ that has intersections with the line $Mx$ in $0$, $2H$, $4H$, \ldots, $2nH$, and jumps in the midpoints of the intervals between consecutive intersections. Since $w_{0}$ is a competitor for the minimum problem (\ref{defn:mu0*}) in the interval $(0,2nH)$, from the monotonicity of $\mu_{0}^{*}$ with respect to $L$ we deduce that
\begin{eqnarray}
\mu_{0}^{*}(\alpha,\beta,L,M) & \leq & \mu_{0}^{*}(\alpha,\beta,2nH,M)
\nonumber
\\[0.5ex]
& \leq & \JF((0,2nH),w_{0})
\nonumber
\\[0.5ex]
& = & n\left(\alpha\sqrt{2MH}+\frac{2\beta M^{2}}{3}H^{3}\right)
\nonumber
\\[0.5ex]
& \leq & \left(\frac{L}{2H}+1\right)
\left(\alpha\sqrt{2MH}+\frac{2\beta M^{2}}{3}H^{3}\right),
\label{est:mu0-H0}
\end{eqnarray}
and we conclude by remarking that the last term coincides with the right-hand side of (\ref{th:est-mu*-abLM}) when $H$ is given by (\ref{defn:lambda-0}).

\paragraph{\textmd{\textit{Structure of entire local minimizers}}}

Let $w_{0}(x)$ be an entire local minimizer. From the previous steps applied in every interval of the form $(-L,L)$, with $L\to +\infty$, we know that the set of intersection points of $w_{0}(x)$ with $Mx$ is discrete and divides the line into segments of the same length $2h>0$, whose midpoints are the unique jump points of $w_{0}$. This is enough to conclude that $w_{0}(x)$ is an oblique translation of some staircase with steps of horizontal length $2h$ and vertical height $2Mh$. It remains only to show that $h=H$, where $H$ is given by (\ref{defn:lambda-0}).

Up to an oblique translation, we can always assume that the intersections are the points of the form $2zh$ with $z\in\z$. Let us consider the interval $(0,2nh)$, where $n$ is a positive integer. Applying (\ref{est:mu0-H0}) with $L=2nh$ we deduce that
\begin{eqnarray*}
n\left(\alpha\sqrt{2Mh}+\frac{2\beta M^{2}}{3}h^{3}\right) & = & \JF((0,2nh),w_{0})
\\
& = & 
\mu_{0}^{*}(\alpha,\beta,2nh,M)
\\[1ex]
& \leq & \left(\frac{2nh}{2H}+1\right)
\left(\alpha\sqrt{2MH}+\frac{2\beta M^{2}}{3}H^{3}\right).
\end{eqnarray*}

Dividing by $nh$, and letting $n\to +\infty$, we conclude that
\begin{equation*}
\alpha\frac{\sqrt{2M}}{\sqrt{h}}+\frac{2\beta M^{2}}{3}h^{2}\leq
\alpha\frac{\sqrt{2M}}{\sqrt{H}}+\frac{2\beta M^{2}}{3}H^{2},
\end{equation*}
and this inequality is possible if and only if $h=H$, because $H$ is the unique minimum point of the left-hand side as a function of $h>0$.

\paragraph{\textmd{\textit{Structure of semi-entire local minimizers}}}

Let $w_{0}:(0,+\infty)\to\re$ be a right-hand semi-entire local minimizer. Let $z_{0}<z_{1}<z_{2}<\ldots$ denote the intersection points of $w_{0}$. Arguing as in the case of entire local minimizers we can show that $z_{k+1}-z_{k}=2H$ for every $k\geq 0$. It remains to find the value of $z_{0}$. To this end, for every real number $\tau$ we consider the function
\begin{equation*}
w_{\tau}(x):=w_{0}(x)+M\tau\mathbbm{1}_{(0,z_{0}+H)}(x)
\qquad
\forall x>0.
\end{equation*}

If we restrict to the interval $(0,z_{1})=(0,z_{0}+2H)$, then $w_{\tau}$ and $w_{0}$ have the same boundary value in $z_{1}$, and therefore by the minimality of $w_{0}$ we know that the function $\varphi(\tau):=\JF_{1/2}((0,z_{1}),w_{\tau})$ has a minimum point in $\tau=0$. On the other hand an easy computation reveals that
\begin{equation*}
\varphi(\tau)=
\alpha\sqrt{M(2H-\tau)}+\beta\int_{0}^{z_{0}+H}M^{2}(z_{0}+\tau-x)^{2}\,dx+
\beta\int_{z_{0}+H}^{z_{1}}M^{2}(z_{1}-x)^{2}\,dx,
\end{equation*}
and therefore
\begin{equation}
0=\varphi'(0)=
-\frac{\alpha\sqrt{M}}{2\sqrt{2H}}+\beta M^{2}\left(z_{0}^{2}-H^{2}\right).
\label{eqn:deriv-phi}
\end{equation}

Finally, we observe that the definition of $H$ in (\ref{defn:lambda-0}) implies that
\begin{equation*}
\frac{\alpha\sqrt{M}}{2\sqrt{2H}}=\frac{2\beta M^{2}}{3}H^{2}.
\end{equation*}

Plugging this identity into (\ref{eqn:deriv-phi}) we obtain that $z_{0}=(5/3)^{1/2}H$, as required.

\paragraph{\textmd{\textit{Existence of entire and semi-entire local minimizers}}}

Up to this point we have just shown that, if entire or semi-entire local minimizers exist, then they have the prescribed form. It remains to show that all oblique translations of the canonical $(H,V)$-staircase are actually entire local minimizers, and that the function $w(x)$ defined by (\ref{defn:semi-entire}) is actually a right-hand semi-entire local minimizer.

The argument is rather standard, and therefore we limit ourselves to sketching the main steps in the case of the canonical $(H,V)$-staircase $S_{H,V}(x)$ (the case of its oblique translations and of semi-entire minimizers is analogous). It is enough to show that, for every positive integer $n$, the function $S_{H,V}(x)$ minimizes $\JF_{1/2}((-2nH,2nH),u)$ among all functions $u\in\PJ((-2nH,2nH))$ that coincide with $S_{H,V}$ at the endpoints. To begin with, we show that the minimum exists. This follows from a standard application of the direct method in the calculus of variations, as in the proof of statement~(\ref{prop:existence}) of Proposition~\ref{prop:mu}. Once we know that the minimum exists, we go back through all the previous steps in order to show that the minimum has only a finite number of equi-spaced intersections points, and their number is the one we expect.

\paragraph{\textmd{\textit{The case $M=0$}}}

When the forcing term vanishes, the estimates of Proposition~\ref{prop:loc-min-discr} are actually trivial. As for Proposition~\ref{prop:loc-min-class}, in this case we have to show that the function $w_{0}(x)\equiv 0$ is the unique entire or semi-entire local minimizer. This can be proved in the following way. Let $w_{0}(x)$ be any entire or semi-entire local minimizer.
\begin{itemize}

\item  We show that the jump set of $w_{0}(x)$ is discrete. This can be done as in the general case, since in that paragraph we never used that $M\neq 0$.

\item  We show the symmetry of jumps (\ref{th:jump-symm}) as in the general case, since that equality was proved without using that $M\neq 0$. As a consequence, we deduce that $|w_{0}(x)|$ is constant.

\item  We show that $w_{0}(x)$ vanishes identically. Indeed, when we consider a long enough interval, any function $w_{0}(x)$ with $|w_{0}(x)|$ constant and different from~0 is worse (due to the overwhelming cost of the fidelity term) than a function with the same boundary values that has two jump points close to the boundary and vanishes elsewhere.

\end{itemize}

This completes the proof also in this special case.
\qed

\begin{rmk}
\begin{em}

The existence of entire and semi-entire local minimizers follows also as a corollary of Proposition~\ref{prop:bu2minloc} and Proposition~\ref{prop:bu2Rminloc}. 

\end{em}
\end{rmk}


\subsection{Compactness and convergence of local minimizers}

In this subsection we prove  Proposition~\ref{prop:bu2minloc} and Proposition~\ref{prop:bu2Rminloc}. The key point in the argument is the following result, where we show that an estimate of order $\ep^{-1}$ for the energy $\RPMF_{\ep}$ in some interval implies an $\ep$-independent estimate for the same energy in a smaller interval. 

\begin{prop}[Boundedness of the energy in a smaller interval]\label{prop:iteration}

Let $L$, $\Gamma_{0}$, $\beta$ be positive real numbers. 

Then there exists two real numbers $\ep_{0}\in(0,1)$ and $\Gamma_{1}>0$ for which the following statement holds true. Let $f:[-(L+1),L+1]\to\re$ be a continuous function such that
\begin{equation}
|f(x)|\leq \Gamma_{0}
\qquad
\forall x\in[-(L+1),L+1],
\label{hp:f-infty}
\end{equation}
let $\ep\in(0,\ep_{0})$, and let 
\begin{equation}
w\in\argmin\loc\left\{\RPMF_{\ep}(\beta,f,(-(L+1),L+1),w):w\in H^{2}((-(L+1),L+1))\right\}
\label{hp:min-loc}
\end{equation}
be a local minimizer such that
\begin{equation}
\RPMF_{\ep}(\beta,f,(-(L+1),L+1),w)\leq\frac{\Gamma_{0}}{\ep}.
\label{hp:univ-bound}
\end{equation}

Then in the smaller interval $(-L,L)$ the local minimizer $w$ satisfies
\begin{equation}
\RPMF_{\ep}(\beta,f,(-L,L),w)\leq 4\Gamma_{1}.
\label{th:univ-bound}
\end{equation}

\end{prop}

\begin{proof}

Let us consider the expression
\begin{equation*}
\Gamma_{2}:=\frac{\Gamma_{1}^{1/4}}{\beta^{1/4}}+\Gamma_{0}^{1/2}+1.
\end{equation*}

We observe that it is possible to choose a real number $\Gamma_{1}\geq\Gamma_{0}$ in such a way that
\begin{equation}
(80+20\beta)\Gamma_{2}+4\beta(L+1)\Gamma_{0}^{2}\leq \Gamma_{1},
\label{defn:M1}
\end{equation}
and it is possible to choose a real number $\ep_{0}\in(0,1/4)$ such that the inequalities
\begin{gather}
\Gamma_{1}\ep^{1/2}|\log\ep|\leq\log 2,
\qquad\qquad
\ep^{3/2}\,\Gamma_{2}\leq L,
\label{hp:ub-ep-1}
\\[1ex]
\frac{2}{|\log\ep|}\log\left(1+\frac{45}{2\ep^{5}}\cdot\Gamma_{2}^{2}\right)\leq 18,
\qquad\qquad
\ep^{5/8}\,\Gamma_{2}^{4}\leq 1
\label{hp:ub-ep-2}
\end{gather}
hold true for every $\ep\in(0,\ep_{0})$.

In the sequel we show that the statement holds true with these values of $\ep_{0}$ and $\Gamma_{1}$.

Since $\ep\in(0,\ep_{0})$ and $\ep_{0}<1/4$, there exists a unique positive integer $n$ such that
\begin{equation}
\frac{1}{4^{2^{n}}}=\frac{1}{2^{2^{n+1}}}\leq\ep<\frac{1}{2^{2^{n}}}.
\label{defn:ep-n}
\end{equation}

For every $k\in\{0,1,\ldots,n\}$ we set
\begin{equation*}
L_{n,k}:=L+1-\frac{1}{2^{n-k}},
\end{equation*}
and we observe that for every $k\in\{0,1,\ldots,n-1\}$ it turns out that
\begin{gather}
L\leq L_{n,k+1}<L_{n,k}<L+1,
\nonumber
\\
L_{n,k}-L_{n,k+1}=\frac{1}{2^{n-k}},
\label{eqn:diff-Lnk}
\end{gather}
and
\begin{equation}
2^{n-k}\leq
2^{2^{n-k-1}}=
\left\{\left(2^{2^{n}}\right)^{1/2}\right\}^{2^{-k}}<
\left(\frac{1}{\ep^{1/2}}\right)^{2^{-k}}.
\label{est:diff-Lnk}
\end{equation}

We claim that
\begin{equation}
\RPMF_{\ep}(\beta,f,(-L_{n,k},L_{n,k}),w)\leq 
\frac{\Gamma_{1}}{\ep^{2^{-k}}}
\qquad
\forall k\in\{0,1,\ldots,n\}.
\label{th:univ-bound-k}
\end{equation}
 
The case $k=0$ follows from assumption (\ref{hp:univ-bound}) because the interval $(-L_{n,0},L_{n,0})$ is contained in the interval $(-(L+1),L+1)$ and $\Gamma_{1}\geq\Gamma_{0}$. Since $L_{n,n}=L$, the case $k=n$ implies (\ref{th:univ-bound}) because of the estimate from below in (\ref{defn:ep-n}).

Now we prove (\ref{th:univ-bound-k}) by finite induction on $k$. Let us assume that (\ref{th:univ-bound-k}) holds true for some $k\in\{0,1,\ldots,n-1\}$, and let  us prove that it holds true also for $k+1$. To begin with, we focus on the interval $(L_{n,k+1},L_{n,k})$, and we observe that
\begin{eqnarray*}
\frac{\Gamma_{1}}{\ep^{2^{-k}}} & \geq & 
\RPMF_{\ep}(\beta,f,(-L_{n,k},L_{n,k}),w)
\\[1ex]
& \geq & \RPMF_{\ep}(\beta,f,(L_{n,k+1},L_{n,k}),w)
\\[1ex]
& \geq & 
\int_{L_{n,k+1}}^{L_{n,k}}\left\{\frac{1}{\omep^{2}}\log\left(1+w'(y)^{2}\right)+
\beta(w(y)-f(y))^{2}\right\}dy
\\[1ex]
& \geq & (L_{n,k}-L_{n,k+1})\cdot
\\[0.5ex]
& & 
\mbox{}\cdot
\min\left\{\frac{1}{\omep^{2}}\log\left(1+w'(y)^{2}\right)+
\beta(w(y)-f(y))^{2}:y\in[L_{n,k+1},L_{n,k}]\right\}.
\end{eqnarray*}

If $b_{k,\ep}\in[L_{n,k+1},L_{n,k}]$ is any minimum point, recalling (\ref{eqn:diff-Lnk}), (\ref{est:diff-Lnk}), and the first inequality in (\ref{hp:ub-ep-1}), this proves that
\begin{equation*}
\log\left(1+w'(b_{k,\ep})^{2}\right)\leq 
\frac{\Gamma_{1}}{\ep^{2^{-k}}}\cdot\omep^{2}\cdot 2^{n-k}\leq
\frac{\Gamma_{1}}{\left(\ep^{3/2}\right)^{2^{-k}}}\cdot\omep^{2}\leq
\Gamma_{1}\ep^{1/2}|\log\ep|\leq
\log 2,
\end{equation*}
and
\begin{equation*}
\beta(w(b_{k,\ep})-f(b_{k,\ep}))^{2}\leq 
\frac{\Gamma_{1}}{\ep^{2^{-k}}}\cdot 2^{n-k}\leq
\frac{\Gamma_{1}}{\left(\ep^{3/2}\right)^{2^{-k}}}.
\end{equation*}

From these two inequalities and (\ref{hp:f-infty}) we deduce that $|w'(b_{k,\ep})|\leq 1$ and
\begin{equation*}
|w(b_{k,\ep})|\leq 
\frac{\Gamma_{1}^{1/2}}{\beta^{1/2}\left(\ep^{3/4}\right)^{2^{-k}}}+\Gamma_{0}\leq
\frac{(\Gamma_{1}/\beta)^{1/2}+\Gamma_{0}}{\left(\ep^{3/4}\right)^{2^{-k}}}.
\end{equation*}

With an analogous argument, we can show that there exists $a_{k,\ep}\in[-L_{n,k},-L_{n,k+1}]$ such that
\begin{equation*}
|w'(a_{k,\ep})|\leq 1
\qquad\quad\mbox{and}\quad\qquad
|w(a_{k,\ep})|\leq\frac{(\Gamma_{1}/\beta)^{1/2}+\Gamma_{0}}{\left(\ep^{3/4}\right)^{2^{-k}}}.
\end{equation*}

Now we exploit that $w$ minimizes $\RPMF_{\ep}$ in the interval $(a_{k,\ep},b_{k,\ep})$ with respect to its boundary conditions, and we estimate the minimum value by applying Lemma~\ref{lemma:bound-D-H} with
\begin{equation*}
(a,b)=(a_{k,\ep},b_{k,\ep}),
\qquad\quad
D:=1,
\qquad\quad
H:=\frac{(\Gamma_{1}/\beta)^{1/2}+\Gamma_{0}}{\left(\ep^{3/4}\right)^{2^{-k}}}.
\end{equation*}

We observe that
\begin{equation*}
\sqrt{H}+\ep^{2}D\leq
\frac{\left[(\Gamma_{1}/\beta)^{1/2}+\Gamma_{0}\right]^{1/2}}{\left(\ep^{3/8}\right)^{2^{-k}}}+1\leq
\frac{\Gamma_{2}}{\left(\ep^{3/8}\right)^{2^{-k}}},
\end{equation*}
and in particular from the second inequality in (\ref{hp:ub-ep-1}) we obtain that
\begin{equation*}
\ep^{2}\left(\sqrt{H}+\ep^{2}D\right)\leq
\ep^{3/2}\Gamma_{2}\leq
L<
\frac{b_{k,\ep}-a_{k,\ep}}{2},
\end{equation*}
which shows that assumption (\ref{hp:bound-D-H-1}) is satisfied, while from the first inequality in (\ref{hp:ub-ep-2}) we obtain that
\begin{equation*}
\frac{2}{|\log\ep|}\log\left(1+\frac{45}{2\ep^{4}}\left(\sqrt{H}+\ep^{2}D\right)^{2}\right)\leq
\frac{2}{|\log\ep|}\log\left(1+\frac{45}{2\ep^{4}}\cdot\frac{\Gamma_{2}^{2}}{\ep}\right)\leq 
18.
\end{equation*}
which shows that assumption (\ref{hp:bound-D-H-2}) is satisfied. Therefore, from Lemma~\ref{lemma:bound-D-H} we deduce the existence of $w_{k,\ep}\in H^{2}((a_{k,\ep},b_{k,\ep}))$, with the same boundary values (function and derivative) of $w$, satisfying
\begin{equation*}
\RPM_{\ep}((a_{k,\ep},b_{k,\ep}),w_{k,\ep})\leq
80\frac{\Gamma_{2}}{\left(\ep^{3/8}\right)^{2^{-k}}}\leq
80\frac{\Gamma_{2}}{\ep^{2^{-k-1}}}
\end{equation*}
and
\begin{equation*}
\int_{a_{k,\ep}}^{b_{k,\ep}}w_{k,\ep}(x)^{2}\,dx\leq
10\ep^{2}\frac{\Gamma_{2}^{5}}{\left(\ep^{15/8}\right)^{2^{-k}}}=
10\frac{\ep^{11/8}}{\left(\ep^{15/8}\right)^{2^{-k}}}\cdot\left(\ep^{5/8}\Gamma_{2}^{4}\right)\cdot\Gamma_{2}\leq
10\frac{\Gamma_{2}}{\ep^{2^{-k-1}}},
\end{equation*}
where the last inequality follows from the second relation in (\ref{hp:ub-ep-2}). From the last two estimates and the minimality of $w$ we conclude that
\begin{eqnarray*}
\lefteqn{\hspace{-2em}
\RPMF_{\ep}(\beta,f,(-L_{n,k+1},L_{n,k+1}),w)
}
\\[1ex]
& \leq & 
\RPMF_{\ep}(\beta,f,(a_{k,\ep},b_{k,\ep}),w)
\\[1ex]
& \leq & 
\RPMF_{\ep}(\beta,f,(a_{k,\ep},b_{k,\ep}),w_{k,\ep})
\\[1ex]
& \leq &
\RPM_{\ep}((a_{k,\ep},b_{k,\ep}),w_{k,\ep})+
2\beta\int_{a_{k,\ep}}^{b_{k,\ep}}w_{k,\ep}(x)^{2}\,dx+
2\beta\int_{a_{k,\ep}}^{b_{k,\ep}}f(x)^{2}\,dx
\\[0.5ex]
& \leq & 
(80+20\beta)\frac{\Gamma_{2}}{\ep^{2^{-k-1}}}+
2\beta(2L+2)\Gamma_{0}^{2}
\\[0.5ex]
& \leq &
\frac{\Gamma_{1}}{\ep^{2^{-k-1}}},
\end{eqnarray*}
where the last two inequalities we exploited (\ref{hp:f-infty}) and (\ref{defn:M1}), respectively. This completes the inductive step, and hence also the proof.
\end{proof}

\begin{rmk}\label{rmk:iteration}
\begin{em}

Proposition~\ref{prop:iteration} can be extended in a straightforward way to one-sided local minimizers. To this end, it is enough to replace in the statement the interval $(-L,L)$ with $(0,L)$, the interval $(-(L+1),L+1)$ with $(0,L+1)$, and ``loc'' with ``R-loc''. The proof is analogous and somewhat simpler, because we just need to work on one side of the interval.

\end{em}
\end{rmk}


\subsubsection{Proof of Proposition~\ref{prop:bu2minloc}}

\paragraph{\textmd{\textit{Existence of a limit}}}

We prove that there exist a function $w_{\infty}:\re\to\re$ and an increasing sequence $\{n_{k}\}$ of positive integers such that, for every $L>0$, the restriction of $w_{\infty}$ to the interval $(-L,L)$ belongs to $\PJ((-L,L))$ and $w_{n_{k}}(x)\to w_{\infty}(x)$ in $L^{2}((-L,L))$. 

To this end, it is enough to prove that, for every fixed real number $L>0$, it turns out that
\begin{equation}
\sup\left\{\RPM_{\ep_{n}}((-L,L),w_{n})+\int_{-L}^{L}w_{n}(x)^{2}\,dx:n\in\n\right\}<+\infty.
\label{th:F<=CL}
\end{equation}

Indeed, once that this uniform bound has been established (the supremum might depend on $L$, of course), the compactness result of statement~(\ref{stat:ABG-cpt}) of Theorem~\ref{thm:ABG} implies that the sequence $\{w_{n}\}$ is relatively compact in $L^{2}((-L,L))$ for this fixed value on $L$, and any limit function lies in $\PJ((-L,L))$. At this point we apply the result for a sequence of intervals $(-L_{k},L_{k})$ with  $L_{k}\to +\infty$, and with a classical diagonal procedure we obtain the subsequence that converges in all bounded intervals.

In order to prove (\ref{th:F<=CL}), we begin by observing that, due to assumption~(ii), there exists a constant $M_{L}$ such that
\begin{equation*}
|g_{n}(x)|\leq M_{L}
\qquad
\forall x\in[-(L+1),L+1],\quad\forall n\in\n.
\end{equation*}

At this point we apply Proposition~\ref{prop:iteration} with
\begin{equation*}
w(x):=w_{n}(x),
\qquad\quad
f(x):=g_{n}(x),
\qquad\quad
\Gamma_{0}:=\max\{M_{L},C_{0}\}.
\end{equation*}

This is possible because assumptions (\ref{hp:min-loc}) and (\ref{hp:univ-bound}) are satisfied for trivial reasons as soon as $[-(L+1),L+1]\subseteq(A_{n},B_{n})$ and $|\log\ep_{n}|\geq 1$. From Proposition~\ref{prop:iteration} we obtain that there exists a constant $\Gamma_{1}$ such that
\begin{equation}
\RPMF_{\ep_{n}}(\beta,g_{n},(-L,L),w_{n})\leq 4\Gamma_{1}
\label{est:ABGF-4Gamma}
\end{equation}
when $n$ is large enough. This implies (\ref{th:F<=CL}) because the left hand-side of (\ref{est:ABGF-4Gamma}) controls the first term in the left-hand side of (\ref{th:F<=CL}), while the integral can be estimated as
\begin{equation*}
\int_{-L}^{L}w_{n}(x)^{2}\,dx\leq
2\int_{-L}^{L}(w_{n}(x)-g_{n}(x))^{2}\,dx+2\int_{-L}^{L}g_{n}(x)^{2}\,dx,
\end{equation*}
where the first integral is controlled again by the left hand-side of (\ref{est:ABGF-4Gamma}), and the second integral is controlled because of the uniform bound on $g_{n}$.

\paragraph{\textmd{\textit{Characterization of the limit}}}

Let $w_{\infty}$ be any limit function identified in the first paragraph of the proof. We claim that $w_{\infty}$ is an entire local minimizer for the functional (\ref{defn:JF}) with $\alpha$ defined by (\ref{defn:alpha-0}).

The function $w_{\infty}(x)$ is by definition the limit in $L^{2}\loc(\re)$ of some sequence $w_{n_{k}}(x)$, and from the uniform bounds (\ref{est:ABGF-4Gamma}) we deduce also that $\log(1+w_{n_{k}}'(x)^{2})\to 0$ in $L^{1}\loc(\re)$. Up to further subsequences (not relabeled) we can assume that in both cases the convergence is also pointwise for almost every $x\in\re$. Now let us consider any interval $(a,b)\subseteq\re$ whose endpoints are not jump points of $w_{\infty}$, and such that $w_{n_{k}}(x)\to w_{\infty}(x)$ and $w_{n_{k}}'(x)\to 0$ for $x\in\{a,b\}$. 

Let $v\in\PJ((a,b))$ be any function with the same boundary conditions of $w_{\infty}(x)$ in the usual sense. From statement~(\ref{stat:ABG-recovery}) of Theorem~\ref{thm:ABG} applied with the quadruple of boundary data
\begin{equation*}
(w_{n_{k}}(a),w_{n_{k}}'(a),w_{n_{k}}(b),w_{n_{k}}'(b))\to
(w_{\infty}(a),0,w_{\infty}(b),0)=
(v(a),0,v(b),0)
\end{equation*}
we obtain a recovery sequence $\{v_{k}\}\subseteq H^{2}((a,b))$ for $v$ that has the same boundary conditions of $w_{n_{k}}$ in $a$ and $b$ (both on the function and on the derivative). From the minimality of $w_{n_{k}}$  we deduce that
\begin{equation*}
\RPMF_{\ep_{n_{k}}}(\beta,g_{n_{k}},(a,b),w_{n_{k}})\leq 
\RPMF_{\ep_{n_{k}}}(\beta,g_{n_{k}},(a,b),v_{k})
\end{equation*}
for every positive integer $k$. Letting $k\to +\infty$, and recalling statement~(\ref{stat:ABG-gconv}) of Theorem~\ref{thm:ABG}, we conclude that
\begin{eqnarray*}
\JF_{1/2}(\alpha_{0},\beta,M,(a,b),w_{\infty}) & \leq & 
\liminf_{k\to +\infty}\RPMF_{\ep_{n_{k}}}(\beta,g_{n_{k}},(a,b),w_{n_{k}})
\\[1ex]
& \leq & \lim_{k\to +\infty}\RPMF_{\ep_{n_{k}}}(\beta,g_{n_{k}},(a,b),v_{k})
\\[1ex]
& = & \JF_{1/2}(\alpha_{0},\beta,M,(a,b),v).
\end{eqnarray*}

Since $v$ is arbitrary, this proves that $w_{\infty}$ is a local minimizer of the limit functional in the interval $(a,b)$. Since intervals of this type invade the real line, this proves that $w_{\infty}$ is an entire local minimizer for the limit functional, as required.

\paragraph{\textmd{\textit{Strict convergence}}}

In the special case where $v(x)\equiv w_{\infty}(x)$ in $(a,b)$, the argument of the previous paragraph gives that
\begin{equation*}
\lim_{k\to +\infty}\RPM_{\ep_{n_{k}}}((a,b),w_{n_{k}})=
\alpha_{0}\J_{1/2}((a,b),w_{\infty}),
\end{equation*}
namely $\{w_{n_{k}}\}$ is a recovery sequence for $w_{\infty}$ in the interval $(a,b)$. At this point, from statement~(\ref{stat:ABG-uc}) of Theorem~\ref{thm:ABG}, we conclude that $w_{n_{k}}\auto w_{\infty}$  in $BV((a,b))$. Since intervals of this type invade the real line, this completes the proof.
\qed


\subsubsection{Proof of Proposition~\ref{prop:bu2Rminloc}}

The proof is analogous to the proof of Proposition~\ref{prop:bu2minloc}, and hence we limit ourselves to sketching the argument. 

In the first step we show that there exist a function $w_{\infty}:(0,+\infty)\to\re$ and an increasing sequence $\{n_{k}\}$ of positive integers such that
\begin{itemize}

\item  the restriction of $w_{\infty}$ to the interval $(0,L)$ belongs to $\PJ((0,L))$ for every $L>0$,

\item  $w_{n_{k}}(x)\to w_{\infty}(x)$ in $L^{2}((0,L))$ and $\log(1+w_{n_{k}}'(x)^{2})\to 0$ in $L^{1}((0,L))$ for every $L>0$,

\item  $w_{n_{k}}(x)\to w_{\infty}(x)$ and $w_{n_{k}}'(x)\to 0$ for almost every $x>0$.

\end{itemize}  

The argument relies on the one-sided version of Proposition~\ref{prop:iteration} (see Remark~\ref{rmk:iteration}), and on the compactness result of statement~(\ref{stat:ABG-cpt}) of Theorem~\ref{thm:ABG}.

In the second step we consider intervals of the form $(0,L)$, where $L$ is any positive real number in which we have the pointwise convergence $w_{n_{k}}(L)\to w_{\infty}(L)$ and $w_{n_{k}}'(L)\to 0$, and such that $L$ is not a jump point of $w_{\infty}$ (both conditions hold true for almost every point in the half-line). Then we consider any function $v\in\PJ((0,L))$ such $v(L)=w_{\infty}(L)$, where boundary values are intended in the usual sense. From statement~(\ref{stat:ABG-recovery}) of Theorem~\ref{thm:ABG}, applied with the quadruple of initial data
\begin{equation*}
(v(0),0,w_{n_{k}}(L),w_{n_{k}}'(L))\to(v(0),0,w_{\infty}(L),0)=(v(0),0,v(L),0),
\end{equation*}
we obtain a recovery sequence $\{v_{k}\}\subseteq H^{2}((0,L))$ for $v$ that has the same boundary conditions of $w_{n_{k}}$ in $x=L$. Thus from the minimality of $w_{n_{k}}$ in $(0,L)$ we deduce as in the previous case that
\begin{equation*}
\JF_{1/2}(\alpha_{0},\beta,M,(0,L),w_{\infty})\leq\JF_{1/2}(\alpha_{0},\beta,M,(0,L),v).
\end{equation*}

Since $L$ can be chosen to be arbitrarily large, this is enough to conclude that $w_{\infty}$ is a right-hand minimizer in $(0,+\infty)$.

Finally, in the third step we conclude as before that the convergence is strict in every interval $(0,L)$ such that $L$ is not a jump point of $w_{\infty}$.
\qed


\setcounter{equation}{0}
\section{Possible extensions}\label{sec:extension}

Our proof of Theorem~\ref{thm:asympt-min} relies just on the Gamma-convergence results for the rescaled functionals (\ref{defn:ABG}), and on the estimates of Proposition~\ref{prop:loc-min-discr} for the minima of the limit functional with linear forcing term. Our proofs of Theorems~\ref{thm:BU} and~\ref{thm:varifold} rely on the characterization of local minima for the limit functional, and on the compactness result that follows from Proposition~\ref{prop:iteration}. For these reasons, we expect that these results can be extended to more general models by just extending to these models the tools that we exploited here. For example, it is possible to consider more general fidelity terms of the form
\begin{equation*}
\int_{0}^{1}\beta(x)|u(x)-f(x)|^{p}\,dx,
\end{equation*}
for suitable choices of the exponent $p\geq 1$ (but also every $p>0$ should be fine) and of the coefficient $\beta(x)$, provided that it is continuous and strictly positive.

In the sequel we focus on less trivial generalizations that involve the principal part, and we discuss three possibilities.

\paragraph{\textmd{\textit{Different convex-concave Lagrangians}}}

We can replace the function $\phi(p):=\log(1+p^{2})$ with different functions, for example those presented in (\ref{defn:phi-1}) or (\ref{defn:phi-2}). This leads to functionals with principal part of the form
\begin{equation*}
\sPM_{\ep}(u):=\int_{0}^{1}\left\{\ep^{6}\omep^{4}u''(x)^{2}+\phi(u'(x))\right\}\,dx,
\end{equation*}
where now $\omep:=\ep\phi(1/\ep^{2})^{1/2}$. Under rather general assumptions on $\phi$, the blow-ups of minimizers at scale $\omep$ are local minimizers for the rescaled functionals
\begin{equation*}
\sRPM_{\ep}(\Omega,v):=\int_{\Omega}\left\{\ep^{6}v''(x)^{2}+\frac{1}{\omep^{2}}\phi(v'(x))\right\}\,dx,
\end{equation*}
and this family Gamma-converges to a suitable multiple of the functional $\J_{\sigma}(\Omega,v)$, which is the natural generalization of (\ref{defn:J}) obtained by replacing 1/2 with a different exponent $\sigma\in(0,1)$ that depends on the growth at infinity of $\phi(p)$ (actually in this case we obtain only exponents in $[1/2,1)$). All the results of this paper can be easily extended, more or less with the same techniques. 

\paragraph{\textmd{\textit{Higher order singular perturbation}}}

We can replace second order derivatives with derivatives of higher order. This leads to functionals with principal part of the form
\begin{equation*}
\sPM_{\ep}(u):=
\int_{0}^{1}\left\{\ep^{4k-2}\omep^{2k}u^{(k)}(x)^{2}+\log\left(1+u'(x)^{2}\right)\right\}\,dx,
\end{equation*}
where $u^{(k)}(x)$ denotes the derivative of $u(x)$ of order $k\geq 2$, and $\omep$ is defined as in (\ref{defn:omep}). Also in this case the rescaled functionals
\begin{equation*}
\sRPM_{\ep}(\Omega,v):=
\int_{\Omega}\left\{\ep^{4k-2}v^{(k)}(x)^{2}+\frac{1}{\omep^{2}}\log\left(1+v'(x)^{2}\right)\right\}\,dx
\end{equation*}
Gamma-converges to a suitable multiple of $\J_{\sigma}(\Omega,v)$, now with $\sigma=1/k$. Therefore, it seems reasonable that the results of this paper can be extended, even if some steps (for example the iteration argument in the compactness result) might require some extra work.

Of course, one can also combine a higher order singular perturbation with a different choice of $\phi(p)$, and/or choose a different exponent for the higher order derivative.

\paragraph{\textmd{\textit{Space discretization}}}

In a different direction, it is possible to consider a space discretization of the problem where derivatives are replaced by finite differences. This leads to functionals of the form
\begin{equation*}
\sPM_{\ep}(u):=\int_{0}^{1-\ep^{2}\omep}\log\left(1+\left(\frac{u(x+\ep^{2}\omep)-u(x)}{\ep^{2}\omep}\right)^{2}\right)\,dx,
\end{equation*}
possibly defined in the space of functions that are piecewise constant with steps of length $\ep^{2}\omep$, where now $\omep$ is defined as in (\ref{defn:omep}). This is equivalent to considering the original functional (\ref{defn:PM}), depending on true derivatives, but restricted to the space of functions that are piecewise affine, again with respect to some grid with size $\ep^{2}\omep$. The natural rescaling corresponding to blow-ups at scale $\omep$ leads to the family of functionals
\begin{equation*}
\sRPM_{\ep}(v):=\frac{1}{\omep^{2}}\int_{\Omega}\log\left(1+\left(\frac{v(x+\ep^{2})-v(x)}{\ep^{2}}\right)^{2}\right)\,dx.
\end{equation*}

Apparently this scaling of the Perona-Malik functional, although elementary, never appeared in the previous literature on the discrete model (see~\cite{2001-CPAM-Esedoglu,2003-M3AS-MorNeg,2011-SIAM-SlowSD,2018-ActApplMath-BraVal}). The Gamma-limit turns out to be a multiple of the functional $\J_{0}(\Omega,v)$, namely the functional that simply counts the number of jumps of $v$ in $\Omega$, regardless of jump heights. Again it is possible to extend the results of this paper, and some steps are even easier, for example the iterative argument in Proposition~\ref{prop:iteration} and the classification of local minimizers for the limit functional.


\setcounter{equation}{0}
\section{Future perspectives and open problems}\label{sec:open}

In this final section we present some questions that remain open, and that could deserve further investigation.

The first one concerns uniqueness of minimizers, which is always a challenging question when the Lagrangian is non-convex. We recall that for the model (\ref{defn:AM}) uniqueness is known in some cases (see~\cite[Theorem~1.1 and subsequent Remark~4]{1993-CalcVar-Muller}), but in that case the forcing term is rather special and there are periodic boundary conditions. 

\begin{open}[Uniqueness of minimizers]

Let us consider the minimum problem (\ref{defn:DMnf}), under the same assumptions of Proposition~\ref{prop:basic}. Determine whether the minimizer is unique, at least when $\ep$ is small enough and/or the forcing term $f(x)$ is smooth enough.

\end{open}

Concerning Theorem~\ref{thm:asympt-min}, it could be interesting to investigate the asymptotic behavior of minima under weaker regularity assumptions on the forcing term $f(x)$. 

\begin{open}[Existence of the limit of rescaled minimum values]

Characterize all functions $f\in L^{2}((0,1))$ such that the limit in (\ref{th:asympt-min}) exists, or exists and is a real number, or exists and coincides with the right-hand side, up to defining $f'(x)$ in a suitable way.

\end{open}

The question is largely open. It is also conceivable that the vanishing order of $m(\ep,\beta,f)$ as $\ep\to 0^{+}$ depends on the regularity of $f(x)$ in terms of H\"older continuity, Sobolev exponents or even fractional Sobolev spaces, which motivates the following question.

\begin{open}[Vanishing order of minima vs regularity of the forcing term]

Find any connection between the vanishing order of $m(\ep,\beta,f)$ as $\ep\to 0^{+}$ and the regularity of the forcing term $f(x)$.

\end{open}

Here we present the results that we know for the time being.
\begin{itemize}

\item  For every $f\in\PJ((0,1))$ with a finite number of jumps it turns out that
\begin{equation}
\lim_{\ep\to 0^{+}}\frac{m(\ep,\beta,f)}{\omep^{5/2}}=
4\left(\frac{2}{3}\right)^{1/2}5^{3/4}\cdot\J_{1/2}((0,1),f).
\label{th:lim-PJ}
\end{equation}

The same should be true when $\J_{1/2}((0,1),f)<+\infty$.

\item It should not be difficult to extend (\ref{th:asympt-min}) to every $f\in H^{1}((0,1))$, and probably also to forcing terms that are the sum of a function $f_{1}\in H^{1}((0,1))$ and a function $f_{2}\in\PJ((0,1))$ with $\J_{1/2}((0,1),f_{2})<+\infty$. This extension should require only some technical adjustments in our proof, because the key point (\ref{lim:fepL2f}) remains true.

\item  Heuristically, when minimizing (\ref{defn:ABGF}) we can replace the rescaled Perona-Malik functional (\ref{defn:ABG}) by its Gamma-limit (\ref{defn:J}). This leads to a minimization problem in the class of pure jump functions, that we can further simplify  by restricting to competitors whose jump points are equally spaced at some fixed distance $\delta$, to be optimized with respect to $\ep$. By formalizing this idea we obtain the following two estimates from above.
\begin{itemize}

\item  If $f(x)$ is $a$-H\"older continuous with some exponent $a\in(0,1]$ and some constant $H$, then it turns out that
\begin{equation*}
\limsup_{\ep\to 0^{+}}\frac{m(\ep,\beta,f)}{\omep^{10a/(3a+2)}}\leq
c_{a}H^{4/(3a+2)}.
\end{equation*}

\item  If $f\in W^{1,p}((0,1))$ for some exponent $p\in[1,2]$, then it turns out that
\begin{equation*}
\limsup_{\ep\to 0^{+}}\frac{m(\ep,\beta,f)}{\omep^{(15p-10)/(7p-4)}}\leq
c_{p}\|f'\|_{L^{p}((0,1))}^{(5p-2)/(7p-4)}.
\end{equation*}

\end{itemize}

\item  The set of forcing terms $f\in L^{2}((0,1))$ for which the limit in (\ref{th:asympt-min}) exists has empty interior, even if we allow the limit to be~$+\infty$, and even if we restrict ourselves to a sequence $\ep_{n}\to 0^{+}$. Indeed, for every fixed $\ep_{n}\in(0,1)$, the function $f\to m(\ep_{n},\beta,f)$ is continuous in $L^{2}((0,1))$, and therefore also the function 
\begin{equation*}
\Psi_{n}(f):=\arctan\left(\frac{m(\ep_{n},\beta,f)}{\omega(\ep_{n})^{2}}\right)
\end{equation*}
is continuous in the same space. Let us assume by contradiction that $\Psi_{n}(f)$ converges to some $\Psi_{\infty}(f)$ for every $f$ in some open set $\mathcal{U}\subseteq L^{2}((0,1))$. Since $\mathcal{U}$ is a Baire space, and $\Psi_{\infty}$ is the pointwise limit of continuous functions, then necessarily $\Psi_{\infty}$ is continuous in some $G_{\delta}$ subset $\mathcal{V}\subseteq\mathcal{U}$. Now on the one hand we know from (\ref{th:lim-PJ}) that $\Psi_{\infty}(f)=0$ for every piecewise constant function with a finite number of jumps, and this class is dense in $L^{2}((0,1))$, and therefore $\Psi_{\infty}(f)=0$ for every $f\in\mathcal{V}$. On the other hand, also functions of class $C^{1}$ with right-hand side of (\ref{th:asympt-min}) greater than~1 are dense in $L^{2}((0,1))$, which implies that $\Psi_{\infty}(f)\geq 1$ for every $f\in\mathcal{V}$.

\end{itemize}

As for the convergence of minimizers, on the one hand we expect that the $C^{1}$ regularity of $f(x)$ is required in order to characterize the blow-ups of minimizers with $\ep$-dependent centers as we did in Theorem~\ref{thm:BU}. On the other hand, the statement of Theorem~\ref{thm:varifold} seems to require less regularity on $f(x)$, in contrast with our proof that is deeply based on Theorem~\ref{thm:BU}.

\begin{open}[Strict and varifold convergence of minimizers]

Extend the results of Theorem~\ref{thm:varifold} to less regular forcing terms, and in particular determine whether the results hold true for every $f\in BV((0,1))$ (up to a suitable extension of identity~(\ref{th:varifold}) to bounded variation functions).

\end{open}

Finally, since this paper deals with the Perona-Malik functional in dimension one, we conclude with the following natural and challenging question.

\begin{open}[Any space dimension]

Extend the results of this paper to higher dimensions.

\end{open}


\setcounter{equation}{0}
\appendix

\section{Appendix}

In this final appendix we prove the results stated in section~\ref{sec:gconv}. To this end, we need three preliminary technical lemmata. The first one is the classical estimate from below for the rescaled Perona-Malik functional in an interval where $|u'(x)|$ is ``large'' (the argument is analogous to a step in the proof of~\cite[Proposition~3.3]{ABG}).

\begin{lemma}[Basic estimate from below]\label{lemma:basic-below}

Let $(\alpha,\beta)\subseteq\re$ be an interval, and let $u\in H^{2}((\alpha,\beta))$. Let us assume that there exists a real number $D>0$ such that $|u'(\alpha)|=|u'(\beta)|=D$, and $|u'(x)|\geq D$ for every $x\in(\alpha,\beta)$.

Then for every $\ep\in(0,1)$ it turns out that
\begin{equation}
\RPM_{\ep}((\alpha,\beta),u)\geq
4\left(\frac{2}{3}\right)^{1/2}\left(\frac{\log(1+D^{2})}{|\log\ep|}\right)^{3/4}
\left(|u(\beta)-u(\alpha)|-D(\beta-\alpha)\strut\right)^{1/2}.
\label{th:usual-below}
\end{equation}

\end{lemma}

\begin{proof}

Let us observe that our assumptions imply that either $u'(x)\geq D$ for every $x\in(\alpha,\beta)$, or $u'(x)\leq -D$ for every $x\in(\alpha,\beta)$. In both cases it turns out that $u'(x)$ has the same value at the two endpoints of the interval. Moreover $u(\beta)-u(\alpha)$ has the same sign of $u'(\alpha)$ and $u'(\beta)$, and
\begin{equation*}
|u(\beta)-u(\alpha)|-D(\beta-\alpha)=
\left|(u(\beta)-u(\alpha))-\frac{u'(\alpha)+u'(\beta)}{2}(\beta-\alpha)\strut\right|\geq
0.
\end{equation*}

As a consequence, from Lemma~\ref{lemma:ABG} we obtain that
\begin{equation*}
\int_{\alpha}^{\beta}u''(x)^{2}\,dx\geq
\frac{12}{(\beta-\alpha)^{3}}\left(|u(\beta)-u(\alpha)|-D(\beta-\alpha)\strut\right)^{2},
\end{equation*}
and therefore
\begin{eqnarray*}
\RPM_{\ep}((\alpha,\beta),u) & = &
\int_{\alpha}^{\beta}\left(\ep^{6}u''(x)^{2}+
\frac{1}{\omep^{2}}\log\left(1+u'(x)^{2}\right)\right)\,dx
\\[1ex]
& \geq &
\frac{12\ep^{6}}{(\beta-\alpha)^{3}}\left(|u(\beta)-u(\alpha)|-D(\beta-\alpha)\strut\right)^{2}+
\frac{\beta-\alpha}{\omep^{2}}\log\left(1+D^{2}\right).
\end{eqnarray*}

Applying the classical inequality
\begin{equation*}
A+B\geq\frac{4}{3^{3/4}}\left(AB^{3}\right)^{1/4}
\qquad
\forall (A,B)\in[0,+\infty)^{2},
\end{equation*}
we obtain exactly (\ref{th:usual-below}).
\end{proof}


The second lemma shows that for every function $u\in H^{2}((a,b))$ one can find a function $z\in\PJ((a,b))$ that is close to $u$ in terms of $L^{p}$ norm and total variation, and such that the $\RPM_{\ep}$ energy of $u$ is controlled from below by the $\J_{1/2}$ energy of $z$. An analogous result is proved in~\cite[Proposition~3.3]{ABG}.

\begin{lemma}[Substitution lemma]\label{lemma:split}

Let $(a,b)\subseteq\re$ be an interval, let $\ep_{n}\subseteq(0,1)$ be a sequence such that $\ep_{n}\to 0^{+}$, and let $\{u_{n}\}\subseteq H^{2}((a,b))$ be a sequence of functions such that
\begin{equation}
\sup\left\{\RPM_{\ep_{n}}((a,b),u_{n}):n\geq 1\right\}<+\infty.
\label{hp:RPM-bounded}
\end{equation}

Then there exists a sequence $\{z_{n}\}\subseteq\PJ((a,b))$ with the following properties.
\begin{enumerate}
\renewcommand{\labelenumi}{(\arabic{enumi})}

\item  It turns out that
\begin{equation}
\RPM_{\ep_{n}}((a,b),u_{n})\geq
M_{n}\cdot\J_{1/2}((a,b),z_{n})
\qquad
\forall n\geq 1,
\label{th:uz-energy}
\end{equation}
where
\begin{equation*}
M_{n}:=4\left(\frac{2}{3}\right)^{1/2}
\left\{\frac{1}{|\log\ep_{n}|}\log\left(1+\frac{1}{\ep_{n}^{4}|\log\ep_{n}|^{8}}\right)\right\}^{3/4}
\qquad
\forall n\geq 1.
\end{equation*}

\item  The function $z_{n}$ is asymptotically close to $u_{n}$ in the sense that
\begin{equation}
\lim_{n\to +\infty}\|u_{n}-z_{n}\|_{L^{p}((a,b))}= 0
\qquad
\forall p\in[1,+\infty).
\label{th:uz-Lp}
\end{equation}

\item  The total variation of $z_{n}$ is asymptotically close to the total variation of $u_{n}$ in the sense that
\begin{equation}
\lim_{n\to +\infty}\int_{a}^{b}|u_{n}'(x)|\,dx-|Dz_{n}|((a,b))=0.
\label{th:uz-TV}
\end{equation}

\end{enumerate}

\end{lemma}

\begin{proof}

Let us consider the set
\begin{equation*}
A_{n}:=\left\{x\in(a,b):|u_{n}'(x)|>\frac{1}{\ep_{n}^{2}|\log\ep_{n}|^{4}}\right\}.
\end{equation*}

Since $A_{n}$ is an open set, we can write it as a finite or countable union of open disjoint intervals (its connected components), namely in the form
\begin{equation*}
A_{n}=\bigcup_{i\in I_{n}}(\alpha_{n,i},\beta_{n,i}),
\end{equation*}
where $I_{n}$ is a suitable index set.

Let $w_{n}:[a,b]\to\re$ be the function of class $C^{1}$ such that $w_{n}(a)=u_{n}(a)$, and
\begin{equation*}
w_{n}'(x):=\begin{cases}
0  & \text{if }x\in(a,b)\setminus A_{n}, \\
u_{n}'(x)-\dfrac{\sign(u_{n}'(x))}{\ep_{n}^{2}|\log\ep_{n}|^{4}}\quad  & \text{if }x\in A_{n}.
\end{cases}
\end{equation*}

We observe that $w_{n}'(x)$ is the difference between $u_{n}'(x)$ and the truncation of $u_{n}'(x)$ between the two values $\pm\ep_{n}^{-2}|\log\ep_{n}|^{-4}$. We deduce that in each of the intervals $
(\alpha_{n,i},\beta_{n,i})$ the sign of $w_{n}'(x)$ is constant and coincides with the sign of $u_{n}'(x)$ in the same interval, and in any case it turns out that
\begin{equation}
|w_{n}(\beta_{n,i})-w_{n}(\alpha_{n,i})|=
|u_{n}(\beta_{n,i})-u_{n}(\alpha_{n,i})|-
\frac{\beta_{n,i}-\alpha_{n,i}}{\ep_{n}^{2}|\log\ep_{n}|^{4}}.
\label{eqn:TV-wn-un}
\end{equation}

Finally, for every $i\in I_{n}$ we consider the midpoint $\gamma_{n,i}:=(\alpha_{n,i}+\beta_{n,i})/2$ of the interval $(\alpha_{n,i},\beta_{n,i})$, and we introduce the function $z_{n}\in\PJ((a,b))$ whose jump points are located in those midpoints, with height equal to the variation of $w_{n}$ in the corresponding intervals, and translated vertically so that $z_{n}(a)=w_{n}(a)=u_{n}(a)$. This function is given by
\begin{equation*}
z_{n}(x):=u_{n}(a)+
\sum_{i\in I_{n}}(w_{n}(\beta_{n,i})-w_{n}(\alpha_{n,i}))\mathbbm{1}_{(\gamma_{n,i},b)}(x)
\qquad
\forall x\in(a,b).
\end{equation*} 

With these definitions, we are now ready to prove the required estimates.

\paragraph{\textmd{\textit{Statement~(1)}}}

Let us apply Lemma~\ref{lemma:basic-below} to the function $u_{n}$ in the interval $(\alpha_{n,i},\beta_{n,i})$ with $D:=\ep_{n}^{-2}|\log\ep_{n}|^{-4}$. Recalling (\ref{eqn:TV-wn-un}), we obtain that
\begin{eqnarray*}
\RPM_{\ep_{n}}((a,b),u_{n}) & \geq &
\RPM_{\ep_{n}}(A_{n},u_{n})
\\[1ex]
& = & 
\sum_{i\in I_{n}}\RPM_{\ep_{n}}((\alpha_{n,i},\beta_{n,i}),u_{n})
\\
& \geq & M_{n}\sum_{i\in I_{n}}\left(|u_{n}(\beta_{n,i})-u_{n}(\alpha_{n,i})|-
\frac{\beta_{n,i}-\alpha_{n,i}}{\ep_{n}^{2}|\log\ep_{n}|^{4}}\right)^{1/2}
\\[0.5ex]
& = &
M_{n}\sum_{i\in I_{n}}\left|w_{n}(\beta_{n,i})-w_{n}(\alpha_{n,i})\right|^{1/2}
\\[0.5ex]
& = &
M_{n}\sum_{i\in I_{n}}\left|J_{z_{n}}(\gamma_{n,i})\right|^{1/2}
\\[0.5ex]
& = &
M_{n}\J_{1/2}((a,b),z_{n}),
\end{eqnarray*}
which proves (\ref{th:uz-energy}). 

\paragraph{\textmd{\textit{Statement~(2)}}}

In order to prove (\ref{th:uz-Lp}), we show that
\begin{equation}
u_{n}(x)-w_{n}(x)\to 0
\qquad
\text{uniformly in }[a,b]
\label{th:un-wn}
\end{equation}
and that for every $p\in[1,+\infty)$ it turns out that
\begin{equation}
w_{n}-z_{n}\to 0
\qquad
\text{in }L^{p}((a,b)).
\label{th:wn-zn}
\end{equation}

In order to prove (\ref{th:un-wn}) we introduce the sets
\begin{equation*}
B_{n}:=\left\{x\in(a,b):\frac{1}{|\log\ep_{n}|}\leq|u_{n}'(x)|\leq\frac{1}{\ep_{n}^{2}|\log\ep_{n}|^{4}}\right\}
\end{equation*}
and
\begin{equation*}
C_{n}:=\left\{x\in(a,b):|u_{n}'(x)|<\frac{1}{|\log\ep_{n}|}\right\}.
\end{equation*}

Let us estimate the measure of $A_{n}$, $B_{n}$, $C_{n}$. In the case of $C_{n}$ we consider the trivial estimate $|C_{n}|\leq b-a$. In the case of $A_{n}$ and $B_{n}$ we consider the term with the logarithm in (\ref{defn:ABG}), and we obtain that
\begin{equation*}
|A_{n}|\leq
\frac{\ep_{n}^{2}|\log\ep_{n}|}{\log(1+\ep_{n}^{-4}|\log\ep_{n}|^{-8})}\RPM_{\ep_{n}}((a,b),u_{n})
\end{equation*}
and
\begin{equation*}
|B_{n}|\leq\frac{\ep_{n}^{2}|\log\ep_{n}|}{\log(1+|\log\ep_{n}|^{-2})}\RPM_{\ep_{n}}((a,b),u_{n}).
\end{equation*}

Recalling (\ref{hp:RPM-bounded}), these estimates imply that
\begin{equation*}
\lim_{n\to +\infty}\frac{|A_{n}|}{\ep_{n}^{2}|\log\ep_{n}|^{4}}=
\lim_{n\to +\infty}\frac{|B_{n}|}{\ep_{n}^{2}|\log\ep_{n}|^{4}}=
0.
\end{equation*}

Now let us consider the function $u_{n}'(x)-w_{n}'(x)$. In $A_{n}$ it turns out that
\begin{equation*}
\int_{A_{n}}|u_{n}'(x)-w_{n}'(x)|\,dx=
\int_{A_{n}}\frac{1}{\ep_{n}^{2}|\log\ep_{n}|^{4}}\,dx=
\frac{|A_{n}|}{\ep_{n}^{2}|\log\ep_{n}|^{4}}.
\end{equation*}

In $B_{n}$ and $C_{n}$ it turns out that $w_{n}'(x)=0$, and hence
\begin{equation*}
\int_{B_{n}}|u_{n}'(x)-w_{n}'(x)|\,dx=
\int_{B_{n}}|u_{n}'(x)|\,dx\leq
\frac{|B_{n}|}{\ep_{n}^{2}|\log\ep_{n}|^{4}}
\end{equation*}
and
\begin{equation*}
\int_{C_{n}}|u_{n}'(x)-w_{n}'(x)|\,dx=
\int_{C_{n}}|u_{n}'(x)|\,dx\leq
\frac{|C_{n}|}{|\log\ep_{n}|}\leq
\frac{b-a}{|\log\ep_{n}|}.
\end{equation*}

From all these estimates we conclude that
\begin{equation}
\lim_{n\to +\infty}\int_{a}^{b}|u_{n}'(x)-w_{n}'(x)|\,dx=0,
\label{est:un'-wn'}
\end{equation}
which implies (\ref{th:un-wn}) because $u_{n}(a)=w_{n}(a)=0$ for every $n\geq 1$.

In order to prove (\ref{th:wn-zn}), we begin by observing that for every $x\in(a,b)\setminus A_{n}$ it turns out that
\begin{eqnarray*}
w_{n}(x) & = & 
w_{n}(a)+\int_{a}^{x}w_{n}'(t)\,dt
\\
& = &
u_{n}(a)+\sum_{\{i\in I_{n}:\beta_{n,i}\leq x\}}
\int_{\alpha_{n,i}}^{\beta_{n,i}}w_{n}'(t)\,dt
\\
& = &
u_{n}(a)+\sum_{\{i\in I_{n}:\beta_{n,i}\leq x\}}(w_{n}(\beta_{n,i})-w_{n}(\alpha_{n,i}))
\\
& = & z_{n}(x),
\end{eqnarray*}
which implies that $w_{n}(x)-z_{n}(x)=0$ when $x\not\in A_{n}$. On the other hand, when $x\in A_{n}$ it turns out that
\begin{equation*}
|w_{n}(x)-z_{n}(x)|\leq
|w_{n}(\beta_{n,i})-w_{n}(\alpha_{n,i})|=
|J_{z_{n}}(\gamma_{n,i})|\leq
\left(\J_{1/2}((a,b),z_{n})\right)^{2},
\end{equation*}
and therefore from (\ref{th:uz-energy}) we conclude that
\begin{eqnarray*}
\int_{a}^{b}|w_{n}(x)-z_{n}(x)|^{p}\,dx & = &
\sum_{i\in I_{n}}\int_{\alpha_{i,\ep_{n}}}^{\beta_{i,\ep_{n}}}|w_{n}(x)-z_{n}(x)|^{p}\,dx
\\[0.5ex]
& \leq &
\sum_{i\in I_{n}}(\beta_{i,\ep_{n}}-\alpha_{i,\ep_{n}})\cdot\left(\J_{1/2}((a,b),z_{n})\right)^{2p}
\\[0.5ex]
& = &
|A_{n}|\cdot\left(\J_{1/2}((a,b),z_{n})\right)^{2p}
\\[0.5ex]
& \leq & |A_{n}|\cdot\left(\frac{1}{M_{n}}\RPM_{\ep_{n}}((a,b),u_{n})\right)^{2p},
\end{eqnarray*}
which implies (\ref{th:wn-zn}) because $\RPM_{\ep_{n}}((a,b),u_{n})$ is bounded from above, $M_{n}$ is bounded from below, and $|A_{n}|\to 0$.

\paragraph{\textmd{\textit{Statement~(3)}}}

It remains to prove (\ref{th:uz-TV}). To this end, we just observe that
\begin{equation*}
|Dz_{n}|((a,b))=
\sum_{i\in I_{n}}|w_{n}(\beta_{n,i})-w_{n}(\alpha_{n,i})|=
\int_{A_{n}}|w_{n}'(x)|\,dx=
\int_{a}^{b}|w_{n}'(x)|\,dx,
\end{equation*}
and
\begin{equation*}
\left|
\int_{a}^{b}|u_{n}'(x)|\,dx-
\int_{a}^{b}|w_{n}'(x)|\,dx
\right|\leq
\int_{a}^{b}|u_{n}'(x)-w_{n}'(x)|\,dx,
\end{equation*}
and we conclude thanks to (\ref{est:un'-wn'}).
\end{proof}


The last preliminary result is the classical lower semicontinuity of $\J_{1/2}$ (see for example~\cite[Theorems~4.7 and~4.8]{AFP}). Here we include an elementary proof in the one dimensional case, different from the original proof in~\cite{1989-BUMI-Ambrosio}, because we need to deduce the reinforced statement (true in dimension one) where the convergence of the energies implies the \emph{strict} convergence of the arguments.

\begin{lemma}[Lower semicontinuity of $\J_{1/2}$]\label{lemma:J-strict}

Let $(a,b)\subseteq\re$ be an interval, and let $\{z_{n}\}\subseteq\PJ((a,a))$ be a sequence with the following properties:
\begin{enumerate}
\renewcommand{\labelenumi}{(\roman{enumi})}

\item  there exists a constant $M$ such that
\begin{equation}
\J_{1/2}((a,b),z_{n})\leq M
\qquad
\forall n\geq 1,
\label{hp:bound-J1/2}
\end{equation}

\item  there exists $p\geq 1$ and $z_{\infty}\in L^{p}((a,b))$ such that $z_{n}\to z_{\infty}$ in $L^{p}((a,b))$.

\end{enumerate}

Then the following two statements hold true.
\begin{enumerate}
\renewcommand{\labelenumi}{(\arabic{enumi})}

\item \emph{(Lower semicontinuity).} It turns out that $z_{\infty}\in\PJ((a,b))$ and
\begin{equation}
\liminf_{n\to +\infty}\J_{1/2}((a,b),z_{n})\geq
\J_{1/2}((a,b),z_{\infty}).
\label{th:J-LSC}
\end{equation}

\item  \emph{(Strict convergence).} If in addition we assume that
\begin{equation}
\lim_{n\to +\infty}\J_{1/2}((a,b),z_{n})=\J_{1/2}((a,b),z_{\infty}),
\label{hp:zn-strict}
\end{equation}
then actually $z_{n}\auto z_{\infty}$ in $BV((a,b))$.

\end{enumerate}

\end{lemma}

\begin{proof}

For every $n\geq 1$, let us write $z_{n}(x)$ in the form
\begin{equation}
z_{n}(x)=c_{n}+\sum_{i=1}^{\infty}J_{n}(i)\mathbbm{1}_{(s_{n}(i),+\infty)}(x)
\qquad
\forall x\in(a,b),
\label{eqn:zn}
\end{equation}
where $c_{n}:=z_{n}(a)$ (the boundary value is intended in the usual sense), $\{s_{n}(i)\}_{i\geq 1}\subseteq(a,b)$ is a sequence of distinct points, and $\{J_{n}(i)\}_{i\geq 1}$ is a sequence of real numbers such that $|J_{n}(i+1)|\leq|J_{n}(i)|$ for every $i\geq 1$. We observe that, even when the function $z_{n}$ has only a finite number of jump points, we can always write it in the form (\ref{eqn:zn}) by introducing infinitely many ``fake jumps'' with ``jump height'' equal to~0.

From assumptions (i) and (ii) we derive two types of estimates.
\begin{itemize}

\item  (Uniform bounds). From assumption (i) and the subadditivity of the square root we deduce that
\begin{equation}
\sum_{i=1}^{\infty}|J_{n}(i)|\leq 
\left(\sum_{i=1}^{\infty}|J_{n}(i)|^{1/2}\right)^{2}\leq 
M^{2}
\qquad
\forall n\geq 1.
\label{est:TV}
\end{equation}

Combined with assumption (ii), this implies that there exists a constant $M_{1}$ such that
\begin{equation}
|c_{n}|\leq M_{1}
\qquad
\forall n\geq 1.
\label{est:cn}
\end{equation}

Finally, from (\ref{est:TV}) and (\ref{est:cn}) we deduce that there exists a constant $M_{2}$ such that
\begin{equation}
\|z_{n}\|_{L^{\infty}((a,b))}\leq M_{2}
\qquad
\forall n\geq 1.
\label{est:zn-Linfty}
\end{equation}

\item  (Uniform smallness of the tails). We claim that for every $\ep>0$ there exists a positive integer $i_{\ep}$ such that
\begin{equation}
\sum_{i=i_{\ep}}^{\infty}|J_{n}(i)|\leq M\sqrt{\ep}
\qquad
\forall n\geq 1.
\label{est:tails}
\end{equation}

Indeed, if we define $i_{\ep}$ as the smallest integer greater that $M/\sqrt{\ep}$, then from (\ref{hp:bound-J1/2}) it turns out that $|J_{n}(i)|\leq\ep$ for at least one index $i\leq i_{\ep}$. At this point, from the monotonicity of $|J_{n}(i)|$ we conclude that
\begin{equation*}
|J_{n}(i)|\leq\ep
\qquad
\forall n\geq 1,
\quad
\forall i\geq i_{\ep},
\end{equation*}
and therefore
\begin{equation*}
\sum_{i=i_{\ep}}^{\infty}|J_{n}(i)|=
\sum_{i=i_{\ep}}^{\infty}|J_{n}(i)|^{1/2}\cdot|J_{n}(i)|^{1/2}\leq
\sqrt{\ep}\cdot\sum_{i=i_{\ep}}^{\infty}|J_{n}(i)|^{1/2}\leq
\sqrt{\ep}\cdot M,
\end{equation*}
which proves (\ref{est:tails}).

\end{itemize}

From the uniform bounds we obtain that, up to subsequences (not relabeled), the following limits exists as $n\to +\infty$:
\begin{equation*}
c_{n}\to c_{\infty},
\qquad\quad
J_{n}(i)\to J_{\infty}(i),
\qquad\quad
s_{n}(i)\to s_{\infty}(i)\in[a,b].
\end{equation*}

From these limits we deduce that
\begin{equation}
\liminf_{n\to +\infty}\J_{1/2}((a,b),z_{n})=
\liminf_{n\to +\infty}\sum_{i=1}^{\infty}|J_{n}(i)|^{1/2}\geq
\sum_{i=1}^{\infty}|J_{\infty}(i)|^{1/2},
\label{ineq:J-zinfty}
\end{equation}
and, due to the uniform smallness of the tails,
\begin{equation}
\lim_{n\to +\infty}|Dz_{n}|((a,b))=
\lim_{n\to +\infty}\sum_{i=1}^{\infty}|J_{n}(i)|=
\sum_{i=1}^{\infty}|J_{\infty}(i)|.
\label{ineq:TV-zinfty}
\end{equation}

At this point we can introduce the function 
\begin{equation}
\widehat{z}_{\infty}(x):=c_{\infty}+
\sum_{i=1}^{\infty}J_{\infty}(i)\mathbbm{1}_{(s_{\infty}(i),+\infty)}(x)
\qquad
\forall x\in(a,b).
\label{defn:z-infty}
\end{equation}

Exploiting again (\ref{est:tails}) we can show that $z_{n}(x)\to \widehat{z}_{\infty}(x)$ for every $x\in(a,b)$ that does not appear in the sequence $\{s_{\infty}(i)\}$. This almost everywhere pointwise convergence, together with the uniform bound (\ref{est:zn-Linfty}), implies that $z_{n}\to\widehat{z}_{\infty}$ in $L^{p}((a,b))$ for every $p\in[1,+\infty)$, and hence in particular that $z_{\infty}(x)=\widehat{z}_{\infty}(x)$.

Now we exploit the representation (\ref{defn:z-infty}) in order to compute the total variation of $z_{\infty}$ and $\J_{1/2}((a,b),z_{\infty})$. This is not immediate, because in the representation (\ref{defn:z-infty}) the points $s_{\infty}(i)$ are not necessarily distinct, and some of them might even coincide with the endpoints of the interval $(a,b)$, in which case they do not contribute to the total variation or to $\J_{1/2}$. In any case, the function defined by (\ref{defn:z-infty}) belongs to $\PJ((a,b))$, and its jump set is contained in the image of the sequence $\{s_{\infty}(i)\}$ intersected with the open interval $(a,b)$. Moreover, for every $s$ in this set, the jump height of $z_{\infty}$ in $s$ is given by
\begin{equation*}
J_{z_{\infty}}(s)=\sum_{\{i\geq 1:s_{\infty}(i)=s\}}J_{\infty}(i),
\end{equation*}
where of course the sum (or series) might also vanish. In particular, for every jump point $s$ of $z_{\infty}$ we obtain that
\begin{equation}
|J_{z_{\infty}}(s)|\leq\sum_{\{i\geq 1:s_{\infty}(i)=s\}}|J_{\infty}(i)|,
\label{est:J}
\end{equation}
with equality if and only if all terms in the sum have the same sign. Analogously, we obtain that
\begin{equation*}
|J_{z_{\infty}}(s)|^{1/2}\leq\sum_{\{i\geq 1:s_{\infty}(i)=s\}}|J_{\infty}(i)|^{1/2},
\end{equation*}
with equality if and only if at most one term in the sum is different from~0 (here we exploit that the square root is strictly subadditive).

From (\ref{est:J}) it follows that
\begin{equation}
|Dz_{\infty}|((a,b))=
\sum_{s\in S_{z_{\infty}}}|J_{z_{\infty}}(s)|\leq
\sum_{i=1}^{\infty}|J_{\infty}(i)|,
\label{eqn:TV-zinfty}
\end{equation}
with equality if and only if $s_{\infty}(i)\in(a,b)$ for every $i\geq 1$ such that $J_{\infty}(i)\neq 0$, and $J_{\infty}(i)\cdot J_{\infty}(j)\geq 0$ for every pair $(i,j)$ of distinct positive integers such that $s_{\infty}(i)=s_{\infty}(j)\in(a,b)$.

Analogously, it turns out that 
\begin{equation}
\J_{1/2}((a,b),z_{\infty})=
\sum_{s\in S_{z_{\infty}}}|J_{z_{\infty}}(s)|^{1/2}\leq
\sum_{i=1}^{\infty}|J_{\infty}(i)|^{1/2},
\label{eqn:J-zinfty}
\end{equation}
with equality if and only if $s_{\infty}(i)\in(a,b)$ for every $i\geq 1$ such that $J_{\infty}(i)\neq 0$, and $J_{\infty}(i)\cdot J_{\infty}(j)= 0$ for every pair $(i,j)$ of distinct positive integers such that $s_{\infty}(i)=s_{\infty}(j)\in(a,b)$. In particular, in all cases where equality occurs in (\ref{eqn:J-zinfty}), then equality occurs also in (\ref{eqn:TV-zinfty}). 

At this point we are ready to complete the proof. Indeed, (\ref{th:J-LSC}) follows from (\ref{ineq:J-zinfty}) and (\ref{eqn:J-zinfty}), provided that we start with the subsequence of $\{z_{n}\}$ that realizes the liminf in (\ref{th:J-LSC}). As for the strict convergence, under assumption (\ref{hp:zn-strict}) we have necessarily equality both in (\ref{ineq:J-zinfty}) and in (\ref{eqn:J-zinfty}), and hence we have equality also in (\ref{eqn:TV-zinfty}). At this point from (\ref{ineq:TV-zinfty}) and (\ref{eqn:TV-zinfty}) we conclude that $|Dz_{n}|((a,b))\to|Dz_{\infty}|((a,b))$ (to be over-pedantic, what we actually proved is that every subsequence of $\{z_{n}\}$ has a further subsequence with this property), which is what we need in order to conclude that the convergence is strict.
\end{proof}

\begin{rmk}
\begin{em}

The only properties of the square root that are relevant for Lemma~\ref{lemma:J-strict} are that it is a nonnegative function that is strictly sub-additive and satisfies $\sqrt{\sigma}/\sigma\to +\infty$ as $\sigma\to 0^{+}$.

\end{em}
\end{rmk}


\subsection{Proof of Theorem~\ref{thm:ABG}}

\paragraph{\textmd{\textit{Statement~(1)}}}

Let us start with the liminf inequality. We need to prove that
\begin{equation}
\liminf_{n\to +\infty}\RPM_{\ep_{n}}((a,b),u_{n})\geq
\alpha_{0}\J_{1/2}((a,b),u)
\label{Gconv:liminf}
\end{equation}
for every sequence $\{u_{n}\}\subseteq H^{2}((a,b))$ such that $u_{n}\to u$ in $L^{2}((a,b))$, and every sequence $\{\ep_{n}\}\subseteq(0,1)$ such that $\ep_{n}\to 0^{+}$. Up to subsequences (not relabeled), we can assume that the left-hand side is bounded and that the liminf is actually a limit, and in particular that the sequence $\{\RPM_{\ep_{n}}((a,b),u_{n})\}$ is bounded. When this is the case, from Lemma~\ref{lemma:split} we obtain a sequence $\{z_{n}\}\subseteq\PJ((a,b))$ such that $z_{n}\to u$ in $L^{2}((a,b))$ and
\begin{equation}
\RPM_{\ep_{n}}((a,b),u_{n})\geq
M_{n}\cdot\J_{1/2}((a,b),z_{n})
\qquad
\forall n\geq 1.
\label{th:Jzn}
\end{equation}

Now we observe that $M_{n}\to\alpha_{0}$ as $n\to +\infty$, and from the lower semicontinuity of $\J_{1/2}$ (with respect to any $L^{p}$ convergence) we conclude that 
\begin{equation*}
\liminf_{n\to +\infty}\RPM_{\ep_{n}}((a,b),u_{n})\geq
\liminf_{n\to +\infty}M_{n}\cdot\J_{1/2}((a,b),z_{n})\geq
\alpha_{0}\J_{1/2}((a,b),u),
\end{equation*}
which proves (\ref{Gconv:liminf}).

For the limsup inequality, we refer to the proof of~\cite[Theorem~4.4]{2008-TAMS-BF}. The idea is rather classical. First of all, we reduce ourselves  to the case where $u$ has only a finite number of jump points, because this class is dense in $L^{2}((a,b))$ with respect to the energy $\J_{1/2}$. Given any function $u\in\PJ((a,b))$ with a finite number of jumps, we consider the function $u_{\ep}(x)$ that coincides with $u(x)$ outside some small intervals that contain a single jump point, and in each of these small intervals coincides with the cubic polynomial that interpolates the values at the boundary of the interval. From Lemma~\ref{lemma:ABG} we obtain the exact value of the integral of $u_{\ep}''(x)^{2}$, and an estimate from above for the integral of $\log(1+u_{\ep}'(x)^{2})$. If we optimize the length of each small interval in terms of $\ep$ and of the jump height, the resulting family is the required recovery family. 

We stress that, in the case where $u$ has a finite number of jumps, there exists a recovery family that coincides with $u$ in a fixed neighborhood of the boundary points $x=a$ and $x=b$.

\paragraph{\textmd{\textit{Statement~(2)}}}

Let us apply again Lemma~\ref{lemma:split}. We obtain a sequence $\{z_{n}\}\subseteq\PJ((a,b))$ satisfying (\ref{th:uz-Lp}) and (\ref{th:Jzn}). In particular, since $M_{n}$ is bounded from below by a positive constant, from (\ref{hp:Gconv-coercive}) we deduce that this sequence satisfies
\begin{equation*}
\sup_{n\in\n}\left\{\J_{1/2}((a,b),z_{n})+\int_{a}^{b}z_{n}(x)^{2}\,dx\right\}<+\infty.
\end{equation*}

From the classical compactness result for the functional $\J_{1/2}$ (whose proof in dimension one is more or less contained in the proof of Lemma~\ref{lemma:J-strict} above), it follows that $\{z_{n}\}$ is relatively compact in $L^{p}((a,b))$ for every $p\in[1,+\infty)$. Due to  (\ref{th:uz-Lp}), the same is true for $\{u_{n}\}$.

\paragraph{\textmd{\textit{Statement~(3)}}}

Let us apply again Lemma~\ref{lemma:split}. The resulting sequence $\{z_{n}\}$ converges to $u$. On the other hand, from the lower semicontinuity of $\J_{1/2}$, estimate (\ref{th:uz-energy}), and assumption (\ref{hp:recovery}), we deduce that
\begin{eqnarray*}
\J_{1/2}((a,b),u) & \leq &
\liminf_{n\to +\infty}\J_{1/2}((a,b),z_{n})
\\[0.5ex]
& \leq &
\limsup_{n\to +\infty}\J_{1/2}((a,b),z_{n})
\\
& \leq &
\limsup_{n\to +\infty}\frac{1}{M_{n}}\cdot\RPM_{\ep_{n}}((a,b),u_{n})
\\
& = &
\frac{1}{\alpha_{0}}\cdot\alpha_{0}\J_{1/2}((a,b),u).
\end{eqnarray*}

This implies that $\J_{1/2}((a,b),z_{n})\to\J_{1/2}((a,b),u)$, and therefore from Lemma~\ref{lemma:J-strict} we conclude that $z_{n}\auto u$ in $BV((a,b))$. In turn, this implies that also $u_{n}\auto u$ in $BV((a,b))$ because of (\ref{th:uz-TV}).

\paragraph{\textmd{\textit{Statement~(4)}}}

As in the proof of the limsup inequality for the Gamma-convergence result, we can assume that $u$ is a pure jump function with a finite number of jump points. When this is the case, we already know that there exists a recovery sequence $\widehat{u}_{n}\to u$ that coincides with $u$ in a neighborhood of the boundary, namely there exists $\eta>0$ such that for every $n\geq 1$ it turns out that $\widehat{u}_{n}(x)=u(x)=u(a)$ for every $x\in(a,a+\eta)$, and similarly $\widehat{u}_{n}(x)=u(x)=u(b)$ for every $x\in(b-\eta,b)$.

Now the idea is to modify $\widehat{u}_{n}$ in the two lateral intervals $(a,a+\eta)$ and $(b-\eta,b)$ in order to fulfill the given boundary conditions (\ref{hp:recovery-BC}). To this end, we set
\begin{equation*}
u_{n}(x):=\begin{cases}
u(a)+w_{1,n}(x)\quad      & \text{if }x\in(a,a+\eta], \\
\widehat{u}_{n}(x)     & \text{if }x\in[a+\eta,b-\eta],  \\
u(b)+w_{2,n}(x) & \text{if }x\in[b-\eta,b).
\end{cases}
\end{equation*}
where $w_{1,n}$ is the function given by Lemma~\ref{lemma:bound-D-H} applied in the interval  $(a,a+\eta)$ with boundary data
\begin{equation*}
\left(w_{1,n}(a),w_{1,n}'(a),w_{1,n}(a+\eta),w_{1,n}'(a+\eta)\right)=
(A_{0,n}-u(a),A_{1,n},0,0),
\end{equation*}
and $w_{2,n}$ is the function given by Lemma~\ref{lemma:bound-D-H} applied in the interval  $(b-\eta,b)$ with boundary data
\begin{equation*}
\left(w_{2,n}(a),w_{2,n}'(a),w_{2,n}(a+\eta),w_{2,n}'(a+\eta)\right)=
(0,0,B_{0,n}-u(b),B_{1,n}).
\end{equation*}

We observe that $u_{n}\in H^{2}((a,b))$ and
\begin{eqnarray*}
\RPM_{\ep_{n}}((a,b),u_{n}) & = &
\RPM_{\ep_{n}}((a,a+\eta),w_{1,n})
\\
& &
+\RPM_{\ep_{n}}((a+\eta,b-\eta),\widehat{u}_{n})
\\
& &
+\RPM_{\ep_{n}}((b-\eta,b),w_{2,n}).
\end{eqnarray*}

The second term coincides with $\RPM_{\ep_{n}}((a,b),\widehat{u}_{n})$, and therefore it converges to $\alpha_{0}\J_{1/2}((a,b),u)$ when $n\to +\infty$. Therefore, it is enough to show that the other two terms vanish in the limit. To this end, we observe that in the interval $(a,a+\eta)$ the assumptions of Lemma~\ref{lemma:bound-D-H} are satisfied with
\begin{equation*}
H=H_{n}:=|A_{0,n}-u(a)|
\qquad\text{and}\qquad
D=D_{n}:=|A_{1,n}|.
\end{equation*}

Since $H_{n}$ and $D_{n}$ tend to~0, we conclude that
\begin{equation*}
\lim_{n\to+\infty}\RPM_{\ep_{n}}((a,a+\eta),w_{1,n})\leq
\lim_{n\to+\infty}80\left(\sqrt{H_{n}}+\ep_{n}^{2}D_{n}\right)=
0.
\end{equation*}

In the same way we obtain that
\begin{equation*}
\lim_{n\to+\infty}\RPM_{\ep_{n}}((b-\eta,b),w_{2,n})=
0,
\end{equation*}
which completes the proof.
\qed


\subsection{Proof of Proposition~\ref{prop:mu}}

\paragraph{\textmd{\textit{Statement~(\ref{prop:existence})}}}

In the case of $\mu_{\ep}$, $\mu_{\ep}^{*}$ and $\mu_{0}$, existence is a standard application of the direct method in the calculus of variations. The case of $\mu_{0}^{*}$ is less trivial because boundary conditions in $\PJ((0,L))$ do not pass to the limit, for example, with respect to $L^{2}$ convergence. This issue, however, can be fixed in a rather standard way. To this end, we relax boundary conditions by allowing ``jumps at the boundary'', namely we minimize
\begin{equation*}
\alpha\J_{1/2}((0,L),v)+\beta\int_{0}^{L}(v(x)-Mx)^{2}\,dx+
\alpha\left(|v(0)|^{1/2}+|v(L)-ML|^{1/2}\right)
\end{equation*}
over $\PJ((0,L))$, without boundary conditions. In this case the direct method works, and we claim that any minimizer $v(x)$ satisfies $v(0)=0$ and $v(L)=ML$. Indeed, let $v(x)$ be any minimizer, and let us consider the value in $x=0$ (the argument in $x=L$ is symmetric). Let us assume that $M>0$ (the case $M=0$ is trivial, and the case $M<0$ is symmetric).  Arguing as at the beginning of section~\ref{subsec:loc-min-proof} we can show that the set of jump points of $v$ is finite, and comparing with a competitor $v_{\tau}(x)$ which is equal to~0 in $(0,\tau)$, and equal to $v(x)$ elsewhere, we can conclude that $v(0)=0$. 

\paragraph{\textmd{\textit{Statement~(\ref{prop:M})}}}

We prove the result in the case of $\mu_{\ep}$, but the argument is analogous in the other three cases.

The symmetry follows from the simple remark that, if $v(x)$ is a minimizer for some $M$, then $-v(x)$ is a minimizer for $-M$.

As for the continuity, it follows from the fact that, if $M_{n}\to M_{\infty}$, then the fidelity term in 
$\RPMF_{\ep}(\beta,M_{n}x,(0,L),v)$ converges to the fidelity term in $\RPMF_{\ep}(\beta,M_{\infty}x,(0,L),v)$ uniformly on bounded subsets of $L^{2}((a,b))$.

As for monotonicity, let us consider any pair $0\leq M_{1}< M_{2}$. Let us choose any minimizer $v_{2}\in H^{2}((0,L))$ in the definition of $\mu_{\ep}(\beta,L,M_{2})$, and let us consider the function $v_{1}(x):=(M_{1}/M_{2})v_{2}(x)$. Elementary computations show that
\begin{equation*}
\RPM_{\ep}((0,L),v_{1})\leq\RPM_{\ep}((0,L),v_{2}),
\end{equation*}
and
\begin{equation*}
\int_{0}^{L}(v_{1}(x)-M_{1}\,x)^{2}\,dx=
\frac{M_{1}^{2}}{M_{2}^{2}}\int_{0}^{L}(v_{2}(x)-M_{2}\,x)^{2}\,dx\leq
\int_{0}^{L}(v_{2}(x)-M_{2}\,x)^{2}\,dx,
\end{equation*}
and therefore
\begin{eqnarray*}
\mu_{\ep}(\beta,L,M_{1}) & \leq &
\RPMF_{\ep}(\beta,M_{1}x,(0,L),v_{1})
\\
& \leq &
\RPMF_{\ep}(\beta,M_{2}x,(0,L),v_{2})
\\
& = &
\mu_{\ep}(\beta,L,M_{2}).
\end{eqnarray*}

\paragraph{\textmd{\textit{Statement~(\ref{prop:L})}}}

Let us consider any pair $0<L_{1}<L_{2}$, and let us examine separately the behavior of the four functions.

In the case of $\mu_{\ep}$, let $v_{2}(x)$ be any minimizer for $\mu_{\ep}(\beta,L_{2},M)$. Then the restriction of $v_{2}(x)$ to $(0,L_{1})$, which we call $v_{1}(x)$, is a competitor in the definition of $\mu_{\ep}(\beta,L_{1},M)$, and therefore as before we conclude that
\begin{eqnarray*}
\mu_{\ep}(\beta,L_{1},M) & \leq &
\RPMF_{\ep}(\beta,Mx,(0,L_{1}),v_{1})
\\
& \leq &
\RPMF_{\ep}(\beta,Mx,(0,L_{2}),v_{2})
\\
& = &
\mu_{\ep}(\beta,L_{2},M).
\end{eqnarray*}

The same argument works in the case of $\mu_{0}$.

In the case of $\mu_{0}^{*}$ we have to keep into account boundary conditions, and therefore we define
\begin{equation}
v_{1}(x)=\frac{L_{1}}{L_{2}}v_{2}\left(\frac{L_{2}x}{L_{1}}\right)
\qquad
\forall x\in(0,L_{1}),
\label{defn:v2->v1}
\end{equation}
and we observe that $\J_{1/2}((0,L_{1}),v_{1})\leq\J_{1/2}((0,L_{2}),v_{2})$ and
\begin{equation*}
\int_{0}^{L_{1}}(v_{1}(x)-Mx)^{2}\,dx=
\left(\frac{L_{1}}{L_{2}}\right)^{3}\int_{0}^{L_{2}}(v_{2}(x)-Mx)^{2}\,dx,
\end{equation*}
which again implies the conclusion.

Finally, the monotonicity of $\mu_{\ep}^{*}$ with respect to $L$ is in general false (the minimum diverges when $L\to 0^{+}$ due to the term with second order derivatives). In this case the natural definition (\ref{defn:v2->v1}), that preserves the boundary conditions (both on the function and on the derivative), reduces the fidelity term and the term with the logarithm, but increases the term with second order derivatives. What we do in this case is the opposite. We consider a minimizer $v_{1}(x)$ in the definition of $\mu_{\ep}^{*}(\beta,L_{1},M)$, and we define a function $v_{2}(x)$ in $(0,L_{2})$ in such a way that (\ref{defn:v2->v1}) holds true. With a simple variable change we see that
\begin{equation*}
\int_{0}^{L_{2}}v_{2}''(x)^{2}\,dx=
\frac{L_{1}}{L_{2}}\int_{0}^{L_{1}}v_{1}''(x)^{2}\,dx,
\end{equation*}
\begin{equation*}
\int_{0}^{L_{2}}\log\left(1+v_{2}'(x)^{2}\right)\,dx=
\frac{L_{2}}{L_{1}}\int_{0}^{L_{1}}\log\left(1+v_{1}'(x)^{2}\right)\,dx,
\end{equation*}
and
\begin{equation*}
\int_{0}^{L_{2}}(v_{2}(x)-Mx)^{2}\,dx=
\left(\frac{L_{2}}{L_{1}}\right)^{3}\int_{0}^{L_{1}}(v_{1}(x)-Mx)^{2}\,dx,
\end{equation*}
so that in particular
\begin{equation*}
\RPMF_{\ep}(\beta,(0,L_{2}),Mx,v_{2})\leq
\left(\frac{L_{2}}{L_{1}}\right)^{3}\RPMF_{\ep}(\beta,(0,L_{1}),Mx,v_{1}).
\end{equation*}

Since $v_{2}(x)$ is a competitor in the definition of $\mu_{\ep}^{*}(\beta,L_{2},M)$, this is enough to establish (\ref{th:monot-muep*}).

\paragraph{\textmd{\textit{Statement~(\ref{prop:pointwise})}}}

Pointwise convergence, namely convergence of minima, is a rather standard consequence of Gamma-convergence and equi-coerciveness. We point out that in the case of (\ref{defn:muep*}) and (\ref{defn:mu0*}) the functionals have to take the boundary conditions into account (the usual way is to set the functionals equal to $+\infty$ when the argument does not satisfy the boundary conditions), and in this case the limsup inequality in the Gamma-convergence result is slightly more delicate because it requires the control of boundary conditions for recovery sequences.

\paragraph{\textmd{\textit{Statement~(\ref{prop:uniform})}}}

The pointwise convergence (\ref{th:lim-mu}) is actually uniform with respect to $M$ (on bounded sets) because of the continuity and monotonicity with respect to $M$ of both $\mu_{\ep}$ and $\mu_{0}$. An analogous argument applies in the case of (\ref{th:lim-mu*}).
\qed


\subsubsection*{\centering Acknowledgments}

We would like to thank Iacopo Ripoli for working on a preliminary version of this project in his master thesis~\cite{ripoli:tesi}. 

The first author is a member of the \selectlanguage{italian} ``Gruppo Nazionale per l'Analisi Matematica, la Probabilità e le loro Applicazioni'' (GNAMPA) of the ``Istituto Nazionale di Alta Matematica'' (INdAM). 

\selectlanguage{english}



\label{NumeroPagine}

\end{document}